\documentclass{amsart}
\usepackage{mathrsfs, amsmath, graphics}
\usepackage{amssymb, comment}

\usepackage{graphicx}

\usepackage{epsfig}

\usepackage{amsfonts}

\usepackage{amssymb,amscd}
\usepackage{tikz}
\usepackage{verbatim}
\usepackage{enumerate}
\usepackage{booktabs}
\usepackage[all]{xy}
\usepackage{subfigure}

\usepackage{caption}
\usepackage{subfigure}
\usepackage{marginnote}

\newtheorem{thm}{Theorem}[section]
\newtheorem{lemma}[thm]{Lemma}

\newtheorem{prop}[thm]{Proposition}

\theoremstyle{definition}
\newtheorem{defn}[thm]{Definition}

\newtheorem{question}[thm]{Question}

\theoremstyle{remark}
\newtheorem{remark}[thm]{Remark}
\usepackage{enumerate}
\numberwithin{equation}{section}

\theoremstyle{notation}

\usepackage{floatrow}
\newfloatcommand{capbtabbox}{table}[][\FBwidth]

\usepackage{lineno}

\usepackage{graphicx}
\usepackage{mathrsfs}
\usepackage{epsfig}
\usepackage{amsmath}
\usepackage{amsfonts}
\usepackage{amssymb}
\usepackage{amssymb,amscd}
\usepackage{tikz}
\usepackage{verbatim}
\usepackage{enumerate}
\usepackage{booktabs}
\usepackage[all]{xy}

\theoremstyle{definition}

\theoremstyle{remark}

\numberwithin{equation}{section}

\newcommand{\hc}{\mathbf{H}^2_{\mathbb{C}}}

\title{Three-dimensional complex reflection groups  via Ford domains}

\author{Jiming Ma}
\address{School of Mathematical Sciences, Fudan University, Shanghai, 200433, P. R. China}
\email{majiming@fudan.edu.cn}



\date{Mar. 3, 2023}

\thanks{Jiming Ma was partially supported by  NSFC 12171092. \\
}

\begin{document}

%\pagewiselinenumbers

%\AddEverypageHook{    % on every page...
%    \begin{tikzpicture}[remember picture,overlay]
%        \node [rotate=60,scale=6,color=gray!40] at (current page.center)
%        {DRAFT VERSION};
%    \end{tikzpicture}
%}

\maketitle
%	Complexification of 3-dimensional hyperbolic tetrahedral Coxeter polytopes

\begin{abstract} We initiate the study of  deformations of groups in three-dimensional complex hyperbolic geometry.
	Let $$G=\left\langle \iota_1, \iota_2, \iota_3, \iota_4 \Bigg| \begin{array}{c}    \iota_1^2=  \iota_2^2 = \iota_3^2=\iota_4^2=id,\\ [3 pt]
(\iota_1 \iota_3)^{2}=(\iota_1 \iota_4)^{3}=(\iota_2 \iota_4)^{2}=id
	\end{array}\right\rangle$$
	be an  abstract  group.
	We  study representations $\rho: G \rightarrow \mathbf{PU}(3,1)$,  where $\rho( \iota_{i})=I_{i}$ is a complex reflection fixing a  complex hyperbolic plane   in ${\bf H}^{3}_{\mathbb C}$   for $1 \leq i \leq 4$,  with the additional condition that  $I_1I_2$ is parabolic. When we assume two pairs of hyper-parallel complex hyperbolic planes  have the same distance, then
	  the moduli space $\mathcal{M}$ is parameterized by $(h,t) \in  [1, \infty) \times [0, \pi]$   but  $t \leq \operatorname{arccos}(-\frac{3h^2+1}{4h^2})$. In particular, $t=0$  and $t=\operatorname{arccos}(-\frac{3h^2+1}{4h^2})$ degenerate to ${\bf H}^{3}_{\mathbb R}$-geometry and  ${\bf H}^{2}_{\mathbb C}$-geometry  respectively.
	Via Ford domain, we   show   $\rho_{(h,t)}$ is a discrete and faithful representation of the group $G$  into $\mathbf{PU}(2,1)$ when $(h,t)=(\sqrt{2},\operatorname{arccos}(-\frac{7}{8}))$. We also show the 3-manifold at infinity of the even subgroup of $\Gamma =\rho_{(\sqrt{2},\operatorname{arccos}(-\frac{7}{8}))}(G)<\mathbf{PU}(2,1)$ is  the connected sum of the trefoil knot  complement  in $\mathbb{S}^3$ and a three dimensional real projective space. 
	Using the Ford domain of $\rho_{(\sqrt{2},\operatorname{arccos}(-\frac{7}{8}))}(G)$  as a guide, we continue to show   $\rho_{(h,t)}$ is a discrete and faithful representation of $G \rightarrow \mathbf{PU}(3,1)$ when $(h,t) \in \mathcal{M}$ is near to $(\sqrt{2}, \operatorname{arccos}(-\frac{7}{8}))$. This is the first nontrivial  example of the Ford domain of a subgroup in  $\mathbf{PU}(3,1)$ that has been studied.

\end{abstract}

% \tableofcontents

\section{Introduction}\label{sec:intro}

%\subsection{Motivation}\label{subsection:motivation}
Hyperbolic $n$-space ${\bf H}^{n}_{\mathbb R}$ is the unique complete simply
connected Riemannian $n$-manifold with all sectional curvatures $-1$.  Complex hyperbolic $n$-space ${\bf H}^{n}_{\mathbb C}$ is the unique complete simply
connected K\"ahler  $n$-manifold with all holomorphic  sectional curvatures $-1$. But the Riemannian  sectional curvatures of complex hyperbolic space  are no longer constant, which are
pinched between $-1$ and $-\frac{1}{4}$, this makes complex hyperbolic  geometry more difficult to study.
The   holomorphic isometry group of ${\bf H}^{n}_{\mathbb C}$ is $\mathbf{PU}(n,1)$, the   orientation preserving isometry group of ${\bf H}^{n}_{\mathbb R}$ is $\mathbf{PO}(n,1)$, then  $\mathbf{PO}(n,1)$ is a subgroup of  $\mathbf{PU}(n,1)$.
%Discrete groups of complex
%hyperbolic isometries have not been studied as widely as their real
%hyperbolic counterparts. Nevertheless, they are interesting to study and
%should be more widely known.

%deformation
%theory of hyperbolic 3-manifolds and Kleinian groups is a rich topic in last sixty years, that is,  deformations of groups into $\mathbf{PO}(3,1)$

Over the last sixty years the theory of Kleinian groups, that  is,   deformations of groups into $\mathbf{PO}(3,1)$,  has flourished because of its close connections with low dimensional topology and geometry. More precisely, pioneered by
Ahlfors and Bers in the 1960's, and then
 Thurston formulated a conjectural classification scheme
for all hyperbolic 3-manifolds with finitely generated fundamental groups in the late 1970's.
The conjecture predicted that an infinite volume hyperbolic 3-manifold with finitely
generated fundamental group is uniquely determined by its topological type and its
end invariants. Thurston's conjecture is completed by a series of works of many mathematicians,   which is one of the most great breakthrough  in 3-manifold theory.  See Minsky's ICM talk 	\cite{Minsky:2006} for related topics and the reference.

%For simplicity, we assume that $p \leq q \leq r$.

A {\it spherical CR-structure} on a smooth 3-manifold $M$ is a maximal collection of distinguished charts modeled on the boundary $\partial \mathbf{H}^2_{\mathbb C}$
of  $\mathbf{H}^2_{\mathbb C}$, where coordinates changes are restrictions of transformations from  $\mathbf{PU}(2,1)$.
In other words, a {\it spherical CR-structure} is a $(G,X)$-structure with $G=\mathbf{PU}(2,1)$ and $X=\mathbb{S}^3$.
 A spherical CR-structure on a 3-manifold $M$  is {\it uniformizable} if it is
obtained as $M=\Gamma\slash \Omega_{\Gamma}$, where  $\Omega_{\Gamma}\subset \partial \mathbf{H}^2_{\mathbb C}$ is the set of discontinuity of a discrete subgroup  $\Gamma$ acting on $\partial \mathbf{H}^2_{\mathbb C}=\mathbb{S}^3$. The \emph{limit set} $\Lambda_{\Gamma}$ of $\Gamma$ is  $\mathbb{S}^3- \Omega_{\Gamma}$ by definition.
For a discrete group $\Gamma < \mathbf{PU}(2,1)$, the 3-manifold $M=\Omega_{\Gamma}/\Gamma$ at infinity of the 4-manifold $\mathbf{H}^2_{\mathbb C}/ \Gamma$  is the analogy of the 2-manifold at infinity of a geometrically finite,  infinite volume hyperbolic 3-manifold. In other words, uniformizable spherical CR-structures  on 3-manifolds in $\mathbf{H}^2_{\mathbb C}$-geometry are the analogies
of  conformal structures on surfaces in $\mathbf{H}^3_{\mathbb R}$-geometry.
But in contrast to results on other geometric structures carried on 3-manifolds, there are relatively few examples known about uniformizable  spherical CR-structures.

%$M=\Gamma\backslash \Omega_{\Gamma}$

There are some remarkable works on deformations of groups into $\mathbf{PU}(2,1)$.
Let $T(p,q,r)$ be the abstract $(p,q,r)$ reflection triangle group with the presentation
$$T(p,q,r)=\langle \sigma_1, \sigma_2, \sigma_3 | \sigma^2_1=\sigma^2_2=\sigma^2_3=(\sigma_2 \sigma_3)^p=(\sigma_3 \sigma_1)^q=(\sigma_1 \sigma_2)^r=id \rangle,$$
where $p,q,r$ are positive integers or $\infty$ satisfying $$\frac{1}{p}+\frac{1}{q}+\frac{1}{r}<1.$$ If $p,q$ or $r$ equals $\infty$, then
the corresponding relation does not appear.  The  ideal triangle group is the case that $p=q=r=\infty$.
A \emph{$(p,q,r)$ complex hyperbolic triangle group} is a representation $\rho$ of $T(p,q,r)$ into $\mathbf{PU}(2,1)$
where the generators fix complex lines, we denote $\rho(\sigma_{i})$ by $I_{i}$, and the image group by $\Delta_{p,q,r}=\langle I_1,I_2,I_3 \rangle$. It is well known  that the space of $(p,q,r)$-complex reflection triangle groups has real dimension one if $3 \leq p \leq q \leq r$. Sometimes, we denote the image of the  representation of the triangle group $T(p,q,r)$ into $\mathbf{PU}(2,1)$ such that $I_1 I_3I_2 I_3$ is  order $n$ by $\Delta_{p,q,r;n}$, so the representation now is not faithful, but in some case it is discrete.

  Goldman and Parker  initiated the study of the deformations of ideal triangle group into  $\mathbf{PU}(2,1)$ in \cite{GoPa}.
They  gave an interval  in the moduli space of complex hyperbolic ideal triangle groups, for points in this interval  the corresponding representations are discrete and faithful.
They conjectured that a complex hyperbolic ideal triangle group $\Gamma=\Delta_{\infty,\infty, \infty}=\langle I_1, I_2, I_3 \rangle$ is discrete and faithful if and only if $I_1 I_2 I_3$ is not elliptic. Schwartz proved  Goldman-Parker's conjecture in \cite{Schwartz:2001ann, schwartz:2006}. Furthermore,  Schwartz analyzed the complex hyperbolic ideal triangle group $\Gamma$ when $I_1 I_2 I_3$ is parabolic, and showed  the 3-manifold at infinity of the quotient space ${\bf H}^2_{\mathbb C}/{\Gamma}$ is commensurable with the
Whitehead link complement in the 3-sphere. In particular, the
Whitehead link complement  admits uniformizable  spherical CR-structures.
 Richard Schwartz has also  conjectured the necessary and sufficient condition for a general complex hyperbolic  triangle group $\Delta_{p,q,r}=\langle I_1,I_2,I_3\rangle < \mathbf{PU}(2,1)$ to be a discrete and faithful  representation of $T(p,q,r)$ \cite{Schwartz-icm}. Schwartz's conjecture has been proved in a few cases \cite{ParkerWX:2016, ParkerWill:2017}.

 From above we know one  way to study discrete subgroups of  $\mathbf{PU}(n,1)$ is the deformations of a representation. From a  finitely presented abstract  group $G$, and a discrete faithful representation $\rho_{0}: G \rightarrow \mathbf{PU}(n,1)$,  we may deform $\rho_{0}$  to $\rho_{1}: G \rightarrow \mathbf{PU}(n,1)$ along a path. We are interested in   whether  $\rho_{1}$ is discrete and faithful. Moreover, even when  $\rho_{1}$ is not faithful, but it also has the chance to be discrete. This case is very interesting, since if we are lucky, we have the  opportunity to get  a complex hyperbolic lattice at  $\rho_{1}$ \cite{dpp:2016, dpp:2021}.

One of the most important questions in complex hyperbolic geometry is
 the existence  of (infinitely many commensurable classes of)  non-arithmetic complex hyperbolic lattices  \cite{Margulis, Fisher:2021}. Right now, people only found 22
 commensurable classes of  non-arithmetic complex hyperbolic lattices in $\mathbf{PU}(2,1)$ \cite{DeligneMostow:1986,dpp:2016, dpp:2021}, and
  2
 commensurable classes of  non-arithmetic complex hyperbolic lattices in $\mathbf{PU}(3,1)$ \cite{DeligneMostow:1986, Deraux:2020}. Both $\mathbf{PO}(3,1)$  and $\mathbf{PU}(2,1)$ are subgroups of $\mathbf{PU}(3,1)$. It is reasonable that deformations of some discrete groups in $\mathbf{PO}(3,1)$  into the larger group $\mathbf{PU}(3,1)$  may give   some discrete, but not faithful representations, which  may give new  ${\bf H}^3_{\mathbb C}$-lattices  as pioneered  by \cite{DeligneMostow:1986, dpp:2016, dpp:2021}.

% In this paper, we initial study complexification of 3-dimensional hyperbolic tetrahedral Coxeter polytopes,

In this paper, we initiate the study of  deformations of groups into $\mathbf{PU}(3,1)$,  which is much more difficult and richer than deformations of  groups into $\mathbf{PO}(3,1)$  and $\mathbf{PU}(2,1)$.  By this we mean:

\begin{itemize}
	
		\item it is well known that the space of discrete and faithful  representations of a group into $\mathbf{PO}(3,1)$ has fractal  boundary in general. For example, the Riley slice has a beautiful fractal  boundary in $\mathbb{C}$ (see Page VIII of  \cite{ASWY:2019});
	\item  people tend to guess that  the space of discrete and faithful  representations of a group into $\mathbf{PU}(2,1)$ has piece-wise smooth boundary. For one of the tractable  cases, the so called complex Riley slice,  which is 2-dimensional, see \cite{ParkerWill:2017}. Moreover, there are very few results on the space of discrete and faithful  representations of a group into $\widehat{\mathbf{PU}(2,1)}$. Here  $\widehat{\mathbf{PU}(2,1)}$ is the full isometry group of  ${\bf H}^2_{\mathbb C}$. To the author's knowledge, the only complete  classification  is in \cite{Falbelparker:2003}, where Falbel-Parker completed the study on the space of discrete and faithful  representations of $\mathbb{Z}_{2}*\mathbb{Z}_{3}$  into $\widehat{\mathbf{PU}(2,1)}$ (with one additional parabolic element), the moduli space is 1-dimensional.

\end{itemize}
So deformations of groups into $\mathbf{PU}(3,1)$   must have fractal boundaries in general, but we are very far from understanding them.

%and $\mathbf{PU}(2,1)<\widehat{\mathbf{PU}(2,1)}$ has index two

% with the ultimate goal is to find some non-arithmetic complex hyperbolic lattices.

Let $$G=\left\langle \iota_1, \iota_2, \iota_3, \iota_4 \Bigg| \begin{array}{c}    \iota_1^2=  \iota_2^2 = \iota_3^2=\iota_4^2=id,\\ [2 pt]
(\iota_1 \iota_3)^{2}=(\iota_1 \iota_4)^{3}=(\iota_2 \iota_4)^{2}=id
\end{array}\right\rangle$$
be an  abstract  group. We also let $$K= \langle \iota_1 \iota_2,\iota_3\iota_1,\iota_4 \iota_1 \rangle$$ be an index two subgroup of $G$, which is isomorphic to $\mathbb{Z}_{2}* \mathbb{Z}_{2}*\mathbb{Z}_{3}$.
We  study representations $\rho: G \rightarrow \mathbf{PU}(3,1)$,  where $\rho( \iota_{i})=I_{i}$ is a complex reflection about a  complex hyperbolic plane   in ${\bf H}^{3}_{\mathbb C}$   for $1 \leq i \leq 4$,  with the additional condition that  $I_1I_2$ is parabolic.  We also assume two pairs of hyper-parallel complex hyperbolic planes  have the same distance for simplicity,  then
the moduli space $\mathcal{M}$ is parameterized by $(h,t) \in  [1, \infty) \times [0, \pi]$   with the condition $t \leq \operatorname{arccos}(-\frac{3h^2+1}{4h^2})$. For a point $(h,t) \in \mathcal{M}$, we denote the corresponding representation by $\rho_{(h,t)}$.  Moreover, $t=0$  and $t=\operatorname{arccos}(-\frac{3h^2+1}{4h^2})$ degenerate to ${\bf H}^{3}_{\mathbb R}$-geometry and  ${\bf H}^{2}_{\mathbb C}$-geometry, by this we mean the group $\rho_{(h,t)}(G)$ preserves a   ${\bf H}^{3}_{\mathbb R} \hookrightarrow {\bf H}^{3}_{\mathbb C}$ or a ${\bf H}^{2}_{\mathbb C}\hookrightarrow {\bf H}^{3}_{\mathbb C}$  invariant respectively.  See Section \ref{sec:moduli} for more details.

 As a model of the proof via Ford domains, the first main result in this paper is

 \begin{thm}\label{thm:3-mfd} $\rho_{(\sqrt{2},  \operatorname{arccos}(-\frac{7}{8}))}$   is a discrete and faithful representation of $G$ into $\mathbf{PU}(2,1)$. Moreover, the 3-manifold at infinity  of  $\Sigma=\rho_{(\sqrt{2}, \operatorname{arccos}(-\frac{7}{8}))}(K)  <\mathbf{PU}(2,1)$ is the connected sum of the trefoil knot  complement in $\mathbb{S}^3$ and a real projective space  ${\mathbb R}{\mathbf P}^3$.
\end{thm}

We denote $A=I_1I_2$, $B=I_3I_1$ and $C=I_4I_1$.
Let
 $$R=\{A^{k}CA^{-k},A^{k}CBCA^{-k},A^{k}C^{-1}BC^{-1}A^{-k}, A^{k}C^{-1}BCA^{-k}, A^{k}CBC^{-1}A^{-k}\} _{k \in \mathbb{Z}}$$
 be a subset of $\rho_{(\sqrt{2}, \operatorname{arccos}(-\frac{7}{8}))}(K)=\Sigma$. We will show the partial Ford domain 
$D_{R}(\Sigma)$ is in fact the Ford domain  of $\Sigma$, then we obtain Theorem \ref{thm:3-mfd}.

%Then as a  generalization of the proof of  Theorem \ref{thm:3-mfd}, we have

% \begin{thm}\label{thm:2dimcomplexhyp}  For any $h \geq \sqrt{2}$, when  $t=\operatorname{arccos}(-\frac{3h^2+1}{4h^2})$, $\rho_{(h,t)}$   is a discrete and faithful  representation of $G$ into $\mathbf{PU}(2,1)$,  the 3-manifold at infinity  of $\rho_{(h,t)}(K)$ is also the connected sum of the trefoil knot  complement in $\mathbb{S}^3$ and a real projective space  ${\mathbb R}{\mathbf P}^3$.
%\end{thm}

%HHHHHHHHHHHHHHHHH WHEN PARA the intersection of two isometric spheres, the %same angles as $\operatorname{arccos}(-\frac{7}{8})$, that is  $(s1,s2)=(0,\pi)$ ,$(s1,s2)=(\pi,\pi)$ or $s1=s2= \frac{\pi}{2} , %\frac{\pi}{3}$   HHHHHHHHHHHHHHHHHhh

Using the 
Ford domain in Theorem \ref{thm:3-mfd} as a guide, the  second main result of this paper  is

\begin{thm}\label{thm:complex3dim} There is a neighborhood $U$ of $(\sqrt{2}, \operatorname{arccos}(-\frac{7}{8}))$ in $\mathcal{M}$, such that   $\rho_{(h,t)}$ is a discrete and faithful representation of $G \rightarrow \mathbf{PU}(3,1)$ when $(h,t) \in U$.
\end{thm}

 To our knowledge, Theorem \ref{thm:complex3dim} is the first nontrivial  example of the Ford domain of a subgroup in  $\mathbf{PU}(3,1)$ that has been studied. Ford domains in ${\bf H}^{3}_{\mathbb C}$-geometry are highly difficult to study in general. The author  believes that right now we can only study very special cases of Ford domains in ${\bf H}^{3}_{\mathbb C}$-geometry, there are many fundamental results/tools demand to establish. For example we do not how to show the intersection of three isometric spheres ${\bf H}^{3}_{\mathbb C}$ is topologically a 3-ball (if non-empty). In \cite{Ma:2023}, the author studied the Dirichlet domain of a variety  of discrete subgroups in $\mathbf{PU}(3,1)$ (which are highly symmetric and so are  very lucky).  We hope this paper  may  attract more interest  on  this promising  direction.

%To our knowledge this is the first nontrivial  example of the Ford domain of a subgroup in  $\mathbf{PU}(3,1)$ that has been studied

%We note that $\operatorname{arccos}(-\frac{7}{8})=2.636232144$ and $\frac{5\pi}{6}=2.617993878$ numerically.   So the two representations  $\rho_{(\sqrt{2}, \operatorname{arccos}(-\frac{7}{8}))}$ and   $\rho_{(\sqrt{2}, \frac{5\pi}{6})}$ are very near (see Figure \ref{figure:moduli}), and then we guess the Ford domains of them are reasonable 
%related. 
 %In fact, we use the same argument as in the proof of  the proof of  Theorem \ref{thm:3-mfd}. 

 Our proof of Theorem \ref{thm:complex3dim} runs the same line as the proof of  Theorem \ref{thm:3-mfd}, so it seems the proof of it is very short here.  But in fact the proof  is technically more difficult, and the reader have to read Section \ref{sec:Ford2dim} before the proof of Theorem  \ref{thm:complex3dim}. Moreover we can  only prove it for a neighborhood $U$ of  $(\sqrt{2}, \operatorname{arccos}(-\frac{7}{8}))$ in $\mathcal{M}$, but we can not sketch how big the neighborhood is. 
 
 One of the technical  reasons that we can prove Theorem \ref{thm:complex3dim} is that  four points, say 
$q_{\infty}$, $C^{-1}(q_{\infty})$,  $CBC^{-1}(q_{\infty})$ and $CBC(q_{\infty})$,  are co-planar. This fact is a little surprising to the author, see Lemma \ref{lemma:coplane}.
 From the proof of Theorem 
 \ref{thm:3-mfd},  we know the triple intersection of three isometric spheres  $I(C)$,  $I(CBC^{-1})$ and $I(C^{-1}BC^{-1})$ in ${\bf H}^{2}_{\mathbb C}$ for the group $\rho_{(\sqrt{2}, \operatorname{arccos}(-\frac{7}{8}))}(K)  <\mathbf{PU}(2,1)$ is a union of two crossed straight segments, see Proposition \ref{prop:cBcandCintersecCBinverseC} and Figure 	\ref{figure:inverseCBinverseCandCintersectCBinverseC}. Then the intersection of three isometric spheres  $I(C)$,  $I(CBC^{-1})$ and $I(C^{-1}BC^{-1})$ for the group $\rho_{(\sqrt{2}, \operatorname{arccos}(-\frac{7}{8}))}(K)  <\mathbf{PU}(3,1)$ in ${\bf H}^{3}_{\mathbb C}$ is a union of two 3-balls which intersect in a 2-ball. So the three  isometric spheres  $I(C)$,  $I(CBC^{-1})$ and $I(C^{-1}BC^{-1})$ are not in general position in ${\bf H}^{3}_{\mathbb C}$. By this we mean if three  5-balls in a 6-ball  are in general position, then the intersection of them is a disjoint union of 3-balls. At first, the author guessed the facets of the Ford domain of  $\rho_{(h, t)}$ in ${\bf H}^{3}_{\mathbb C}$  are in general position  when $(h,t)$ is near to $(\sqrt{2}, \operatorname{arccos}(-\frac{7}{8}))$ but $t < \operatorname{arccos}(-\frac{3h^2+1}{4h^2})$, by this we mean we guessed the intersections of any three facets of the Ford domain is a disjoint union of 3-balls (if it is not empty). And then when $t$ converges to  $\operatorname{arccos}(-\frac{7}{8})$, the Ford domain of  $\rho_{(\sqrt{2},t)}(K)$ degenerates to the non-generic Ford domain of  $\rho_{(\sqrt{2},\operatorname{arccos}(-\frac{7}{8}))}(K)$. But with the careful  calculation, we show the Ford domain of  $\rho_{(h, t})$ in ${\bf H}^{3}_{\mathbb C}$ is not in  general position, in fact it has the same combinatorics as $\rho_{(\sqrt{2},\operatorname{arccos}(-\frac{7}{8}))}$, see  Lemma  \ref{lemma:intersection}, which is also a little surprising to the author.

%We also have
% \begin{thm}\label{thm:nearreal} For any $\epsilon >0$, there is a neighborhood $U$ of $[1+\epsilon, \infty)$ in the moduli space $\mathcal{M}$, such that for any $(h,t) \in U$,   $\rho_{(h,t)}$  is a discrete and faithful  representation of $G \rightarrow \mathbf{PU}(3,1)$.
%\end{thm}

%Recall that  $A=I_1I_2$, $B=I_3I_1$ and $C=I_4I_1$.
%Let
%$$S=\{A^{k}CA^{-k},A^{k}BA^{-k}\} _{k \in \mathbb{Z}}$$
%be a subset of $\rho_{(h,t)}(K)$, we will show the partial Ford domain 
%$D_{S}(\rho_{(h,t)}(K))$ is in fact the Ford domain of  $\rho_{(h,t)}(K)$ when $(h,t) \in U$, then we get Theorem \ref{thm:nearreal}.
% Even through it is trivial that  $\rho_{(h,0)}$ is  discrete and faithful representation  for any $h \in [1, \infty)$, but we prove  Theorem \ref{thm:nearreal}  by first consider the Ford domain of  $\rho_{(h,0)}(K)$ for $h \in [1, \infty)$ in  ${\bf H}^{3}_{\mathbb C}$,  which is not difficult. We show the combinatorics of the  Ford domains are generic  when $t=0$, so the  combinatorics of the Ford domain  are preserved
% under small deformations, which is also  the strategy used in 	\cite{Deraux:2016gt}.

%(it seems it changes three times)

%\textbf{Outline of the paper}
 {\bf The paper is organized as follows.} In Section \ref{sec:background}  we give well known background
 material on complex hyperbolic geometry. In Section \ref{sec:gram}, we give the matrix representations of $G$ into $\mathbf{PU}(3,1)$ with complex reflection generators.
  Section \ref{sec:Ford2dim}  is devoted to the descriptions of the isometric spheres that bound the Ford domain of  $\rho_{(\sqrt{2},\operatorname{arccos}(-\frac{7}{8}))}(K) <\mathbf{PU}(2,1)$. Based on Section \ref{sec:Ford2dim}, we study the  3-manifold at infinity  of $\rho_{(\sqrt{2},\operatorname{arccos}(-\frac{7}{8}))}(K)<\mathbf{PU}(2,1)$ in Section \ref{sec:3mfd}.
 We prove Theorem \ref{thm:complex3dim}  in Section  \ref{sec:complex3dim}.
Finally,  we propose a few related questions in Section \ref{sec:question} which seem interesting. 

%   We prove Theorem \ref{thm:nearreal} in Section \ref{sec:nearreal}. 
%  , we recommend  the reader read the proof  in this section only after   Section   \ref{sec:Ford2dim}

\textbf{Acknowledgement}: The author would like to thank Ying Zhang for helpful discussions. The author also would like to thank his co-author Baohua Xie
\cite{MaX:2021}, the author learned a lot  from Baohua on complex hyperbolic geometry.

 \section{Background}\label{sec:background}

 The purpose of this section is to introduce briefly complex hyperbolic geometry. One can refer to Goldman's book \cite{Go} for more details.

% where ${\bf z}=(z_1,z_2,z_3)^T$ and ${\bf w}=(w_1,w_2,w_3)^T$

%$$\mathbf{z}=\left[\begin{matrix}
%z_1\\ z_2\\z_3\\z_4\end{matrix}\right],\quad \mathbf{w}=\left[\begin{matrix}
%w_1\\ w_2\\w_3\\4_4\end{matrix}\right]$$
%are vectors in ${\mathbb C}^4$.
%Take
\subsection{Complex hyperbolic space}  \label{subsec:chs}
Let ${\mathbb C}^{n,1}$  denote the vector space ${\mathbb C}^{n+1}$ equipped with the Hermitian
form  of signature $(n,1)$:
 $$\langle {\bf{z}}, {\bf{w}} \rangle=\bf{w}^* \cdot  H \cdot \bf{z},$$ where $ \bf{w}^*$ is the  Hermitian transpose  of  $\bf{w}$ and 
$$H=\begin{pmatrix}
0& 0&1\\
0 &I_{n-1} & 0\\
1 &0 & 0\\
\end{pmatrix}. $$
Then
the Hermitian form divides ${\mathbb C}^{n,1}$ into three parts $V_{-}, V_{0}$ and $V_{+}$, which are
\begin{eqnarray*}
  V_{-} &=& \{{\bf z}\in {\mathbb C}^{n+1}-\{0\} : \langle {\bf z}, {\bf z} \rangle <0 \}, \\
  V_{0} &=& \{{\bf z}\in {\mathbb C}^{n+1}-\{0\} : \langle {\bf z}, {\bf z} \rangle =0 \}, \\
  V_{+} &=& \{{\bf z}\in {\mathbb C}^{n+1}-\{0\} : \langle {\bf z}, {\bf z} \rangle >0 \}.
\end{eqnarray*}

Let $$ \mathbb{P}: {\mathbb C}^{n+1}-\{0\}\rightarrow {\mathbb C}{\mathbf P}^{n}$$ be  the canonical projection onto the  complex projective space.
Then the {\it complex hyperbolic space} ${\bf H}^{n}_{\mathbb C}$ is the image of $V_{-}$ in ${\mathbb C}{\mathbf P}^{n}$
by the  map ${\mathbb P}$  and its {\it ideal boundary}, or {\it  boundary at infinity}, is  the image of $V_{0}$ in
 ${\mathbb C}{\mathbf P}^{n}$, we denote it by $\partial {\bf H}^{n}_{\mathbb C}$.

There is a typical anti-holomorphic isometry $\iota$ of ${\bf H}^{n}_{\mathbb C}$. $\iota$ is given on the level of homogeneous coordinates by complex conjugate

\begin{equation}\label{antiholo}
\iota:\left[\begin{matrix} z_1 \\ z_2\\ \vdots \\ z_{n+1} \end{matrix}\right]
\longmapsto \left[\begin{matrix} \overline{z_1} \\
\overline{z_2} \\ \vdots \\\overline{z_{n+1}} \end{matrix}\right].
\end{equation}

%$$
%\iota:\left[\begin{matrix} z_1 \\ z_2\\ \vdots \\ z_{n+1} \end{matrix}\right]
%\longmapsto \left[\begin{matrix} \overline{z}_1 \\
%\overline{z}_2 \\ \vdots \\\overline{z}_{n+1} \end{matrix}\right].
%$$

\subsection{Totally geodesic submanifolds and complex reflections}
There are two kinds of totally geodesic submanifolds  in ${\bf H}^{n}_{\mathbb C}$:

\begin{itemize}

\item Given any point $x \in {\bf H}^{n}_{\mathbb C}$, and a complex linear subspace $F$ of dimension $k$ in the tangent space $T_{x}{\bf H}^{n}_{\mathbb C}$, there is a unique complete holomorphic totally geodesic     submanifold contains $x$ and is tangent  to $F$. Such a holomorphic submanifold a called a \emph{$\mathbb{C}^{k}$-plane}.  A $\mathbb{C}^{k}$-plane is the intersection of a complex $k$-dimensional projective subspace in  $\mathbb{C}{\bf P}^{n}$ with ${\bf H}^{n}_{\mathbb C}$, and it is holomophic to ${\bf H}^{k}_{\mathbb C}$.  A $\mathbb{C}^{1}$-plane is called a \emph{complex geodesic}. The intersection of a $\mathbb{C}^{k}$-plane with $\partial {\bf H}^{n}_{\mathbb C}=\mathbb{S}^{2n-1}$ is a smoothly embedded sphere $\mathbb{S}^{2k-1}$, which is called a  $\mathbb{C}^{k}$-chain.

\item  Corresponding to the compatible real structures on $\mathbb{C}^{n,1}$ are the real forms of ${\bf H}^{n}_{\mathbb C}$; that is, the maximal totally real totally geodesic subspaces of ${\bf H}^{n}_{\mathbb C}$, which has real dimension $n$.  A maximal totally real totally geodesic subspace of ${\bf H}^{n}_{\mathbb C}$  is the  fixed-point set  of an  anti-holomorphic isometry of ${\bf H}^{n}_{\mathbb C}$, we have give an example of  anti-holomorphic isometry $\iota$     in (\ref{antiholo}) of Subsection  \ref{subsec:chs}.
For the usual real structure, this submanifold is the real hyperbolic $n$-space ${\bf H}^{n}_{\mathbb R}$ with curvature $-\frac{1}{4}$. Any  totally geodesic subspace of a maximal totally real totally geodesic subspace is a  totally real totally geodesic subspace,  which  is the real hyperbolic $k$-space ${\bf H}^{k}_{\mathbb R}$ for some $k$.

\end{itemize}

Since the Riemannian sectional curvature of the  complex hyperbolic space
is non-constant, there are no totally geodesic hyperplanes in  ${\bf H}^{n}_{\mathbb C}$ when $n \geq 2$.

%Consider the complex hyperbolic space ${\bf H}^2_{\mathbb C}$  and its boundary $\partial{\bf H}^2_{\mathbb C}$. We define
%\emph{$\mathbb{C}$-circles} in $\partial{\bf H}^2_{\mathbb C}$ to be the boundaries %of complex geodesics in ${\bf H}^2_{\mathbb C}$. Analogously,
%We define \emph{$\mathbb{R}$-circles} in $\partial{\bf H}^2_{\mathbb C}$ to be the %boundaries of Lagrangian planes in ${\bf H}^2_{\mathbb C}$.

Let $L$ be a $(n-1)$-dimensional complex plane in ${\bf H}^{n}_{\mathbb C}$,  a {\it polar vector} of $L$ is the unique vector (up to scaling) perpendicular to this complex plane with
respect to the Hermitian form. A polar vector of a $(n-1)$-dimensional complex plane belongs to  $V_{+}$ and each vector in $V_{+}$ corresponds to a $(n-1)$-dimensional complex plane.
Moreover, let $L$ be a $(n-1)$-dimensional complex plane with polar vector
${\bf n}\in V_{+}$,
%then $L$ is the projection of the subspace $\{{\bf z}\in V_{-}:\langle {\bf z}, {\bf c} \rangle=0\}$.
then the {\it complex reflection} fixing $L$ with rotation angle $\theta$  is given by
$$
 I_{\bf n, \theta}({\bf z}) = -{\bf z}+(1-\mathrm{e}^{\theta   \mathrm{i}})\frac{\langle {\bf z}, {\bf n} \rangle}{\langle {\bf n},{\bf n}\rangle}{\bf n}.
$$
The complex plane $L$ is also called the \emph{mirror} of $I_{\bf n, \theta}$. In this paper, we may assume $\theta= \frac{2 \pi}{m}$ for some $m \in \mathbb{Z}_{\geq 2}$, and then the complex reflection has order $m$.

 \subsection{The isometries} The complex hyperbolic space is a K\"{a}hler manifold of constant holomorphic sectional curvature $-1$.
 We denote by $\mathbf{U}(n,1)$ the Lie group of $\langle \cdot,\cdot\rangle$ preserving complex linear
 transformations and by $\mathbf{PU}(n,1)$ the group modulo scalar matrices. The group of holomorphic
 isometries of ${\bf H}^{n}_{\mathbb C}$ is exactly $\mathbf{PU}(n,1)$. It is sometimes convenient to work with
 $\mathbf{SU}(n,1)$.
 The full isometry group of ${\bf H}^{n}_{\mathbb C}$ is
 $$\widehat{\mathbf{PU}(n,1)}=\langle \mathbf{PU}(n,1),\iota\rangle,$$
 where $\iota$ is the antiholomotphic isometry in Subsection   \ref{subsec:chs}.

%given on the level of homogeneous coordinates by complex conjugacy
% $$
% \iota:\left[\begin{matrix} z_1 \\ z_2 \\ z_3  \\ z_4\end{matrix}\right]
% \longmapsto \left[\begin{matrix} \overline{z}_1 \\
% \overline{z}_2 \\ \overline{z}_3 \\ \overline{z}_4  \end{matrix}\right].
% $$
 
 Elements of $\mathbf{SU}(n,1)$ fall into three types, according to the number and types of  fixed points of the corresponding
 isometry. Namely,
 \begin{itemize}
 	
 	\item
 	
  an isometry is {\it loxodromic} if it has exactly two fixed points 
 on $\partial {\bf H}^{n}_{\mathbb C}$;
 
 	\item  an isometry is  {\it parabolic} if it has exactly one fixed point
 	on $\partial {\bf H}^{n}_{\mathbb C}$;
 
	\item   an isometry is {\it elliptic}  when it has (at least) one fixed point inside ${\bf H}^{n}_{\mathbb C}$. 
	
 \end{itemize}

 	An element $A\in \mathbf{SU}(n,1)$ is called {\it regular} whenever it has  distinct eigenvalues,
 	an elliptic $A\in \mathbf{SU}(n,1)$ is called  {\it special elliptic} if
 	it has a repeated eigenvalue. Complex reflection about a ${\bf H}^{n-1}_{\mathbb C} \hookrightarrow {\bf H}^{n}_{\mathbb C}$ is an example of  special elliptic element in  $\mathbf{SU}(n,1)$.

% which is a 4-fold cover of  $\mathbf{PU}(n,1)$.

\subsection{Goldman's function}
%The types of isometries  of ${\bf H}^2_{\mathbb C}$  can be determined by the traces of their matrix realizations.
For an element in $\mathbf{PU}(2,1)$, there is a simple way to study its type.  We denote $tr(A)$ the trace of a  matrix $A$. Goldman \cite{Go} defined a function $\mathcal{G}: \mathbb{C} \rightarrow \mathbb{R}$, which can determine whether $A \in \mathbf{PU}(2,1)$ is loxodromic or not. The function $\mathcal{G}$ is
\begin{equation}\label{goldmanfunction}
\mathcal{G}(z)= |z|^4-8 Re(z^3)+18|z|^2-27
\end{equation}
for $z \in \mathbb{C}$.

\begin{thm} (Goldman \cite{Go}) \label{thm:Goldman} The trace map $tr:  \mathbf{SU}(2,1) \rightarrow \mathbb{C}$ is surjective. For $A \in \mathbf{SU}(2,1)$:
	\begin{itemize}
		\item $A$ is regular elliptic if and only if  $\mathcal{G}(tr(A)) < 0$;
		
		\item $A$ is loxodromic  if and only if  $\mathcal{G}(tr(A))  > 0$;

		\item $A$ is unipotent   if and only if
		$tr(A) \in 3 \cdot\{1, \omega, \omega^2\}$;
		
		\item If
		$tr(A) \notin 3 \cdot \{1, \omega, \omega^2\}$, and $f(tr(A))=0$, $A$ maybe a $\mathbb{C}$-reflection about a $\mathbb{C}$-line, a $\mathbb{C}$-reflection about a point in ${\bf H}^2_{\mathbb C}$, or an ellipto-parabolic element.
	\end{itemize}
\end{thm}

We note that for $A \in \mathbf{PU}(2,1)$, there are three lifts $A_1$, $A_2=\omega A_1$ and $A_3=\omega^2 A_1$ of $A$ in  $\mathbf{SU}(2,1)$, so $tr(A)$ is only well-defined up to scaling by $\omega=\dfrac{-1+ \mathrm{i}\sqrt{3}}{2}$ and $\omega^2$, but
the function $\mathcal{G}$ also has a $\mathbb{Z}_3$-symmetry about the multiplicity  by $\omega$ and $\omega^2$, so whether $\mathcal{G}(tr(A))$ is bigger, equal, or smaller than zero is independent on the lift of $A$  to  $\mathbf{SU}(2,1)$.

In particular, suppose that $A\in\mathbf{SU}(2,1)$ has real trace. Then $A$ is elliptic if and only if  $-1\leq{\rm{tr}(A)}<3$.
Moreover, $A$ is unipotent if and only if ${\rm{tr}(A)}=3$. If ${\rm{tr}(A)}=-1,0,1$, $A$ is elliptic of order 2, 3, 4 respectively.
The locus of $z \in \mathbb{C}$ such  that $\mathcal{G}(z)=0$ is a concave deltoid with three vertices $3$, $\dfrac{-3+3\sqrt{3} \mathrm{i}}{2}$ and $\dfrac{-3-3\sqrt{3} \mathrm{i}}{2}$ respectively, see Figure 6.1 of  \cite{Go}.

%There is a special class of unipotent elements in $\mathbf{SU}(3,1)$,
%suppose that $T\in \mathbf{SU}(3,1)$ has trace  $3$, then all eigenvalues of $T$ equal to $1$, that is $T$ is {\it unipotent}.

\subsection{Holy grail function}
%The types of isometries  of ${\bf H}^2_{\mathbb C}$  can be determined by the traces of their matrix realizations.

Gongopadhyay-Parker-Parsad  \cite{GongopadhyayPP}  generalized Goldman's function to $\mathbf{PU}(3,1)$. 
 For a  matrix $A \in \mathbf{PU}(3,1)$, we denote by $tr(A)$ the trace of $A$, then  $\tau(A)=tr(A)=a+ b\cdot \mathrm{i}$ with $a,b \in \mathbb{R}$, and $\sigma(A)=\frac{tr^2(A)-tr(A^2)}{2} \in \mathbb{R}$. Then  the characteristic polynomial of $A$ is $$\chi_{A}(X)=X^4- \tau X^3 + \sigma X^2 - \overline{\tau} X+1.$$
Gongopadhyay-Parker-Parsad \cite{GongopadhyayPP} classified the dynamical action of $A \in \mathbf{PU}(3,1)$ on $\overline{\bf H}^3_{\mathbb C}$.

% which is a generalization of Goldman's work on $\mathbf{PU}(2,1)$ \cite{Go}.

For $A \in \mathbf{SU}(3,1)$, consider the function
\begin{equation}\label{HGfunction}
\mathcal{H}(A)=\mathcal{H}(\tau, \sigma)=4(\frac{\sigma^2}{3}-|\tau|^2+4)^3-27(\frac{2 \sigma^3}{27}-\frac{|\tau|^2 \sigma}{3}-\frac{8 \sigma}{3}+(\tau^2+\overline{\tau}^2))^2,
\end{equation}
 which is exactly the resultant of the characteristic polynomial of $A$.

\begin{thm} (Gongopadhyay-Parker-Parsad  \cite{GongopadhyayPP}) \label{thm:holly}For $A \in \mathbf{SU}(3,1)$:
	\begin{itemize}
		\item $A$ is regular elliptic if and only if  $\mathcal{H}(A) > 0$;
		
			\item $A$ is regular loxodromic  if and only if  $\mathcal{H}(A)< 0$;

			\item $A$ has a repeated eigenvalue if and only if   $\mathcal{H}(A)=0$. Moreover,  if $\mathcal{H}(A)=0$ and $A$ is diagnalisable, then $A$ is either elliptic or loxodromic. Moreover,  if $\mathcal{H}(A)=0$ and $A$ is not diagnalisable, then $A$ is parabolic.
		
		\end{itemize}
\end{thm}

%\begin{figure}
%	\begin{center}
%		\begin{tikzpicture}
%		\node at (0,0) {\includegraphics[width=10cm,height=10cm]{{holygailwhisters.png}}};
%		\end{tikzpicture}
%%	\end{center}
%	\caption{The holy grail.}
%	\label{figure:holygailwhisters}
%\end{figure}

\begin{figure}
	\begin{center}
		\begin{tikzpicture}
		\node at (0,0) {\includegraphics[width=12cm,height=8cm]{{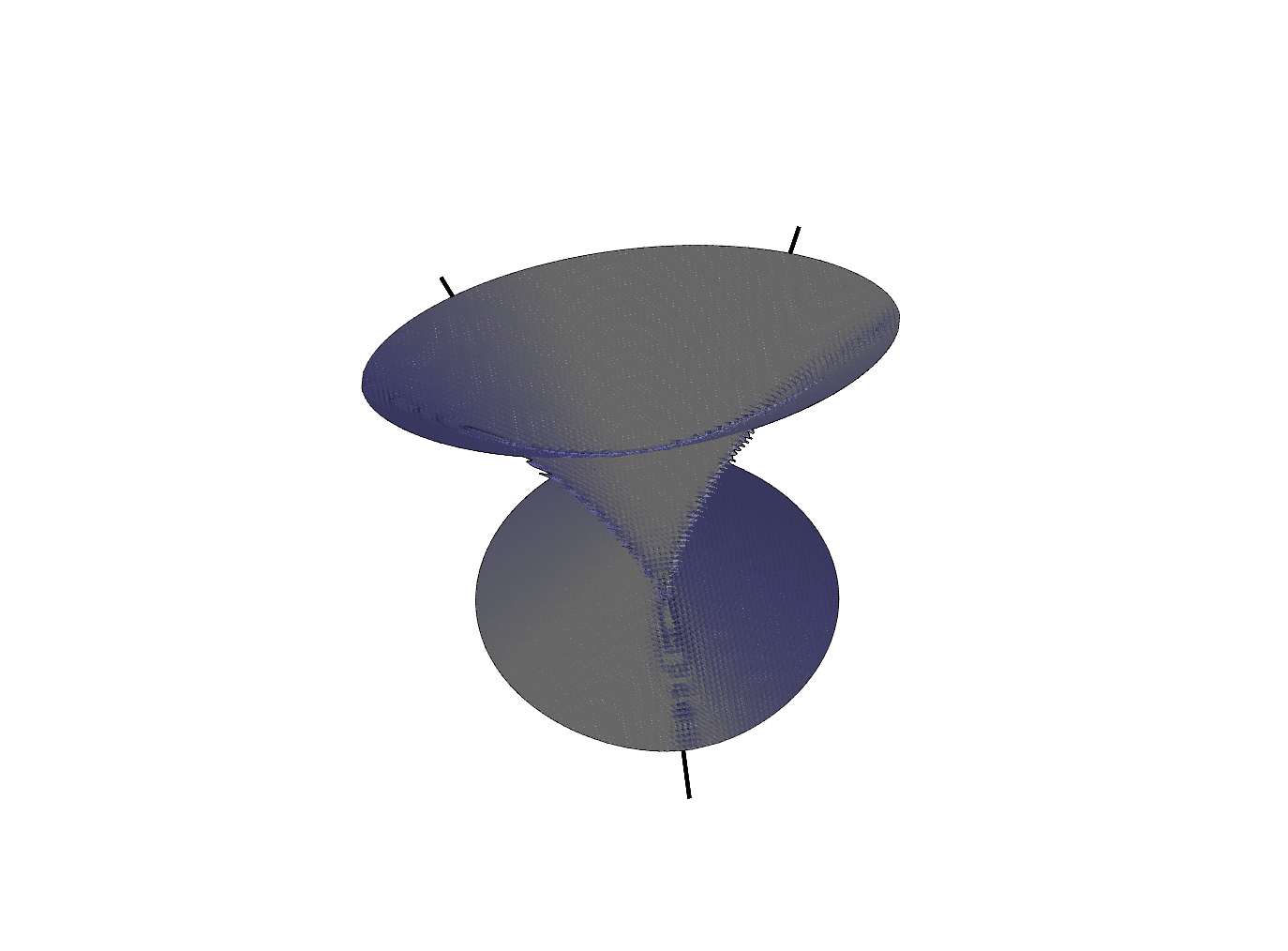}}};
		\end{tikzpicture}
	\end{center}
	\caption{The holy grail.}
	\label{figure:holygailwhisters}
\end{figure}
The locus of  $(\tau, \sigma) \in \mathbb{C} \times \mathbb{R}$ such $\mathcal{H}(\tau, \sigma)=0$
is the so called \emph{holy grail}. The holy grail  is a union of a  ruled surface and four \emph{whiskers}. The  ruled surface consists of three parts: upper bowl, central tetrahedron and lower bowel. See Figure 2 of  \cite{GongopadhyayPP} and Figure 	\ref{figure:holygailwhisters}.

%\begin{prop} (\cite{GongopadhyayPP}) \label{thm:whiskers}The whiskers in the holy grail are given by
%	$$(\tau, \sigma)=(\pm 2\cosh(k), 4 \cosh^2(k)+2),$$
%		$$(\tau, \sigma)=(\pm 2 \mathrm{i} \cosh(k), -4 \cosh^2(k)-2)$$
%		where $k$ is a positive real parameter.
%\end{prop}

We remark that holy grail function $\mathcal{H}(A)$ for $A \in  \mathbf{PU}(3,1)$ is a degree-6 polynomial on the trace of $A$ and the other $\mathbb{R}$-invariant $\sigma$ of $A$, which is much more  difficult than the degree-4 function of Goldman.  This is one of the many reasons that deformations of a group into $\mathbf{PU}(3,1)$ is much difficult than into $\mathbf{PU}(2,1)$.

%Elements of $\mathbf{SU}(2,1)$ fall into three types, according to the number and types of  fixed points of the corresponding
%isometry. Namely, an isometry is {\it loxodromic} (resp. {\it parabolic}) if it has exactly two fixed points (resp. one fixed point)
%on $\partial {\bf H}^2_{\mathbb C}$. It is called {\it elliptic}  when it has (at least) one fixed point inside ${\bf H}^2_{\mathbb C}$.
%An elliptic $A\in \mathbf{SU}(2,1)$ is called {\it regular elliptic} whenever it has three distinct eigenvalues, and {\it special elliptic} if
%it has a repeated eigenvalue. Complex reflection about a $\mathbb{C}$-line is an example of  special elliptic element.
%There is a special class of unipotent elements in $\mathbf{SU}(2,1)$,
%suppose that $T\in \mathbf{SU}(2,1)$ has trace  $3$, then all eigenvalues of %$T$ equal to $1$, that is $T$ is {\it unipotent}.

\subsection{Bisectors, isometric spheres and Ford polyhedron}

Isometric spheres are special examples of bisectors. In this subsection, we will describe a convenient set of
coordinates for bisector intersections, deduced from  the slice decomposition of a bisector.
\begin{defn} Given two distinct points $p_0$  and $p_1$ in ${\bf H}^{n}_{\mathbb C}$ with the same norm (e.g. one could
	take $\langle p_0,p_0\rangle=\langle p_1,p_1\rangle= -1$), the \emph{bisector} $\mathcal{B}(p_0,p_1)$ is the projectivization  of the set of negative vector
	$x$ with
	$$|\langle x,p_0\rangle|=|\langle x,p_1\rangle|.$$
\end{defn}

The  {\it spinal sphere} of the bisector $\mathcal{B}(p_0,p_1)$ is the intersection of $\partial {\bf H}^{n}_{\mathbb C}$ with the closure of $\mathcal{B}(p_0,p_1)$ in $\overline{{\bf H}^{n}_{\mathbb C}}= {\bf H}^{n}_{\mathbb C}\cup \partial { {\bf H}^{n}_{\mathbb C}}$. The bisector $\mathcal{B}(p_0,p_1)$ is a topological $(2n-1)$-ball, and its spinal sphere is a $(2n-2)$-sphere.
The  {\it complex spine} of $\mathcal{B}(p_0,p_1)$ is the complex line through the two points $p_0$ and $p_1$. The {\it real spine} of $\mathcal{B}(p_0,p_1)$
is the intersection of the complex spine with the bisector itself, which is a (real) geodesic; it is the locus of points inside the complex spine which are equidistant from $p_0$ and $p_1$.
Bisectors are not totally geodesic, but they have a very nice foliation by two different families of totally geodesic submanifolds. Mostow \cite{Mostow:1980}  showed that a bisector is the preimage of the real
spine under the orthogonal projection onto the complex spine. The fibres of this projection are complex planes ${\bf H}^{n-1}_{\mathbb C} \hookrightarrow {\bf H}^{n}_{\mathbb C}$ called the {\it complex slices} of the bisector. Goldman \cite{Go} showed that a bisector is
the union of all  totally real totally geodesic planes containing the real spine. Such Lagrangian planes are called the {\it real slices} or {\it meridians} of the bisector.

The standard lift $(z_1,z_2,\cdots, z_{n},1)^T$ of $z=(z_1,z_2, \cdots, z_{n})\in \mathbb {C}^{n}$ is negative if and only if
$$z_1+|z_2|^2+\cdots +|z_{n}|^2+\overline{z}_1=2{\rm Re}(z_1)+|z_2|^2+\cdots +|z_{n}|^2<0.$$  Thus $\mathbb{P}(V_{-})$ is a paraboloid in ${\mathbb C}^{n}$, called the {\it Siegel domain}.
Its boundary $\mathbb{P}(V_{0})$ satisfies
$$2{\rm Re}(z_1)+|z_2|^2+\cdots +|z_{n}|^2=0.$$
Therefore, the Siegel domain has an  analogue construction of the upper half space model for the  real hyperbolic space ${\bf H}^{n}_{\mathbb R}$.

Let $\mathcal{N}=\mathbb{C}^{n-1}\times \mathbb{R}$ be the Heisenberg group with product
$$
[z,t]\cdot [w,s]=[z+w,t+s + 2 \cdot Im\langle \langle  z, w\rangle\rangle],
$$
where $z=(z_1, z_2, \cdots, z_{n-1}) \in \mathbb{C}^{n-1}$, $w=(w_1, w_2, \cdots, w_{n-1}) \in \mathbb{C}^{n-1}$, and 
$$\langle \langle  z, w \rangle \rangle=w^{*}\cdot z$$ is the standard positive Hermitian form  on $\mathbb{C}^{n-1}$.
Then the boundary of complex hyperbolic space $\partial {\bf H}^{n}_{\mathbb C}$ can be identified to the union $\mathcal{N}\cup \{q_{\infty}\}$, where $q_{\infty}$ is the point at infinity.
The \emph{standard lift} of $q_{\infty}$ and $q=[z_1,z_2, \cdots, z_{n-1},t]\in\mathcal{N}$ in $\mathbb{C}^{n+1}$ are
\begin{equation}\label{eq:lift}
{\bf{q}_{\infty}}=\left[
\begin{array}{c}
1 \\
0 \\
\vdots  \\
0 \\
0 \\
\end{array}
\right],  \quad
{\bf{q}}=\left[
\begin{array}{c}
\frac{-(|z_1|^2+|z_2|^2+\cdots+|z_{n-1}|^2)+it}{2} \\
z_1 \\
\vdots  \\
z_{n-1} \\
1 \\
\end{array}
\right].
\end{equation}

%We write $q=[z,t]\in\mathcal{N}$ for $z \in \mathbb{C}$ and $t \in \mathbb{R}$ or $q=[x, y,t]\in\mathcal{N}$ for $x,y,t$ reals and $z=x+yi$.

So complex hyperbolic space and its boundary ${\bf H}^{n}_{\mathbb C} \cup \partial {\bf H}^{n}_{\mathbb C}$ can be identified to ${\mathcal{N}}\times{\mathbb{R}_{\geq 0}}\cup q_{\infty}$.
Any point $q=(z,t,u)\in{\mathcal{N}}\times{\mathbb{R}_{\geq 0}}$ has the standard lift
$$
\quad
{\bf{q}}=\left[
\begin{array}{c}
	\frac{-(|z_1|^2+|z_2|^2+\cdots+|z_{n-1}|^2)-u+it}{2} \\
	z_1 \\
	\vdots  \\
	z_{n-1} \\
	1 \\
\end{array}
\right].
$$
Here $(z_1,z_2,\cdots, z_{n-1},t,u)$ is called the \emph{horospherical coordinates} of $\overline {\bf H}^{n}_{\mathbb{C}}={\bf H}^{n}_{\mathbb{C}} \cup \partial {\bf H}^{n}_{\mathbb{C}}$. The natural projection $\mathcal{N}=\mathbb{C}^{n-1} \times \mathbb{R} \rightarrow \mathbb{C}^{n-1}$ is called the {\it vertical projection}.

%\begin{equation}\label{eq:cygan-metric}
%d_{\textrm{Cyg}}(p,q)=|2\langle {\bf{p}}, {\bf{q}} \rangle|^{1/2}=\left| |z_1-w_1|^2+\cdots+|z_{n-1}-w_{n-1}|^2-i(t-s+2{\rm{Im}}\langle \langle  z, w \rangle \rangle) \right|^{1/2},
%\end{equation}

\begin{defn}
	The \emph{Cygan metric} $d_{\textrm{Cyg}}$ on $\partial {\bf H}^{n}_{\mathbb{C}} \backslash\{q_{\infty}\}$ is defined to be
	\begin{equation}\label{eq:cygan-metric}
	d_{\textrm{Cyg}}(p,q)=|2\langle {\bf{p}}, {\bf{q}} \rangle|^{1/2}=\left| |z_1-w_1|^2+\cdots+|z_{n-1}-w_{n-1}|^2-i(t-s+2{\rm{Im}}\langle \langle  z, w \rangle \rangle) \right|^{1/2},
	\end{equation}
	where $p=[z_1,\cdots, z_{n-1},t]$ and $q=[w_1, \cdots,w_{n-1},s]$.
	
	So the {\it Cygan sphere} with center $(w,s,0)$  and radius $r$ has equation
	$$d_{\textrm{Cyg}}\left((z,t,u),(w,s,0)\right)=\left||z_1-w_1|^2+\cdots+|z_{n-1}-w_{n-1}|^2+u+i(t-s+2{\rm Im}\langle \langle  z, w \rangle \rangle)\right|=r^2.$$

	The \emph{extended Cygan metric} on $ {\bf H}^{n}_{\mathbb{C}}$ is given by the formula
$$d_{\textrm{Cyg}}(p,q)=\left| |z_1-w_1|^2+\cdots+|z_{n-1}-w_{n-1}|^2+|u-v|-i(t-s+2{\rm{Im}}\langle \langle  z, w \rangle \rangle)) \right|^{1/2},
$$
	where $p=(z,t,u)$ and $q=(w,s,v)$.
\end{defn}

Suppose that $g=(g_{i,j})^{n+1}_{i,j=1}\in \mathbf{PU}(n,1)$ does not fix $q_{\infty}$. Then it is obvious that $g_{n+1,1}\neq 0$.
\begin{defn}
	The \emph{isometric sphere} of $g$, denoted by $I(g)$, is the set
	\begin{equation}\label{eq:isom-sphere}
	I(g)=\{ p \in { {\bf H}^{n}_{\mathbb C} \cup \partial  {\bf H}^{n}_{\mathbb C}} : |\langle {\bf{p}}, {\bf{q}}_{\infty} \rangle | = |\langle {\bf{p}}, g^{-1}({\bf{q}}_{\infty}) \rangle| \}.
	\end{equation}
	Here $\bf{p}$ is any lift of $p$.
\end{defn}
The isometric sphere $I(g)$ is the Cygan sphere with center
$$g^{-1}({\bf{q}}_{\infty})=\left[\frac{\overline{g_{n+1,2}}}{\overline{g_{n+1,1}}},\frac{\overline{g_{n+1,3}}}{\overline{g_{n+1,1}}},\cdots,\frac{\overline{g_{n+1,n}}}{\overline{g_{n+1,1}}},2 \cdot {\rm{Im}}(\frac{\overline{g_{n+1,n+1}}}{\overline{g_{n+1,1}}})\right]$$
and radius $r_g=\sqrt{\frac{2}{|g_{n+1,1}|}}$.
An isometric sphere is special bisector.

The \emph{interior} of $I(g)$ is  the component of its complement in ${\bf H}^{n}_{\mathbb C} \cup \partial {\bf H}^{n}_{\mathbb C}$
that do not contain $q_{\infty}$, namely
\begin{equation}\label{eq:exterior}
\{ p \in {{\bf H}^{n}_{\mathbb C} \cup \partial {\bf H}^{n}_{\mathbb C}} : |\langle {\bf{p}}, {\bf{q}}_{\infty} \rangle | > |\langle {\bf{p}}, g^{-1}({\bf{q}}_{\infty}) \rangle| \}.
\end{equation}
The \emph{exterior} of $I(g)$ is the set
\begin{equation}\label{eq:interior}
\{ p \in {{\bf H}^{n}_{\mathbb C} \cup \partial {\bf H}^{n}_{\mathbb C}} : |\langle {\bf{p}}, {\bf{q}}_{\infty} \rangle | < |\langle {\bf{p}}, g^{-1}({\bf{q}}_{\infty}) \rangle| \},
\end{equation}
which contains the point at infinity $q_{\infty}$.

We are interested in the intersection of Cygan spheres.
\begin{prop}[Goldman \cite{Go}, Parker and Will \cite{ParkerWill:2017}]
	The intersection of two Cygan spheres is connected.
\end{prop}

%We summarize without proofs the relevant properties on isometric sphere in the following proposition.
It is easy to have 
\begin{prop} Let $g$ and $h$ be elements of $\mathbf{PU}(n,1)$ which do not fix $q_{\infty}$, such that $g^{-1}(q_{\infty})\neq h^{-1}(q_{\infty})$ and let $f \in\mathbf{PU}(n,1)$ be an unipotent transformation  fixing $q_{\infty}$.  Then the followings hold:
	\begin{itemize}
		\item  $g$ maps $I(g)$ to $I(g^{-1})$, and the exterior of $I(g)$ to the interior of $I(g^{-1})$.
		\item  $I(gf)=f^{-1}I(g)$, $I(fg)=I(g)$.
		\item $g(I(g)\cap I(h))=I(g^{-1})\cap I(hg^{-1})$, $h(I(g)\cap I(h))=I(gh^{-1})\cap I(h^{-1})$.
	\end{itemize}
\end{prop}

\begin{defn}The \emph{Ford domain} $D_{\Gamma}$ for a discrete group $\Gamma < \mathbf{PU}(n,1)$ centered at $q_{\infty}$ is
	the intersection of the (closures of the) exteriors of all isometric spheres of elements in $\Gamma$ not fixing $q_{\infty}$. That is,
	$$D_{\Gamma}=\{p\in {\bf H}^{n}_{\mathbb C} \cup \partial{\bf H}^{n}_{\mathbb C}: |\langle \mathbf{p},q_{\infty}\rangle|\leq|\langle \mathbf{p},g^{-1}(q_{\infty})\rangle|,
	\ \forall g \in \Gamma \ \mbox{with} \ g(q_{\infty})\neq q_{\infty} \}.$$
\end{defn}

From the definition, one can see that parts of  isometric spheres form the boundary of a Ford domain.  For $g \in \Gamma$, we will denote by $s(g)=I(g)  \cap D_{\Gamma}$, which is a side of $D_{\Gamma}$.

When $q_{\infty}$ is either in the domain  of discontinuity or is a parabolic fixed point, the Ford  domain  is preserved by $\Gamma_{\infty}$, the stabilizer of
$q_{\infty}$ in $\Gamma$.  In this case, $D_{\Gamma}$ is only a fundamental domain modulo the action of $\Gamma_{\infty}$.  In other words, the fundamental domain for
$\Gamma$ is the intersection of the Ford domain with a fundamental domain for $\Gamma_{\infty}$. Facets of codimension one in $D_{\Gamma}$   will be called {\it sides}.  Facets of codimension two in $D_{\Gamma}$   will be called {\it ridges}.
Facets of dimension one and zero in $D_{\Gamma}$ will be called  {\it edges} and {\it vertices} respectively. Moreover, a  {\it bounded ridge} is a ridge which does not intersect  $\partial {\bf H}^{n}_{\mathbb C}$, and if the intersection of a ridge $r$ and $\partial {\bf H}^{n}_{\mathbb C}$ is non-empty, then
$r$ is an {\it infinite ridge}.

It is usually very hard to determine $D_{\Gamma}$ because one should check infinitely many inequalities.
Therefore a general method will be to guess the Ford polyhedron and check  it using the Poincar\'e polyhedron theorem.
The basic idea is that the sides of $D_{\Gamma}$ should be paired by isometries, and the images of $D_{\Gamma}$
under these so-called side-pairing maps should give a local tiling of ${\bf H}^{n}_{\mathbb C}$.  If they do (and if
the quotient of $D_{\Gamma}$ by the identification given by the side-pairing maps is complete), then the Poincar\'{e} polyhedron
theorem implies that the images of $D_{\Gamma}$ actually give a global tiling of ${\bf H}^{n}_{\mathbb C}$.

Once a fundamental domain is obtained, one gets an explicit presentation of $\Gamma$ in terms of the generators given by the side-pairing maps together
with a generating set for the stabilizer $\Gamma_{\infty}$, where the relations corresponding to so-called ridge cycles, which correspond to the local tilings bear each co-dimension two face. For more on the Poincar\'e polyhedron theorem, see \cite{dpp:2016, ParkerWill:2017}.

\subsection{Spinal coordinates for bisectors in  ${\bf H}^2_{\mathbb C}$}\label{subsection:Spinalcoordinates}

The intersections of isometric spheres in ${\bf H}^2_{\mathbb C}$ is a little easier to describe than in higher dimensions. We show this in this subsection.

Since isometric spheres are Cygan spheres, we now recall some facts about Cygan spheres.
Let $S_{[z_0,t_0]}(r)$ be the Cygan sphere with center $[z_0,t_0]$ and radius $r>0$ in ${\bf H}^2_{\mathbb C}$. Then
%\begin{equation}\label{eq:cygan-sphere}
$$
S_{[z_0,t_0]}(r)=\left\{ (z,t,u)\in\hc \cup \partial\hc: ||z-z_0|^2+u+\rm{i}(t-t_0+2 \cdot Im( z \cdot \overline{z_0})|=r^2 \right\}
$$
%\end{equation}
in horospherical coordinates.

\begin{defn}\label{def:geographic}
	The \emph{geographic coordinates} $(\alpha,\beta,w)$ of $q=q(\alpha,\beta,w)\in S_{[0,0]}(r)$ is given by the lift
	\begin{equation}\label{eq:geog-coor}
	{\bf{q}}={\bf q}(\alpha,\beta,w)=\left[
	\begin{array}{c}
	-r^2e^{-i\alpha}/2 \\
	rwe^{i(-\alpha/2+\beta)} \\
	1 \\
	\end{array}
	\right],
	\end{equation}
	where $\alpha\in [-\pi/2,\pi/2]$, $\beta\in [0, \pi)$ and $w\in [-\sqrt{\cos(\alpha)},\sqrt{\cos(\alpha)}]$.
	In particular, the ideal boundary of $S_{[0,0]}(r)$ on $\partial\hc$ are the points with $w=\pm\sqrt{\cos(\alpha)}$.
\end{defn}

The following property should be useful to describe the intersection of Cygan spheres.
\begin{prop}[Goldman \cite{Go}, Parker and Will \cite{ParkerWill:2017}]
	Let $S_{[0,0]}(r)$ be a Cygan sphere with {geographic coordinates} $(\alpha,\beta,w)$.
	\begin{enumerate}
		\item The spine of  $S_{[0,0]}(r)$ is given by $w=0$.
		\item The slices of  $S_{[0,0]}(r)$ are  given by $\theta=\theta_0$ for fixed $\theta_0\in [-\pi/2,\pi/2]$.
		\item The meridians of  $S_{[0,0]}(r)$ are  given by $\alpha=\alpha_0$ for fixed $\alpha_0\in [0,\pi)$.
		
	\end{enumerate}
\end{prop}

If
$$\mathbf{p}=\left[\begin{matrix}
p_1\\ p_2\\p_3\end{matrix}\right],\quad \mathbf{q}=\left[\begin{matrix}
q_1\\ q_2\\q_3\end{matrix}\right]$$ are lifts of $p,q$ in $\mathbf{H}^2_{\mathbb C}$, then the {\it Hermitian cross product} of $p$ and $q$  is defined by

$$\mathbf{p} \boxtimes
\mathbf{q}=
\left[\begin{matrix}
\overline{p}_3\overline{q}_2-\overline{p}_2\overline{q}_3\\ \overline{p}_1\overline{q}_3-\overline{p}_3\overline{q}_1\\ \overline{p}_1\overline{q}_2-\overline{p}_2\overline{q}_1
\end{matrix}
\right].$$
This vector is orthogonal to $p$ and $q$  with  respect to the Hermitian form  $\langle \cdot,\cdot\rangle$.
It is a Hermitian version of the Euclidean cross product.
In order to analyze  2-faces of a Ford polyhedron, we must study the intersections of isometric spheres.

From the detailed analysis in \cite{Go}, we know that the intersection of two bisectors is usually not totally geodesic and can be somewhat complicated. In  this paper, we shall only consider the intersection of
coequidistant bisectors, i.e. bisectors equidistant from a common point.   When $p,q$ and $r$ are not in a common complex line, that is,  their lifts are linearly independent in $\mathbb {C}^{2,1}$, then the locus $\mathcal{B}(p,q,r)$ of points in $ {\bf H}^2_{\mathbb C}$
equidistant to  $p,q$ and $r$ is a smooth disk that is not totally geodesic, and is often called a \emph{Giraud disk}.  The following property is crucial when studying fundamental domain.

\begin{prop}[Giraud] \label{prop:Giraud}
	If $p,q$ and $r$ are not in a common complex line, then the Giraud disk $\mathcal{B}(p,q,r)$ is contained in precisely three bisectors, namely $\mathcal{B}(p,q),  \mathcal{B}(q,r)$ and  $\mathcal{B}(p,r)$.
\end{prop}

Note that checking whether an isometry maps  a Giraud disk to another is equivalent to checking that corresponding triple of  points are mapped to each other.

In order to study Giraud  disks, we will use {\it spinal coordinates}. The complex slices of $\mathcal{B}(p,q)$ are given explicitly by choosing a lift $\mathbf{p}$ (resp. $\mathbf{q}$) of $p$ (resp. $q$).
When $p,q\in  {\bf H}^2_{\mathbb C}$, we simply choose lifts such that $\langle \mathbf{p},\mathbf{p}\rangle= \langle \mathbf{q},\mathbf{q}\rangle$.
In this paper, we will mainly use these parametrization when  $p,q\in  \partial{\bf H}^2_{\mathbb C}$. In that case, the condition
$\langle \mathbf{p},\mathbf{p}\rangle= \langle \mathbf{q},\mathbf{q}\rangle$  is vacuous, since all lifts are null vectors; we then choose some fixed lift $\mathbf{p}$
for  the center of the Ford domain, and we take $ \mathbf{q}=G( \mathbf{p})$ for some $G\in \mathbf{U}(2,1)$. For  a different matrix $G'=SG$, with $S$
is a scalar matrix, note  that the diagonal element of  $S$ is a unit complex number, so $\mathbf{q}$ is well defined up to a unit complex number.
The complex slices of $\mathcal{B}(p,q)$ are obtained as the set of negative lines $(\overline{z}\mathbf{p}-\mathbf{q})^{\bot}$ in ${\bf H}^2_{\mathbb C}$ for some arc of values of $z\in S^1$,  which is determined by requiring that $\langle \overline{z}\mathbf{p}-\mathbf{q},\overline{z}\mathbf{p}-\mathbf{q}\rangle>0$.

Since a point of the bisector is on precisely one complex slice, we can parameterize the {\it Giraud torus}  $\hat{\mathcal{B}}(p,q,r)$  in ${\bf P}^2_{\mathbb C}$   by $(z_1,z_2)=(e^{it_1},e^{it_2})\in S^1\times S^1 $ via
	\begin{equation} \label{equaation:girauddisk}
	V(z_1,z_2)=(\overline{z}_1\mathbf{p}-\mathbf{q})\boxtimes (\overline{z}_2\mathbf{p}-\mathbf{r})=\mathbf{q}\boxtimes \mathbf{r}+z_1 \mathbf{r}\boxtimes \mathbf{p}+z_2 \mathbf{p}\boxtimes \mathbf{q}.
	\end{equation}
The Giraud disk $\mathcal{B}(p,q,r)$ corresponds to the $(z_1,z_2)\in S^1\times S^1 $  with  $$\langle V(z_1,z_2),V(z_1,z_2)\rangle<0.$$ It follows from the fact the bisectors are covertical that  this region is a topological disk, see \cite{Go}.

The boundary at infinity $\partial\mathcal{B}(p,q,r)$ is a circle, given in spinal coordinates by the equation
$$\langle V(z_1,z_2),V(z_1,z_2)\rangle=0.$$

Note that the choices of two lifts of $q$ and $r$ affect the spinal coordinates by rotation on each of the $S^1$-factors.

A defining equation for the trace of another bisector $\mathcal{B}(u,v)$ on the Giraud disk $\mathcal{B}(p,q,r)$ can be written in the form
$$| \langle V(z_1,z_2),u\rangle|=| \langle V(z_1,z_2),v\rangle|,$$  provided that $u$ and $v$ are suitably chosen lifts.
The expressions $\langle V(z_1,z_2),u\rangle$ and $\langle V(z_1,z_2),v\rangle$ are affine in $z_1$ and $z_2$.

This triple bisector intersection can be parameterized  fairly explicitly, because one can solve the equation
$$|\langle V(z_1,z_2),u\rangle|^2=|\langle V(z_1,z_2),v\rangle|^2 $$ for one of the variables $z_1$ or $z_2$ simply by solving a quadratic equation.
A detailed explanation of how this works can be found in \cite{Deraux:2016gt, DerauxF:2015, dpp:2016}.

\section{The moduli space of representations of $G$ into $\mathbf{PU}(3,1)$  with $I_1I_2$ parabolic} \label{sec:gram}
\label{sec:moduli}

In this section, we give the matrix representations of $G$ into $\mathbf{PU}(3,1)$ with complex reflection generators.

\subsection{The Gram matrices of four complex hyperbolic planes in $\mathbf{H}^3_{\mathbb C}$ for the group $G$} \label{subsec:gram}

Recall that for two $\mathbb C$-planes $\mathcal{P}$ and $\mathcal{P}'$  in $\mathbf{H}^3_{\mathbb C}$ with  polar vectors  $n$ and $n'$ such that
$\langle n, n\rangle =\langle n', n'\rangle =1$:
\begin{itemize}
	\item
	If  $\mathcal{P}$ and $\mathcal{P}'$  intersect in a $\mathbb{C}$-line in $\mathbf{H}^3_{\mathbb C}$, then the angle $\alpha $ between
	them has $|\langle n, n'\rangle|=\cos(\alpha)$;

	\item If  $\mathcal{P}$ and $\mathcal{P}'$ are hyper-parallel  in  $\mathbf{H}^3_{\mathbb C}$,
	then the distance $d$ between them has $|\langle n, n'\rangle |=\cosh \frac{d}{2}$;   	
	
	\item If  $\mathcal{P}$ and $\mathcal{P}'$ are asymptotic in  $\mathbf{H}^3_{\mathbb C}$,  then $|\langle n, n' \rangle|=1$.\end{itemize}

We consider  $\mathbb C$-planes $\mathcal{P}_{i}$ for $i=1,2,3,4$, so each $\mathcal{P}_{i}$ is a  totally geodesic ${\bf H}^2_{\mathbb C} \hookrightarrow {\bf H}^3_{\mathbb C}$. Let $n_{i}$ be the polar vector of  $\mathcal{P}_{i}$ in $\mathbb{C}{\mathbf P}^{3}-\overline{\mathbf{H}}^3_{\mathbb C}$.
We assume
\begin{itemize}
	\item  the angle between $\mathcal{P}_{1}$ and $\mathcal{P}_{3}$ is $\frac{\pi}{2}$;
	\item the angle between $\mathcal{P}_{2}$ and $\mathcal{P}_{4}$ is $\frac{\pi}{2}$;
	
		\item the angle between $\mathcal{P}_{1}$ and $\mathcal{P}_{4}$ is $\frac{\pi}{3}$;
		\item the planes $\mathcal{P}_{1}$ and $\mathcal{P}_{2}$ are asymototic;
		\item the planes $\mathcal{P}_{2}$ and $\mathcal{P}_{3}$ are hyper-parallel with distance  $2 \cdot \operatorname{arccosh}(h)$  if $h>1$, and the planes $\mathcal{P}_{2}$ and $\mathcal{P}_{3}$ intersect in a $\mathbb{C}$-line with angle  $\frac{\pi}{m}$  if $h=\cos(\frac{\pi}{m})$;
		\item the planes $\mathcal{P}_{3}$ and $\mathcal{P}_{4}$ are hyper-parallel with distance  $2 \cdot \operatorname{arccosh}(h)$ if $h>1$, and the planes $\mathcal{P}_{3}$ and $\mathcal{P}_{4}$ intersect in a $\mathbb{C}$-line with angle  $\frac{\pi}{m}$  if $h=\cos(\frac{\pi}{m})$.
	\end{itemize}
Then we can normalize the Gram matrix
to  the following form
	$$\mathscr{G}=(
	\langle  n_{i},n_{j}\rangle)_{1\leq i, j \leq 4}=\begin{pmatrix}
	1& -1& 0&-\frac{\mathrm{e}^{t\mathrm{i}}}{2}\\
	-1& 1 & -h& 0\\
	0&-h&1 &-h\\
	-\frac{\mathrm{e}^{-t\mathrm{i}}}{2} & 0&-h&1\\
	\end{pmatrix}.$$
	Where up to anti-holomorphic isometry of  ${\bf H}^3_{\mathbb C}$, we may assume $t \in [0,\pi]$.
	Moreover $t=0$ corresponds the case of (infinite volume) 3-dimensional real hyperbolic Coxeter tetrahedra, see \cite{VinbergS:1993}.
	By 	\cite{CunhaDGT:2012}, for  the Gram matrix above, there is a unique configuration of four
	 $\mathbb C$-planes $\mathcal{P}_{i}$ in ${\bf H}^3_{\mathbb C}$ for $i=1,2,3,4$  up to $\mathbf{PU}(3,1)$  realize the Gram matrix.
	
	%the planes $\mathcal{P}_{2}$ and $\mathcal{P}_{3}$ are hyper-parallel with distance  $\operatorname{arccosh}( \frac{h}{2})$  if $h>1$

	Then it is easy to see $$\det(\mathscr{G})=-\frac{3h^2}{4}-\frac{1}{4}-h^2 \cos(t).$$ The eigenvalues of $\mathscr{G}$ are $$1 \pm \frac{\sqrt{16h^2+10\pm2 \sqrt{64h^4+64h^2 \cos(t)+25}}}{4}.$$
	When $$\cos(t)=-\frac{3h^2+1}{4h^2},$$ $\mathscr{G}$  has eigenvalues  $$0,~~2,~~ 1+\frac{\sqrt{8h^2+1}}{2},~~1-\frac{\sqrt{8h^2+1}}{2}.$$
  When $t \in ( \operatorname{arccos}(-\frac{3h^2+1}{4h^2}), \pi]$, $\mathscr{G}$ has signature $(2,2)$, we will not study them in this paper.
So our moduli space is
	\begin{equation}
	\mathscr{M}=\left\{(h,t) \in \mathbb{R}_{\geq \frac{1}{2}} \times[0, \pi] \Bigg|t \in \left[0, \operatorname{arccos}\left(-\frac{3h^2+1}{4h^2}\right)\right] \right\}.
	\end{equation}

	Let $G$ be the  abstract group with the presentation
	$$G=\left\langle \iota_{1}, \iota_{2},\iota_{3},\iota_{4} \Bigg| \begin{array}{c}    \iota_{1}^2= \iota_{2}^2= \iota_{3}^2=  \iota_{4}^2=id,\\[ 3 pt]
	( \iota_{1} \iota_{3})^{2}=  (\iota_{2} \iota_{4})^{2}= ( \iota_{1} \iota_{4})^3=id
	\end{array}\right\rangle.$$
 Then $K=\langle \iota_{1}\iota_{2},\iota_{3}\iota_{1},\iota_{4}\iota_{1} \rangle$ is an index  two subgroup of $G$, which is isomorphic to $\mathbb{Z}_2 * \mathbb{Z}_2* \mathbb{Z}_3$.

 	Now for any $(h,t)\in \mathcal{M}$, let $\rho_{(h,t)}:G \rightarrow  \mathbf{PU}(3,1)$  be the representation with $\rho_{(h,t)}(\iota_{i})=I_{i}$ be the order two $\mathbb{C}$-reflection about $\mathcal{P}_{i}$.
 	We also denote by $\Gamma=\Gamma(h,t) =\langle I_1,I_2,I_3,I_4 \rangle$.
	 When $t= \operatorname{arccos}(-\frac{3h^2+1}{4h^2})$, $\Gamma$ preserves a totally geodesic ${\bf H}^2_{\mathbb C} \hookrightarrow{\bf H}^3_{\mathbb C}$ invariant, so we will view the representation as degenerating to a representation into  $\mathbf{PU}(2,1)$.
When $h \in [1, \infty)$ is fixed and $t=0$, $\Gamma$ preserves a totally geodesic ${\bf H}^3_{\mathbb R} \hookrightarrow{\bf H}^3_{\mathbb C}$ invariant,	we have a 3-dimensional real hyperbolic Coxeter tetrahedron, so we have a discrete and faithful representation of the  group
$G$ into $\mathbf{PO}(3,1)$. For this fixed $h$, when we increasing $t$, we still have a representation of $G$ in   $\mathbf{PU}(3,1)$, but the discreteness of the  representation is highly non-trivial, this is what we will do in the paper for some pairs of  $(h,t)$.

\begin{figure}
	\begin{center}
		\begin{tikzpicture}
		\node at (0,0) {\includegraphics[width=12cm,height=6cm]{{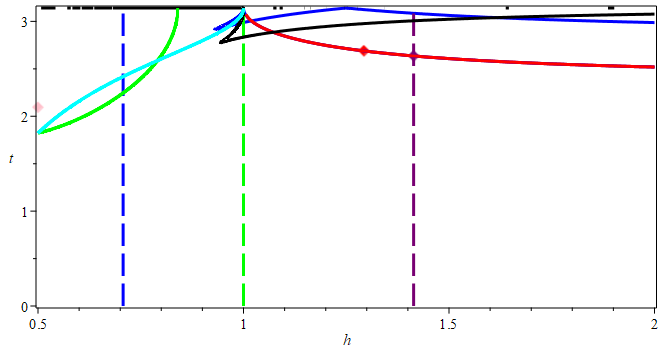}}};
		\end{tikzpicture}
	\end{center}
	\caption{The moduli space $\mathcal{M}$ is the region below the red curve: the red curve is where  $t= \operatorname{arccos}(-\frac{3h^2+1}{4h^2})$; the blue  curve is the locus where $\mathcal{H}(I_1I_2I_3I_4)=0$; the cyan  curve is the locus where $\mathcal{H}(I_1I_4I_1I_2I_1I_4I_3)=0$; 
		the black  curve is the locus where $\mathcal{H}(I_1I_3I_4I_1I_2)=0$; 
		the green  curve is the locus where $\mathcal{H}(I_1I_3I_2I_4I_1I_3I_4)=0$.   The red diamond marked  point is $(h_1,t_1)$ in the red curve, where $h_1$ is $1.29326...$ numerically. At this point the isometric spheres $I(B)$ and   $I(C)$ are tangent; The intersection of the red curve and the dashed purple vertical line is $(\sqrt{2},\operatorname{arccos}(-\frac{7}{8}))$, we draw it in purple diamond; We study the discreteness and faithfulness of $\rho$ in a small neighborhood of this point. 
		 The pink diamond marked point is $(\frac{1}{2}, \frac{2\pi}{3})$, so $(I_2I_3)^3=(I_3I_4)^3=id$, but it also has an accidental-elliptic element, say  $I_1I_4I_1I_2I_1I_4I_3$
		has order 6.}
	\label{figure:moduli}
\end{figure}

%\begin{figure}
%	\begin{center}
%		\begin{tikzpicture}
%		\node at (0,0) {\includegraphics[width=12cm,height=6cm]{{moduli.png}}};
%		\end{tikzpicture}
%	\end{center}
%	\caption{The moduli space $\mathcal{M}$ is the region below the red curve: the red curve is where  $t= \operatorname{arccos}(-\frac{3h^2+1}{4h^2})$; the blue  curve is the locus where $\mathcal{H}(I_1I_2I_3I_4)=0$;  the red diamond marked  point is $(h_1,t_1)$ in the red curve, where $h_1$ is $1.29326...$ numerically. At this point the isometric sphere $I(B)$ and the isometric sphere  $I(C)$ are tangent. The intersection of the red curve and the dashed purple vertical line is $(\sqrt{2},\operatorname{arccos}(-\frac{7}{8}))$, we draw it in purple;  
%		The black  curve is the locus where $\mathcal{H}(I_1I_3I_4I_1I_2)=0$; 
%		The green  curve is the locus where $\mathcal{H}(I_1I_3I_2I_4I_1I_3I_4)=0$; 
%			The pink diamond marked point is $(\frac{1}{2}, \frac{2\pi}{3})$, so %$(I_2I_3)^3=(I_3I_4)^3=id$, but it also has an accidental-elliptic element, say  $I_1I_2I_3I_4I_3I_2I_3$
%		has order 6.}
%	\label{figure:moduli}
%\end{figure}

See Figure \ref{figure:moduli} for the moduli space  $\mathcal{M}$, the moduli space is the region bellow the red curve. In the moduli space $\mathcal{M}$ the function $\mathcal{H}(I_1I_2I_3I_4)$ is complicated.  We draw the locus where $\mathcal{H}(I_1I_2I_3I_4)$ is zero in blue curve in  Figure \ref{figure:moduli} (so the blue curve above the red curve is not meaningful, since the  signature of the Gram matrix is $(2,2)$ but not $(3,1)$).  There are also several marked  points in Figure \ref{figure:moduli}, we will explain them later. 
Moreover, when $h \in (0,1)$, we are only interested that $h = \cos(\frac{\pi}{m})$ for some $m \in \mathbb{Z}_{\geq 3}$, that is when $\mathcal{P}_{2}$ and $\mathcal{P}_{3}$ ($\mathcal{P}_{3}$ and $\mathcal{P}_{4}$ respectively) intersect at an angle $\frac{\pi}{m}$. 
	
	% The steelblue point is $(\sqrt{2},\frac{5 \pi}{6})$,  which is in fact under the red curve.  The purple and the steelblue point are so closed that  we almost can not see the purple one in the figure.

%	\begin{remark} The function $\mathcal{H}(I_1I_2I_3I_4)$ is complicated. In  Figure \ref{figure:moduli},  we draw the locus where the function $\mathcal{H}(I_1I_2I_3I_4)$ is zero, which is the blue curve in $\mathcal{M}$ (so the blue curve above the red curve is not meaningful, since the  signature of $G$  is $(2,2)$ but not $(3,1)$). 
%	\end{remark}

			\subsection{Matrices  representations of the group $G$ into $\mathbf{PU}(3,1)$ } \label{subsec:matricesin3dim}
	For each pair $(h,t)\in \mathcal{M}$, we now give the matrix presentation of $I_{i}$, the order-two complex reflection with mirror  $\mathcal{P}_{i}$. We first take
		\begin{equation}\label{eq:2p2q12}
		{n_1}=\left[
		\begin{array}{c}
		0 \\
		1 \\
			0 \\
		0 \\
		\end{array}
		\right],\quad
		{n_2}=\left[
		\begin{array}{c}
		1\\
			-1\\
			0 \\
		0\\
		\end{array}
		\right].
		\end{equation}		
		Then the Hermitian product of $n_{i}$  and $n_{j}$ for $i=1,2$ and $j=3,4$  determine the second and the fourth  entries of $n_{3}$ and $n_{4}$,  the norms of $n_{3}$ and $n_{4}$  determine  the absolute  values of the third entries of $n_{3}$ and $n_{4}$,  and we can normalize them as
		\begin{equation}\label{eq:2p2q34}
		{n_3}=\left[
		\begin{array}{c}
		-\dfrac{1}{2h}\\ [ 6 pt]
		0
		\\
		0
	 \\-h \\
		\end{array}
		\right],\quad
		{n_4}=\left[
		\begin{array}{c}
		\dfrac{4h^2+\rm{e}^{-t \rm{i}}}{4h^2}\\[ 8 pt]
		-\dfrac{\rm{e}^{-t \rm{i}}}{2}\\  [ 8 pt]
	\dfrac{\sqrt{4h^2 \cos(t)+3h^2+1}}{2h}\\   [ 8 pt]
		-\dfrac{\rm{e}^{-t \rm{i}}}{2}
		\\
		\end{array}
		\right]
		\end{equation}
	By direct  calculation, we have  $(\langle n_{i}, n_{j} \rangle )_{1 \leq i, j \leq 4}=\mathscr{G}$.

We denote by  $I_{i}$ the order-two complex reflection about the $\mathbb {C}$-plane $\mathcal{P}_{i}$ with polar vector $n_{i}$ for $i=1, 2, 3, 4$.
We have
\begin{equation}\nonumber
I_1=\left(\begin{matrix}
-1&0 &0&0\\
0&1 &0 &0\\
0&0&-1 &0\\
0&0&0 &-1\\ \end{matrix}\right),  \quad    \quad
I_2=\left(\begin{matrix}
-1 & -2 & 0&2\\
 0& 1 & 0&-2\\
0&0 & -1&0\\
0&0&0&-1\\ \end{matrix}\right),
\end{equation}
$$I_{3}=\left(\begin{matrix}
0 & 0 & 0&\frac{1}{2h^2}\\
0 & -1&0&0 \\
0&0 & -1&0\\
2h^2&0 &0&0\\ \end{matrix}\right),$$
 $$I_{4}=\left(\begin{array}{cccc}
 -\dfrac{4h^2+4h^2\rm{e}^{t \rm{i}}+1}{4h^2}& -\dfrac{4h^2\rm{e}^{t \rm{i}}+1}{4h^2} & -\dfrac{(4h^2+\rm{e}^{-t \rm{i}} )D}{4h^3}&\dfrac{16h^4+8h^2 \cos(t)+1}{8h^4}\\ [ 8 pt]
 
 \dfrac{1}{2} & -\frac{1}{2}& -\dfrac{\rm{e}^{-t \rm{i}} D}{2h} &-\dfrac{4h^2\rm{e}^{-t \rm{i}}+1}{4h^2}\\[ 8 pt]
 
 -\dfrac{\rm{e}^{t \rm{i}}D}{2h}&-\dfrac{\rm{e}^{t \rm{i}} D}{2h} & \dfrac{4h^2 \cos(t)+h^2+1}{2h^2}& \dfrac{(4h^2+\rm{e}^{t \rm{i}}) D}{4h^3}\\[ 8 pt]
 \dfrac{1}{2}&\frac{1}{2} &-\dfrac{\rm{e}^{-t \rm{i}} D}{2h} &-\dfrac{4h^2+4h^2\rm{e}^{-t \rm{i}}+1}{4h^2}\\ \end{array}\right),$$
where $D=\sqrt{4h^2 \cos(t)+3h^2+1}$.
By direct calculation we have $$\det(I_1)=\det(I_2)=\det(I_3)=\det(I_4)=-1,$$  and $$I^*_{i} H I_{i}=H$$ for $i=1,2,3,4$, here  $I^*_{i}$ is the Hermitian transpose of $I_{i}$. Moreover  $$(I_1 I_3)^{2}=(I_1 I_4)^{3}=(I_2 I_4)^{2}=id,$$ and $I_1I_2$ is parabolic, so we get a representation $\rho$ of  $G$ into $\mathbf{U}(3,1)$ (with additional condition that $I_1I_2$ is parabolic).

\begin{remark}	Note that since $\det(I_i)=-1$, so $I_i \in \mathbf{U}(3,1)$, but $I_i \notin \mathbf{PU}(3,1)$, and  ${\rm e}^{-\frac{\pi\rm{i}}{4}}\cdot I_i \in \mathbf{PU}(3,1)$. But we can also use $I_i$ to study  the geometry of the groups  for simplicity of notations.  But if we consider the holy grail function of a word on them, we must use the matrix ${\rm e}^{-\frac{\pi\rm{i}}{4}}\cdot I_i$ since it has determinant 1. 
	 
	\end{remark}

%and $I_4$ is
%%%$$\left(\begin{matrix}
%%-\frac{4h^2+4h^2\rm{e}^{t \rm{i}}+1}{4h^2}& -\frac{4h^2\rm{e}^{t %\rm{i}}+1}{4h^2} & -\frac{(4h^2+\rm{e}^{-t \rm{i}} )\sqrt{4h^2 %\cos(t)+3h^2+1}}{4h^3}&\frac{16h^4+8h^2 \cos(t)+1}{8h^4}\\
%\frac{1}{2} & -\frac{1}{2}& -\frac{\rm{e}^{-t \rm{i}} \sqrt{4h^2 %%\cos(t)+3h^2+1}}{2h} &-\frac{4h^2\rm{e}^{-t \rm{i}}+1}{4h^2}\\
%-\frac{\rm{e}^{t \rm{i}} \sqrt{4h^2 \cos(t)+3h^2+1}}{2h}&-\frac{\rm{e}^{t %\rm{i}} \sqrt{4h^2 \cos(t)+3h^2+1}}{2h} & \frac{4h^2 \cos(t)+h^2+1}{2h^2}& %\frac{(4h^2+\rm{e}^{t \rm{i}}) \sqrt{4h^2 \cos(t)+3h^2+1}}{4h^3}\\
%\frac{1}{2}&\frac{1}{2} &-\frac{\rm{e}^{-t \rm{i}} \sqrt{4h^2 %%\cos(t)+3h^2+1}}{2h} &-\frac{4h^2+4h^2\rm{e}^{-t \rm{i}}+1}{4h^2}\\ %\end{matrix}\right),$$

 In the following, we take $$A=I_1I_2,~~~~ B=I_3I_1,~~~~C=I_4I_1,$$
then $B^2=id$,  $C^3=id$, $(AC)^2=id$, and $A$ is parabolic.  $\langle A, B, C \rangle$ is an index two subgroup of $\Gamma=\rho_{(h,t)}(G)$.  We have
			\begin{equation}\nonumber
			A=\left(\begin{matrix}
			1&2 &0&-2\\
			0&1 &0 &-2\\
			0&0&1 &0\\
			0&0&0 &1\\ \end{matrix}\right),  \quad    \quad
			B=\left(\begin{matrix}
			0 & 0 & 0&-\frac{1}{2h^2}\\
			0& -1 & 0&0\\
			0&0 & 1&0\\
			-2h^2&0&0&0\\ \end{matrix}\right),
			\end{equation}
					and	$C$ is 
			$$\left(\begin{array}{cccc}
			\dfrac{4h^2+4h^2\rm{e}^{t \rm{i}}+1}{4h^2}& -\dfrac{4h^2\rm{e}^{t \rm{i}}+1}{4h^2} & -\dfrac{(4h^2+\rm{e}^{-t \rm{i}} )D}{4h^3}&-\dfrac{16h^4+8h^2 \cos(t)+1}{8h^4}\\[ 10 pt]
			-\dfrac{1}{2} & -\frac{1}{2}& \dfrac{\rm{e}^{-t \rm{i}} D}{2h} &\dfrac{4h^2\rm{e}^{-t \rm{i}}+1}{4h^2}\\[ 10 pt]
			\dfrac{\rm{e}^{t \rm{i}} D}{2h}&-\dfrac{\rm{e}^{t \rm{i}} D}{2h} & -\dfrac{4h^2 \cos(t)+h^2+1}{2h^2}& -\dfrac{(4h^2+\rm{e}^{t \rm{i}}) D}{4h^3}\\[ 10 pt]
			-\dfrac{1}{2}&\frac{1}{2} &\dfrac{\rm{e}^{-t \rm{i}} D}{2h} &\dfrac{4h^2+4h^2\rm{e}^{-t \rm{i}}+1}{4h^2}\\ \end{array}\right).$$
			
	%		 $C$ is
	%		$$\left(\begin{matrix}
	%		\frac{4h^2+4h^2\rm{e}^{t \rm{i}}+1}{4h^2}& -\frac{4h^2\rm{e}^{t \rm{i}}+1}{4h^2} & -\frac{(4h^2+\rm{e}^{-t \rm{i}} )\sqrt{4h^2 \cos(t)+3h^2+1}}{4h^3}&-\frac{16h^4+8h^2 \cos(t)+1}{8h^4}\\
	%		-\frac{1}{2} & -\frac{1}{2}& \frac{\rm{e}^{-t \rm{i}} \sqrt{4h^2 \cos(t)+3h^2+1}}{2h} &\frac{4h^2\rm{e}^{-t \rm{i}}+1}{4h^2}\\
	%		\frac{\rm{e}^{t \rm{i}} \sqrt{4h^2 \cos(t)+3h^2+1}}{2h}&-\frac{\rm{e}^{t \rm{i}} \sqrt{4h^2 \cos(t)+3h^2+1}}{2h} & -\frac{4h^2 \cos(t)+h^2+1}{2h^2}& -\frac{(4h^2+\rm{e}^{t \rm{i}}) \sqrt{4h^2 \cos(t)+3h^2+1}}{4h^3}\\
	%		-\frac{1}{2}&\frac{1}{2} &\frac{\rm{e}^{-t \rm{i}} \sqrt{4h^2 %\cos(t)+3h^2+1}}{2h} &\frac{4h^2+4h^2\rm{e}^{-t \rm{i}}+1}{4h^2}\\ %\end{matrix}\right).$$

		\subsection{Matrices in $\mathbf{PU}(2,1)$}\label{subsec:matricesin2dim}
	When $\cos(t)=-\frac{3h^2+1}{4h^2}$, the third entry  of each $n_{i}$
	is zero for $i=1, 2, 3, 4$. So take \begin{equation}\label{eq:2p2q12}
	{n}=\left[
	\begin{array}{c}
	0 \\
	0 \\
	1 \\
	0 \\
	\end{array}
	\right],
	\end{equation} in  $\mathbb C{\bf P}^3-\overline{{\bf H}^3_{\mathbb C}}$. The intersection of  ${\bf H}^3_{\mathbb C}$  and  the dual of $n$ in  $\mathbb C{\bf P}^3$  is a copy of  ${\bf H}^2_{\mathbb C}$,  each $I_{i}$ preserves this ${\bf H}^2_{\mathbb C}$ invariant.
	So we delete the third column and the third row of the matrix  $I_{i}$ in Subsection  \ref{subsec:matricesin3dim}, we get new matrices in $\mathbf{PU}(2,1)$, but we still denote them by $I_{i}$ for the simplification of notations. 
	We have
	\begin{equation}\nonumber
	I_1=\left(\begin{matrix}
	-1&0 &0\\
	0&1  &0\\
	
	0&0 &-1\\ \end{matrix}\right),  \quad    \quad
	I_2=\left(\begin{matrix}
	-1 & -2 &2\\
	0& 1 &-2\\
	
	0&0&-1\\ \end{matrix}\right),
	\end{equation}

	$$I_{3}=\left(\begin{matrix}
	0 & 0 &\frac{1}{2h^2}\\
	0 & -1&0 \\
	
	2h^2&0 &0\\ \end{matrix}\right),$$
and $I_4$ is 
$$\left(\begin{array}{ccc}
-\dfrac{h^2+\rm{i}\sqrt{7h^4-6h^2-1}}{4h^2}& \dfrac{3h^2-\rm{i}\sqrt{7h^4-6h^2-1}}{4h^2} &\dfrac{16h^4-6h^2-1}{8h^4}\\[ 10 pt]
\dfrac{1}{2} & -\dfrac{1}{2} &\dfrac{3h^2+\rm{i}\sqrt{7h^4-6h^2-1}}{4h^2}\\[ 10 pt]
\dfrac{1}{2}&\dfrac{1}{2}  &\dfrac{-h^2+\rm{i}\sqrt{7h^4-6h^2-1}}{4h^2}\\ \end{array}\right).$$

\section{Ford domain  of $\rho_{(\sqrt{2},\operatorname{arccos}(-\frac{7}{8}))}(K) <\mathbf{PU}(2,1)$} \label{sec:Ford2dim}

In this section, we prove the first part of  Theorem \ref{thm:3-mfd}. The  proof in this section is also a  model for the proof of Theorem \ref{thm:complex3dim}.

When  $h= \sqrt{2}$, and $t=\operatorname{arccos}(-\frac{7}{8})$,  from  Subsection \ref{subsec:matricesin2dim} we have
$$I_{3}=\left(\begin{matrix}
0 & 0 &\frac{1}{4}\\
0 & -1&0 \\

4&0 &0\\ \end{matrix}\right),$$
and 
$$I_4=\left(\begin{array}{ccc}
-\frac{2+\rm{i}\sqrt{15}}{8}& \frac{6-\rm{i}\sqrt{15}}{8} &\frac{51}{32}\\[ 8 pt]
\frac{1}{2} & -\frac{1}{2} & \frac{6+\rm{i}\sqrt{15}}{8}\\[ 8 pt]
\frac{1}{2}&\frac{1}{2}  &\frac{-2+\rm{i}\sqrt{15}}{8}\\ \end{array}\right).$$
In this section, we let $\Sigma=\rho_{(\sqrt{2},\operatorname{arccos}(-\frac{7}{8}))}(K)$.

\subsection{Outline of the Ford domain of $\Sigma=\rho_{(\sqrt{2},\operatorname{arccos}(-\frac{7}{8}))}(K)$ in ${\bf H}^2_{\mathbb C}$}\label{subsec:ford3mfd}

We will study the local combinatorial structure of the  Ford domain  of  $\Sigma=\langle A, B, C\rangle= \rho_{(\sqrt{2},\operatorname{arccos}(-\frac{7}{8}))}(K)< \rho_{(\sqrt{2},\operatorname{arccos}(-\frac{7}{8}))}(G)$ in this section.
Where
\begin{equation}\nonumber
A=I_1I_2=\left(\begin{matrix}
1&2 &-2\\
0&1  &-2\\

0&0 &1\\ \end{matrix}\right),  \quad    \quad
B=I_3I_1=\left(\begin{matrix}
0 & 0 &-\frac{1}{4}\\
0& -1 &0\\
-4&0&0\\ \end{matrix}\right),
\end{equation}
and 
$$C=I_4I_1=\left(\begin{array}{ccc}
\frac{2+\rm{i}\sqrt{15}}{8}&\frac{6-\rm{i}\sqrt{15}}{8} &-\frac{51}{32}\\ [6  pt]
-\frac{1}{2} & -\frac{1}{2} &-\frac{6+\rm{i}\sqrt{15}}{8}\\[ 6 pt]
-\frac{1}{2}&\frac{1}{2}&\frac{2-\rm{i}\sqrt{15}}{8}\\ \end{array}\right).$$

The action of $A$ on the Heisenberg group $\mathbb{C} \times \mathbb{R}$ is given by
\begin{equation}\label{eq-A}
(z,t)\rightarrow (z-2, t+4 \cdot  Im(z)).
\end{equation}
Let $R \subset \rho_{(\sqrt{2},\operatorname{arccos}(-\frac{7}{8}))}(K)$  be the set of words 
$$\left\{A^{k}CA^{-k}, A^{k}C^{-1}BCA^{-k}, A^{k}CBC^{-1}A^{-k}, A^{k}CBCA^{-k}, A^{k}C^{-1}BC^{-1}A^{-k}  \right\}_{k \in \mathbb{Z}}.$$
 We will show the boundary of the Ford domain of $\Sigma$ consists of part of the isometric sphere of $I(g)$ for all $g \in R$. 
Note that since $(AC)^2=id$, the isometric sphere $I(C^{-1})$ identifies with the isometric sphere $I(A^{-1}CA)$. Since $A$ is unipotent with fixed point $q_{\infty}=(1,0,0)^{T} \in \partial {\bf H}^2_{\mathbb C}$, it is a Cygan isometry, and thus $A$-action preserves the radii of isometric spheres.

%$$\left\{A^{k} \cdot \{C, CBC, C^{-1}BC, CBC^{-1}, C^{-1}BC^{-1} \}  \cdot %A^{-k} \right\}_{k \in \mathbb{Z}}.$$

% Moreover, it follows from
%Equation (\ref{eq-A}) that $A^{k}$ acts on $\partial{\bf H}^2_{\mathbb C}$ by left Heisenberg multiplication of $[-2k,8k/\sqrt{3}???]$.

\begin{defn}
	
	For $\Sigma=\rho_{(\sqrt{2},\operatorname{arccos}(-\frac{7}{8}))}(K) < \mathbf{PU}(2,1)$,  the partial Ford domain $D_{R}$ of $\Sigma$  in $ {\bf H}^2_{\mathbb C}$
	is the intersection of the exteriors of
	all  isometric spheres of elements in	$R$,
	that is $$D_{R}=\{p \in  \overline{ {\bf H}^2_{\mathbb{C}} } : |\langle p, q_{\infty} \rangle | \leq  |\langle p, g^{-1}(q_{\infty})|\}_{g \in R}.$$
\end{defn}
We will show $D_{R}$ is in fact the Ford domain of $\Sigma$, the main tool for our study is the Poincar\'e polyhedron theorem, which gives
sufficient condition for $D_{R}$
to be a fundamental domain for the group. We shall use a version of the Poincar\'e polyhedron theorem for coset
decompositions rather than for groups, because $D_{R}$
is stabilized by the cyclic
subgroup generated by $A$. We refer to \cite{ParkerWill:2017} for the precise statement of this
version of the Poincar\'e polyhedron theorem we need.
The  main technical result in this section is

\begin{thm} \label{thm:2dimford}
	$D_{R}$ is a fundamental domain for the cosets of $\langle A \rangle $ in $\Sigma$. Moreover,
	the group $\Sigma = \langle A, B,C \rangle$ is discrete and has a presentation  $$\langle A,B,C : B^2=C^3=(AC)^2=id\rangle.$$
\end{thm}

\subsection{Intersection patterns of the isometric spheres for $D_{R}$}\label{subsec:intersection}
In this subsection, we will study the rough information on intersection patterns of the isometric spheres for $D_{R}$.  We summarize these patterns in Table \ref{table:intersecton} and we will show this carefully. Moreover,  Table \ref{table:intersecton} should be compared with Figure 	\ref{figure:pointw} in Subsection \ref{subsec:poincare2dim} and  Figure \ref{figure:abstract} in Section \ref{sec:3mfd}.

\begin{table}[htbp]
	\caption{The intersections of isometric spheres we  should be   concerned about (up to $A$-action).}
	\centering
	\begin{tabular}{c}
		\hline
	%	Intersection   \\  [2 ex]
		%heading
	%	\hline
		 $I(C)\cap I(A^{-1}CA)$     \\  [2 ex]
		$I(C)\cap I(CBC^{-1})$  
		\\ [2 ex]
		$I(C)\cap I(C^{-1}BC^{-1})$ \\  [2 ex]
		$I(C)\cap I(AC^{-1}BCA^{-1})$     
		 \\  [2 ex]
		$I(C)\cap I(ACBCA^{-1})$   
		\\  [2 ex]
		$I(CBC)\cap I(C^{-1}BC)$    
		\\  [2 ex]
		$I(CBC^{-1})\cap I(C^{-1}BC^{-1})$    
		\\ 
		\hline
	\end{tabular}
	\label{table:intersecton}
\end{table}

%\begin{table}[htbp]
%	\caption{The intersections of isometric spheres we  should be   concerned with.}
%	\centering
%	\begin{tabular}{c c | c  c}
%		\hline
%		Intersection & The value of $k$ & Intersection & The value of $k$ \\ [1 ex]
		%heading
%		\hline
%		$I(C)\cap I(A^{k}CA^{-1})$  & $k=\pm 1$ &$\mathcal{I}_{B}\cap \mathcal{I}_{b}^{k}$ & $k=0, \pm1$    \\ [2 ex]
%		$I(C)\cap I(A^{k}CBC^{-1}A^{-1})$ &  $k=0$ &$\mathcal{I}_{B}\cap \mathcal{I}_{bAB}^{k}$ & $k=0,\pm1$   
%		 \\ [2 ex]
%		$I(C)\cap I(A^{k}C^{-1}BC^{-1}A^{-1})$ &  $k=0$ &$\mathcal{I}_{B}\cap \mathcal{I}_{BAB}^{k}$ & $k=0,\pm1$    \\ [2 ex]
%		$I(C)\cap I(A^{k}C^{-1}BC^{-1}A)$ &  $k=1 $ &$\mathcal{I}_{B}\cap %\mathcal{I}_{Bab}^{k}$ & $k=0,\pm1$    \\ [2 ex]
%		\hline
%	\end{tabular}
%	\label{table:intersecton}
%\end{table}

 First, from the $A$-action on the Heisenberg group and matrix presentations of $B$ and $C$, it is easy to have the follow three propositions.
 
\begin{prop}\label{prop:center-radiusC}
	For any integer $k\in \mathbb{Z}$, the isometric sphere  $I(A^{k}CA^{-k})$ has radius $2$ and is  centered at  $$[-2k-1,-\frac{\sqrt{15}}{2}] \in \mathbb{C} \times \mathbb{R}$$
	in the Heisenberg group.
\end{prop}

\begin{prop}\label{prop:center-radiuscBCCBc}
	For any integer $k\in \mathbb{Z}$, the isometric spheres $I(A^{k}C^{-1}BCA^{-k})$ and  $I(A^{k}CBC^{-1}A^{-k})$ have the same radius  $\frac{2}{\sqrt{3}}$, they are centered at $$[-2k+ \sqrt{\frac{5}{3}}\rm{i},\frac{-7+8k}{2}\sqrt{\frac{5}{3}}]\in \mathbb{C} \times \mathbb{R}$$ and  $$[-2k- \sqrt{\frac{5}{3}}\rm{i},\frac{-7-8k}{2}\sqrt{\frac{5}{3}}]\in \mathbb{C} \times \mathbb{R}$$
		in the Heisenberg group  respectively.
\end{prop}

\begin{prop}\label{prop:center-radiusCBCcBc}
	For any integer $k\in \mathbb{Z}$, the isometric spheres $I(A^{k}CBCA^{-k})$ and  $I(A^{k}C^{-1}BC^{-1}A^{-k})$ have the same radius $\sqrt{2}$, they are centered at  $$[-\frac{1+4k}{2}+\frac{\sqrt{15}}{2} \rm{i},\frac{4k-3}{2}\sqrt{15}]\in \mathbb{C} \times \mathbb{R}$$ and $$[\frac{1-4k}{2}-\frac{\sqrt{15}}{2} \rm{i},\frac{-4k-3}{2}\sqrt{15}]\in \mathbb{C} \times \mathbb{R}$$
		in the Heisenberg group respectively.
\end{prop}

Secondly, we consider the intersections of  $I(C)$ with other isometric spheres. 

\begin{prop}\label{prop:C} For the isometric sphere $I(C)$ of $C$, we have
	
	\begin{enumerate}
		\item  \label{item:CandClarge}
   $I(C)$  does not intersect the  isometric sphere  $I(A^{k}CA^{-k})$ for $|k| > 2$;
  	\item \label{item:CandCsmall} $I(C)$ is tangent to the  isometric sphere  $I(A^{k}CA^{-k})$ for $k= \pm 2$. Moreover, these two tangent points lie in the interiors  of the isometric spheres  $I(ACA^{-1})$ and $I(A^{-1}CA)$ respectively;

  	\item \label{item:CandCBC}  $I(C)$ is disjoint from $I(A^{k}CBCA^{-k})$  for all  $k \in \mathbb{Z}$ except when $k=0,1$. Moreover, $I(C) \cap I(CBC)$ is in the interior of $I(A^{-1}CA)$;
  	
  	\item  \label{item:CandinverseCBinverseC} $I(C)$ is disjoint from $I(A^{k}C^{-1}BC^{-1}A^{-k})$  for all $k \in \mathbb{Z}$ except $k=0,1$. Moreover, $I(C) \cap I(AC^{-1}BC^{-1}A^{-1})$ is in the interior of $I(ACA^{-1})$;
  	
  	 	\item \label{item:CandCBinverseC}
  	 	$I(C)$ is disjoint from $I(A^{k}CBC^{-1}A^{-k})$  for all  $k \in \mathbb{Z}$ except when $k=0,1$; Moreover, $I(C) \cap I(ACBC^{-1}A^{-1})$ is in the interior of $I(ACA^{-1})$;

  	 		\item \label{item:CandinverseCBC} $I(C)$ is disjoint from $I(A^{k}C^{-1}BCA^{-k})$  for all  $k \in \mathbb{Z}$ except when $k=0,1$. Moreover, $I(C) \cap I(C^{-1}BC)$ is in the interior of $I(A^{-1}CA)$.
  	
  	 \end{enumerate}
		\end{prop}

\begin{proof}By Proposition \ref{prop:center-radiusC},  the Cygan  distance between the centers of $I(C)$ and $I(A^{k}CA^{-k})$  is $|2k|$.  The radius of these isometric spheres are $2$, so when $|k| > 2$, the isometric sphere  $I(C)$ does not intersect the  isometric sphere  $I(A^{k}CA^{-k})$. This ends the proof of (\ref{item:CandClarge}).

	  The defining equation of the isometric sphere   $I(C)$ is
	\begin{equation}\label{eq:cygan-sphereC}
	\left\{ (z,t,u)\in\hc \cup \partial\hc: ||z+1|^2+u+\rm{i}(t+\frac{\sqrt{15}}{2}-2 \cdot Im(z)|=4 \right\}.
	\end{equation}
The defining equation of the isometric sphere 	$I(A^{2}CA^{-2})$
		\begin{equation}\label{eq:cygan-sphereA2Ca2}
	\left\{ (z,t,u)\in\hc \cup \partial\hc: ||z+5|^2+u+\rm{i}(t+\frac{\sqrt{15}}{2}-10 \cdot Im(z)|=4 \right\}.
	\end{equation}
Then it is easy to see the point $P$ with horospherical coordinates $$(z,t,u)=(-3,-\frac{\sqrt{15}}{2},0)$$ is the only point  which  satisfies both  	 (\ref{eq:cygan-sphereC})  and (\ref{eq:cygan-sphereA2Ca2}). The point $P$ is the center of the isometric sphere $I(ACA^{-1})$, so it lies in the interior of the isometric sphere $I(ACA^{-1})$. The tangent point between  the isometric spheres $I(C)$ and  $I(A^{-2}CA^{2})$ has similar behavior. This ends the proof of (\ref{item:CandCsmall}).

From Propositions \ref{prop:center-radiusC} and  \ref{prop:center-radiusCBCcBc}, we know the Cygan distance between the center of the isometric spheres $I(C)$ and $I(A^{k}CBCA^{-k})$ is
$$\left| |\frac{4k-1}{2}-\frac{\sqrt{15}}{2} \rm{i}|^2+(2-2k)\sqrt{15}\rm{i}\right|^{\frac{1}{2}},$$
it is $$(16k^4-16k^3+96k^2-136k+76)^{\frac{1}{4}}.$$
Which is bigger than $2+\sqrt{2}$ for any  $k \in \mathbb{Z}$ except $k=0,1$.

Moreover, consider the intersection of $I(C)$ and  $I(CBC)$. 
In Equation (\ref{equaation:girauddisk}), we take ${\bf q}=C^{-1}(q_{\infty})$,  ${\bf r}=(CBC)^{-1}(q_{\infty})$ and  ${\bf p}=q_{\infty}$, then we can parameterize the intersection of the isometric spheres  $I(C)$ and $I(CBC)$  by $V=V(z_1,z_2)$ with $\langle V,V \rangle <0$.
Where
$$V=\left(\begin{array}{c}
\frac{1}{16}+\frac{7\rm{i}\sqrt{15}}{16}-\frac{\rm{e}^{r \rm{i}}(1+\rm{i}\sqrt{15})}{2}+\frac{\rm{e}^{s \rm{i}}}{2}
\\[ 6 pt]
-\frac{3}{4}+\frac{\rm{i}\sqrt{15}}{4}-\rm{e}^{r \rm{i}}+\frac{\rm{e}^{s \rm{i}}}{2}\\[ 6 pt]
-\frac{1}{4}+\frac{\rm{i}\sqrt{15}}{4}\\ \end{array}\right),$$
with $(z_1,z_2)=(\rm{e}^{r \rm{i}},\rm{e}^{s \rm{i}}) \in S^{1}\times S^1$.

We have $\langle V, V\rangle$  is
$$\frac{(-2\sin(r)+\sin(s))\sqrt{15}}{2}-(\cos(s)+2)\cos(r)-\sin(s)\sin(r)-\cos(s)+6.$$
Now $$\langle V, q_{\infty}\rangle \langle  q_{\infty}, V\rangle=1,$$
and $$\langle V, A^{-1}C^{-1}A(q_{\infty})\rangle \langle A^{-1}C^{-1}A(q_{\infty}), V\rangle $$ is
\begin{flalign}
\nonumber &
\frac{(\sin(r-s)-3\sin(r)+\sin(s))\sqrt{15}}{4}-\frac{3\cos(r)}{4}-\frac{3\cos(s)}{4}-\frac{3\cos(r+s)}{4}+\frac{13}{4}.
&
\end{flalign}
Using Maple, the maximum of $|\langle V, A^{-1}C^{-1}A(q_{\infty})\rangle|^2 $ is 0.7370031 numerically with the condition $\langle V ,V \rangle <0$. Which is smaller than $|\langle V,q_{\infty}\rangle|^2=1$. So the triple intersection $$I(C) \cap I(CBC) \cap I(A^{-1}CA)$$ is empty. Moreover, take a sample point $(r,s)=(\frac{\pi}{3}, \frac{15 \pi}{8})$, then it is easy to see $V(\rm{e}^{\frac{\pi \rm{i}}{3}},\rm{e}^{\frac{15\pi \rm{i}}{8}}) \in {\bf H}^2_{\mathbb C}$ and $V(\rm{e}^{\frac{\pi \rm{i}}{3}},\rm{e}^{\frac{15\pi \rm{i}}{8}})$ lies in the interior of $I(A^{-1}CA)$. So the Giraud disk $I(C) \cap I(CBC)$ lies entirely  in the interior of the isometric sphere $I(A^{-1}CA)$.  This proves (\ref{item:CandCBC}).

The proof  of (\ref{item:CandinverseCBinverseC}) is very similar to the proof of
(\ref{item:CandCBC}), even with the same functions at later steps,  we omit the details.

From Propositions  \ref{prop:center-radiusC} and \ref{prop:center-radiuscBCCBc}, we know the Cygan distance between the centers of the isometric spheres $I(C)$ and $I(A^{k}CBC^{-1}A^{-k})$ is
$$\left| |2k-1+\sqrt{\frac{5}{3}} \rm{i}|^2+\rm{i} \frac{4k\sqrt{15}}{3}\right|^{\frac{1}{2}},$$
it is $$(\frac{64}{9}-\frac{64k}{3}+64k^2-32k^3+16k^4)^{\frac{1}{4}}.$$
Which is bigger than $2+\frac{2}{\sqrt{3}}$ for any  $k \in \mathbb{Z}$ except $k=0,1$.

Moreover, consider the intersection of $I(C)$ and  $I(ACBC^{-1}A^{-1})$.
In Equation (\ref{equaation:girauddisk}), we take ${\bf q}=C^{-1}(q_{\infty})$,  ${\bf r}=ACBC^{-1}A^{-1}(q_{\infty})$ and  ${\bf p}=q_{\infty}$, then we can parameterize  the intersection of the isometric spheres  $I(C)$ and $I(ACBC^{-1}A^{-1})$  by $V=V(z_1,z_2)$ with $\langle V,V \rangle <0$.
Where
$$V=\left(\begin{array}{c}
-\frac{37}{16}+\frac{7\rm{i}\sqrt{15}}{16}+\frac{\rm{e}^{r \rm{i}}(-3+\rm{i}\sqrt{15})}{2}+\frac{\rm{e}^{s \rm{i}}}{2}
\\[ 6 pt]
-\frac{7}{4}+\frac{3\rm{i}\sqrt{15}}{4}-\frac{\rm{e}^{r \rm{i}}}{2}+\frac{\rm{e}^{s \rm{i}}}{2}\\[ 6 pt]
\frac{3}{4}-\frac{\rm{i}\sqrt{15}}{4}\\ \end{array}\right),$$
with $(z_1,z_2)=(\rm{e}^{r \rm{i}},\rm{e}^{s \rm{i}}) \in S^{1}\times S^1$.

We have $\langle V, V\rangle$  is
$$\frac{(-3\sin(r)+\sin(s))\sqrt{15}}{2}-\frac{(3\cos(s)+6)\cos(r)}{2}-\frac{3\sin(s)\sin(r)}{2}-\cos(s)+\frac{29}{4}.$$
Now $$\langle V, q_{\infty}\rangle \langle  q_{\infty}, V\rangle=\frac{3}{2},$$
and $$\langle V, A^{-1}C^{-1}A(q_{\infty})\rangle \langle A^{-1}C^{-1}A(q_{\infty}), V\rangle $$ is
\begin{flalign}
\nonumber &\frac{\sqrt{15}(\sin(s)-\sin(r))}{4}+\frac{\sqrt{15}\sin(r-s)}{4}-\frac{3\cos(r-s)}{4}-\frac{3\cos(r)}{2}-\frac{\cos(s)}{4}+\frac{11}{4}.
&
\end{flalign}
Using Maple, the maximum of $|\langle V, A^{-1}C^{-1}A(q_{\infty})\rangle|^2 $ is 1.30600826 numerically with the condition $\langle V ,V \rangle <0$. Which is smaller than $|\langle V,q_{\infty}\rangle|^2=\frac{3}{2}$. So the triple intersection $$I(C) \cap I(ACBC^{-1}A^{-1}) \cap I(ACA^{-1})$$ is empty. Moreover, take a sample point $(r,s)=(\frac{\pi}{3}, 0)$, then it is easy to see $V(\rm{e}^{\frac{\pi \rm{i}}{3}},\rm{e}^{0 \cdot \rm{i}}) \in {\bf H}^2_{\mathbb C}$ and $V(\rm{e}^{\frac{\pi \rm{i}}{3}},\rm{e}^{0 \cdot \rm{i}})$ lies in the interior of $I(ACA^{-1})$. So the Giraud disk $I(C) \cap I(ACBC^{-1}A^{-1})$ lies entirely  in the interior of the isometric sphere $I(ACA^{-1})$.  This proves (\ref{item:CandCBinverseC}).

The proof of (\ref{item:CandinverseCBC}) is very similar to the proof of (\ref{item:CandCBinverseC}), we omits the details.
\end{proof}

\begin{prop}\label{prop:inverseCBinverseC} For the isometric sphere $I(C^{-1}BC^{-1})$, we have
	
	\begin{enumerate}
		\item  \label{item:inverseCBinverseCandCBeC}
	 $I(C^{-1}BC^{-1})$ does not intersect the  isometric sphere $I(A^{k}CBCA^{-k})$ for any $k \in \mathbb{Z}$;
		\item \label{item:inverseCBinverseCandinverseCBC}
		 $I(C^{-1}BC^{-1})$ does not intersect the  isometric sphere  $I(A^{k}C^{-1}BCA^{-k})$ for any $k \in \mathbb{Z}$;
		
		\item  \label{item:inverseCBinverseCandinverseCBinverseC}
		
	 $I(C^{-1}BC^{-1})$ does not intersect the  isometric sphere  $I(A^{k}C^{-1}BC^{-1}A^{-k})$ for any non-zero $k \in \mathbb{Z}$;
		
		\item  \label{item:inverseCBinverseCandCBinverseC}
		 $I(C^{-1}BC^{-1})$ does not intersect the  isometric sphere  $I(A^{k}CBC^{-1}A^{-k})$ for any non-zero $k \in \mathbb{Z}$.

	\end{enumerate}
\end{prop}
\begin{proof}

From Proposition  \ref{prop:center-radiusCBCcBc}, the Cygan distance between the centers of the isometric spheres  $I(C^{-1}BC^{-1})$ and $I(A^{k}CBCA^{-k})$ is
$$\left| |(2k+1)-\sqrt{15} \rm{i}|^2+\rm{i}\left(-2k\sqrt{15}+2 Im \left(\left(\frac{1}{2}-\frac{\sqrt{15}}{2}\rm{i}\right)\left(-\frac{1+4k}{2}-\frac{\sqrt{15}}{2}\rm{i}\right)\right)\right)\right|^{\frac{1}{2}}.$$
Which is
\begin{equation}\label{eq:cBcAkCBCak}
(4k^2+4k+16)^{\frac{1}{2}}.
\end{equation}
It is now easy to see the term in (\ref{eq:cBcAkCBCak}) is bigger than  $\sqrt{2}+\frac{2}{\sqrt{3}}$ for any $k \in \mathbb{Z}$, so the isometric sphere  $I(C^{-1}BC^{-1})$ does not intersect the  isometric sphere  $I(A^{k}CBCA^{-k})$ for any $k \in \mathbb{Z}$. This proves (\ref{item:inverseCBinverseCandCBeC}).

From Propositions \ref{prop:center-radiusCBCcBc}  and \ref{prop:center-radiuscBCCBc},  the Cygan distance between the centers of the isometric spheres  $I(C^{-1}BC^{-1})$ and $I(A^{k}C^{-1}BCA^{-k})$ is
$$\left| |(2k+\frac{1}{2})-\sqrt{\frac{5}{3}} \rm{i}-\frac{\sqrt{15}}{2} \rm{i}|^2+\rm{i}\left(-\frac{3\sqrt{15}}{2}+\frac{(7-8k)}{2}\sqrt{\frac{5}{3}}+ Im\left(\left(\frac{1}{2}-\frac{\sqrt{15}}{2}\rm{i}\right)\left(-2k-\sqrt{\frac{5}{3}}\rm{i}\right)\right)\right)\right|^{\frac{1}{2}}.$$
Which is
\begin{equation}\label{eq:cBcAkCBcak}
(16k^4+16k^3+\frac{448}{3}k^2-\frac{172k}{3}+\frac{1399}{9})^{\frac{1}{4}}.
\end{equation}
It can be showed that the term in (\ref{eq:cBcAkCBcak})  is bigger than $\sqrt{2}+\frac{2}{\sqrt{3}}$ for any  $k \in \mathbb{Z}$. So  the isometric sphere  $I(C^{-1}BC^{-1})$ does not intersect the  isometric sphere  $I(A^{k}C^{-1}BCA^{-k})$ for any $k \in \mathbb{Z}$. This proves (\ref{item:inverseCBinverseCandinverseCBC}).

%From Propositions \ref{prop:center-radiusCBCcBc}  and \ref{prop:center-radiuscBCCBc},  the Cygan distance between the centers of the isometric spheres  $I(C^{-1}BC^{-1})$ and $I(A^{k}CBC^{-1}A^{-k})$ is
%$$\left| |(2k-\frac{1}{2})-\sqrt{\frac{5}{3}} \rm{i}-\frac{\sqrt{15}}{2} \rm{i}|^2+\rm{i}\left(-\frac{3\sqrt{15}}{2}+\frac{(7-8k)}{2}\sqrt{\frac{5}{3}}+ Im\left(\left(-\frac{1}{2}-\frac{\sqrt{15}}{2}\rm{i}\right)\left(-2k-\sqrt{\frac{5}{3}}\rm{i}\right)\right)\right)\right|^{\frac{1}{2}}.$$
%Which is
%\begin{equation}\label{eq:cBcAkCBcak}
%(16k^4-16k^3+96k^2-\frac{128k}{3}+\frac{1024}{9})^{\frac{1}{4}}.
%\end{equation}
%It can be showed that the term in (\ref{eq:cBcAkCBcak})  is bigger that $\sqrt{2}+\frac{2}{\sqrt{3}}$ for any  $k \in \mathbb{Z}$. So  the isometric sphere  $I(C^{-1}BC^{-1})$ does not intersect the  isometric sphere  $I(A^{k}C^{-1}BCA^{-k})$ for any $k \in \mathbb{Z}$. This proves (\ref{item:inverseCBinverseCandinverseCBC}).

From   Proposition \ref{prop:center-radiusCBCcBc}, the Cygan distance between the center of the isometric spheres  $I(C^{-1}BC^{-1})$ and $I(A^{k}C^{-1}BC^{-1}A^{-k})$ is
$$\left|2k^2+4k\sqrt{15} \rm{i}\right|^{\frac{1}{2}}.$$
It is now easy to see the isometric sphere $I(C^{-1}BC^{-1})$ does not intersect the  isometric sphere  $I(A^{k}C^{-1}BC^{-1}A^{-k})$ for any non-zero $k \in \mathbb{Z}$. This proves (\ref{item:inverseCBinverseCandinverseCBinverseC}).

Note that the sum of the radius of  $C^{-1}BC^{-1}$ and $A^{k}CBC^{-1}A^{-k}$  is $\sqrt{2}+\frac{2}{\sqrt{3}}$, which is 2.568914101 numerically.
From Propositions \ref{prop:center-radiusCBCcBc}  and  \ref{prop:center-radiuscBCCBc},  the Cygan distance between the center of the isometric spheres  $I(C^{-1}BC^{-1})$ and $I(A^{k}CBC^{-1}A^{-k})$ is
$$\left| |(2k+\frac{1}{2})+\sqrt{\frac{5}{3}} \rm{i}-\frac{\sqrt{15}}{2} \rm{i}|^2+\rm{i}\left(-\frac{3\sqrt{15}}{2}+\frac{(7+8k)}{2}\sqrt{\frac{5}{3}}+ Im\left(\left(\frac{1}{2}-\frac{\sqrt{15}}{2}\rm{i}\right)\left(-2k+\sqrt{\frac{5}{3}}\rm{i}\right) \right) \right)\right|^{\frac{1}{2}}.$$
Which is
\begin{equation}\label{eq:cBcAkCBcak}
(176k^2+\frac{4}{9}+\frac{8k}{3}+16k^4+16k^3)^{\frac{1}{4}}.
\end{equation}
The term in (\ref{eq:cBcAkCBcak}) is 0.8164965807   numerically when $k=0$, and it is bigger than $\sqrt{2}+\frac{2}{\sqrt{3}}$ for any non-zero $k \in \mathbb{Z}$. So  the isometric sphere $I(C^{-1}BC^{-1})$ does not intersect the  isometric sphere  $I(A^{k}CBC^{-1}A^{-k})$ for any $|k| \geq 1$. This proves (\ref{item:inverseCBinverseCandCBinverseC}).
\end{proof}

The proof of Proposition \ref{prop:CBC} is similar to the proof of  Proposition \ref{prop:inverseCBinverseC}, we omit the details.

\begin{prop}\label{prop:CBC} For the isometric sphere $I(CBC)$, we have
	
	\begin{enumerate}
		\item  \label{item:CBCandCBinverseC}
		 $I(CBC)$ does not intersect the  isometric sphere  $I(A^{k}CBC^{-1}A^{-k})$ for any $k \in \mathbb{Z}$;
		\item  \label{item:CBCandinverseCBinverseC}
	 $I(CBC)$ does not intersect the  isometric sphere  $I(A^{k}C^{-1}BC^{-1}A^{-k})$ for any $k \in \mathbb{Z}$;
		\item  \label{item:CBCandinverseCBinverseC}
		
	 $I(CBC)$ does not intersect the  isometric sphere  $I(A^{k}CBCA^{-k})$ for any non-zero $k \in \mathbb{Z}$;
		
		\item  \label{item:CBCandinverseCBinverseC}
	 $I(CBC)$ does not intersect the  isometric sphere  $I(A^{k}C^{-1}BCA^{-k})$ for any non-zero $k \in \mathbb{Z}$.

	\end{enumerate}
\end{prop}

We also have 
\begin{prop}\label{prop:CBinverseC} For the isometric sphere $I(CBC^{-1})$, we have
	
	\begin{enumerate}
		\item  \label{item:CBinverseCandinverseCBC}
		 $I(CBC^{-1})$ does not intersect the  isometric sphere  $I(A^{k}C^{-1}BCA^{-k})$ for any $k \in \mathbb{Z}$;
		\item  \label{item:CBinverseCandinverseCBinverseC}
		 $I(CBC^{-1})$ does not intersect the  isometric sphere  $I(A^{k}C^{-1}BC^{-1}A^{-k})$ for any non-zero $k \in \mathbb{Z}$;
		
		\item  \label{item:CBinverseCandCBinverseC}
		 $I(CBC^{-1})$ does not intersect the  isometric sphere  $I(A^{k}CBC^{-1}A^{-k})$ for any non-zero $k \in \mathbb{Z}$;
		
		\item  \label{item:CBinverseCandCBC}
	 $I(CBC^{-1})$ does not intersect the  isometric sphere  $I(A^{k}CBCA^{-k})$ for any $k \in \mathbb{Z}$.

	\end{enumerate}
\end{prop}

\begin{proof}From Propositions  \ref{prop:CBC}  and \ref{prop:CBinverseC}, we only need to prove (\ref{item:CBinverseCandinverseCBC}) and (\ref{item:CBinverseCandCBinverseC}) of Proposition \ref{prop:CBinverseC}.
	
	From Proposition \ref{prop:center-radiusCBCcBc},  the Cygan distance between the centers of the isometric spheres  $I(CBC^{-1})$ and $I(A^{k}C^{-1}BCA^{-k})$ is
	$$ 2\sqrt{\frac{5}{3}+k^2},$$
	which is bigger than $2 \cdot \frac{2}{\sqrt{3}}$ for any $k \in \mathbb{Z}$,
		so the isometric spheres  $I(CBC^{-1})$ and $I(A^{k}C^{-1}BCA^{-k})$ do not intersect for any  $k \in \mathbb{Z}$.
		
		From Proposition \ref{prop:center-radiusCBCcBc},  the Cygan distance between the centers of the isometric spheres  $I(CBC^{-1})$ and $I(A^{k}CBC^{-1}A^{-k})$ is
		$$2\sqrt{|\frac{2 \sqrt{15}k\rm{i}}{3}+k^2|},$$
		which is bigger than $2 \cdot \frac{2}{\sqrt{3}}$ for any non-zero $k \in \mathbb{Z}$,
			so the isometric spheres  $I(CBC^{-1})$ and $I(A^{k}CBC^{-1}A^{-k})$ do not intersect for any non-zero  $k \in \mathbb{Z}$.
			
\end{proof}

The proof of the  follow Proposition  \ref{prop:inverseCBC} is similar, we omit the detailed calculations.
\begin{prop}\label{prop:inverseCBC} For the isometric sphere $I(C^{-1}BC)$, we have
	
	\begin{enumerate}
		\item  \label{item:CBCandCBinverseC}
		 $I(C^{-1}BC)$ does not intersect the  isometric sphere  $I(A^{k}CBCA^{-k})$ for any non-zero $k \in \mathbb{Z}$;
		\item  \label{item:CBCandinverseCBinverseC}
		 $I(C^{-1}BC)$ does not intersect the  isometric sphere  $I(A^{k}CBC^{-1}A^{-k})$ for any $k \in \mathbb{Z}$;
		\item  \label{item:CBCandinverseCBinverseC}
		
		 $I(C^{-1}BC)$ does not intersect the  isometric sphere  $I(A^{k}C^{-1}BCA^{-k})$ for any non-zero $k \in \mathbb{Z}$;
		
		\item  \label{item:CBCandinverseCBinverseC}
		 $I(C^{-1}BC)$ does not intersect the  isometric sphere  $I(A^{k}C^{-1}BC^{-1}A^{-k})$ for any  $k \in \mathbb{Z}$.

	\end{enumerate}
\end{prop}

\subsection{The combinatorics of ridges of  $D_{R}$ in  ${\bf H}^2_{\mathbb{C}}$}\label{subsec:ridge}
We study carefully the  combinatorics of ridges of  $D_{R}$ in this subsection.
Which are crucial for the application of the  Poincar\'e polyhedron theorem.
For each $g \in R$, the side $s(g)$ by definition is $I(g) \cap D_{R}$, which is a 3-dimensional object, we also describe the sides in details in this subsection.

We first take four points $u_i$ for $i=1,2,3,4$, which lie on the isometric spheres $I(C^{-1})$,  $I(CBC)$ and $I(CBC^{-1})$. See Figures \ref{figure:pointw}, \ref{figure:pointv} and  \ref{figure:abstract}.

In Equation (\ref{equaation:girauddisk}), we take ${\bf q}=(CBC)^{-1}(q_{\infty})$,  ${\bf r}=q_{\infty}$ and  ${\bf p}=C^{-1}BC(q_{\infty})$, then we can parameterize the intersection of the isometric spheres  $I(CBC)$ and $I(C^{-1}BC)$  by $V=V(z_1,z_2)$ with $\langle V,V \rangle <0$.
Where
$$V=\left(\begin{array}{c}
\frac{13\sqrt{15}\rm{i}}{16}-\frac{25}{16}+\frac{\sqrt{15}}{2}(-\rm{i}\cos(r)+\sin(r))+\rm{e}^{s \rm{i}}(\frac{1+\sqrt{15}\rm{i}}{2})\\[ 6 pt]
\frac{7}{4}-\frac{\sqrt{15}\rm{i}}{4}-\frac{3\rm{e}^{r \rm{i}}}{2}+\rm{e}^{s \rm{i}}\\[6 pt]
-\frac{3}{4}-\frac{\sqrt{15} \rm{i}}{4}\\ \end{array}\right),$$
with $(z_1,z_2)=(\rm{e}^{r \rm{i}},\rm{e}^{s \rm{i}}) \in S^{1}\times S^1$.

  %\begin{equation}\label{eq:CBCinverseCBCintersecinverseC}
  %(3\cos(s)+3)\cos(r)+3\sin(r)\sin(s)-3\cos(s)-3=0.
  %\end{equation}

Note that $\langle V,V \rangle = V^{*}H V$ is
\begin{equation}\label{eq:CBCinverseCBC}
(-3\cos(r)-1)\cos(s)-3\sin(s)\sin(r)-\frac{3\cos(r)}{2}+\frac{7}{2}.
\end{equation}
Consider the intersection of this Giraud  disk and the isometric sphere  $I(C^{-1})$, it is given by  the following
\begin{equation}\label{eq:CBCinverseCBCintersecinverseC}
\cos(s-r)+\cos(r)-\cos(s)-1=0.
\end{equation}
From the common solutions of (\ref{eq:CBCinverseCBC}) and  (\ref{eq:CBCinverseCBCintersecinverseC}),
we have four points in the Heisenberg group:
\begin{itemize}
	\item
	
	$u_1$ corresponds to $(0, \frac{\pi}{3})$ in (\ref{eq:CBCinverseCBC}), in horospherical coordinates it is
	$$u_1=[(\frac{1}{4}-\frac{\sqrt{5}}{4})+\rm{i}(\frac{\sqrt{15}}{4}-\frac{\sqrt{3}}{4}),-\sqrt{15}-\frac{3\sqrt{3}}{2}] \in \mathbb{C} \times \mathbb{R}.$$
	Which is one of  the cyan points in Figures \ref{figure:pointw}, \ref{figure:pointv} and  \ref{figure:abstract},  where we draw these points by very small cyan balls;
	\item
	$u_2$ corresponds to $(0, -\frac{\pi}{3})$ in (\ref{eq:CBCinverseCBC}), in horospherical coordinates it is
	$$u_2=[(\frac{1}{4}+\frac{\sqrt{5}}{4})+\rm{i}(\frac{\sqrt{15}}{4}+\frac{\sqrt{3}}{4}),-\sqrt{15}+\frac{3\sqrt{3}}{2}]\in \mathbb{C} \times \mathbb{R}.$$
	Which is one of the black points in Figures \ref{figure:pointw}, \ref{figure:pointv} and  \ref{figure:abstract}, where we draw these points by  very small black balls;
	\item
	$u_3$ corresponds to $(\arctan(2\sqrt{6}), \arctan(2\sqrt{6}))$ in (\ref{eq:CBCinverseCBC}), in horospherical coordinates it is
	$$u_3=[(-\frac{1}{5}+\frac{\sqrt{10}}{10})+\rm{i}(\frac{2\sqrt{15}}{5}+\frac{\sqrt{6}}{10}),-\frac{13\sqrt{15}}{10}-\frac{\sqrt{6}}{5}] \in \mathbb{C} \times \mathbb{R}.$$
	Which is one  of the blue points in Figures \ref{figure:pointw}, \ref{figure:pointv} and  \ref{figure:abstract}, where we draw these points by  very small blue balls;
	\item
	$u_4$ corresponds to $(-\arctan(2\sqrt{6}), -\arctan(2\sqrt{6}))$ in (\ref{eq:CBCinverseCBC}), in horospherical coordinates it is
	$$u_4=[(-\frac{1}{5}-\frac{\sqrt{10}}{10})+\rm{i}(\frac{2\sqrt{15}}{5}-\frac{\sqrt{6}}{10}),-\frac{13\sqrt{15}}{10}+\frac{\sqrt{6}}{5}]\in \mathbb{C} \times \mathbb{R}.$$
	Which is one of the red points in Figures \ref{figure:pointw}, \ref{figure:pointv} and  \ref{figure:abstract}, where we draw these points by  very small red balls.
\end{itemize}

By direct calculations, we have $$C^{-1}BC(u_1)=u_2, ~~~~C^{-1}BC(u_2)=u_1,~~~C^{-1}BC(u_3)=u_4, ~~~~C^{-1}BC(u_4)=u_3.$$
Moreover,  in horospherical coordinates we have
$$CBC(u_1)=[(-\frac{1}{4}+\frac{\sqrt{5}}{4})+\rm{i}(-\frac{\sqrt{15}}{4}+\frac{\sqrt{3}}{4}),-\sqrt{15}-\frac{3\sqrt{3}}{2}];$$
$$CBC(u_2)=[(-\frac{1}{4}-\frac{\sqrt{5}}{4})-\rm{i}(\frac{\sqrt{15}}{4}+\frac{\sqrt{3}}{4}),-\sqrt{15}+\frac{3\sqrt{3}}{2}];$$
$$CBC(u_3)=[(\frac{1}{5}+\frac{\sqrt{10}}{10})+\rm{i}(-\frac{2\sqrt{15}}{5}+\frac{\sqrt{6}}{10}),-\frac{13\sqrt{15}}{10}+\frac{\sqrt{6}}{5}];$$
$$CBC(u_4)=[(\frac{1}{5}-\frac{\sqrt{10}}{10})-\rm{i}(\frac{2\sqrt{15}}{5}+\frac{\sqrt{6}}{10}),-\frac{13\sqrt{15}}{10}-\frac{\sqrt{6}}{5}].$$
These four points lie on the isometric spheres $I(C)$, $I(CBC^{-1})$ and $I(C^{-1}BC^{-1})$.
We also draw $CBC(u_i)$ for $i=1,2,3,4$ in Figures \ref{figure:pointw}, \ref{figure:pointv} and  \ref{figure:abstract} by very small balls in cyan, blue, black and red colors.

\begin{prop}\label{prop:cCintersection}
The isometric sphere  $I(C)$ intersects the  isometric sphere  $I(A^{-1}CA)$ in a Giraud disk. This disk is disjoint from the isometric spheres  $I(A^{k}CBCA^{-k})$, $I(A^{k}C^{-1}BCA^{-k})$, $I(A^{k}CBC^{-1}A^{-k})$   and $I(A^{k}C^{-1}BC^{-1}A^{-k})$ for all $k \in \mathbb{Z}$.  So the ridge $s(C) \cap s(A^{-1}CA)$ is a disk with boundary entirely in  $\partial {\bf H}^2_{\mathbb{C}}$.	
\end{prop}

\begin{proof}
	%The defining equation of the isometric sphere of 	$ACA^{-1}$
%	\begin{equation}\label{eq:cygan-sphereACa}
%	\left\{ (z,t,u)\in\hc \cup \partial\hc: ||z+3|^2+u+\rm{i}(t+\frac{\sqrt{15}}{2}-6 \cdot  Im(z)|=4 \right\}.
%	\end{equation}
	Recall that $C^{-1}$	 and $A^{-1}CA$ have the same isometric sphere.  It is easy to see that $C^{-1}(q_{\infty})$,  $C(q_{\infty})$ and  $q_{\infty}$ are linearly independent.
	In Equation (\ref{equaation:girauddisk}), we take ${\bf q}=C^{-1}(q_{\infty})$,  ${\bf r}=C(q_{\infty})$ and  ${\bf p}=q_{\infty}$, then we can parameterize the intersection of the isometric spheres  $I(C)$ and $I(C^{-1})$  by $V=V(z_1,z_2)$ with $\langle V,V \rangle <0$.
	Where
	$$V=\left(\begin{array}{c}
	-\frac{1}{4}+\frac{\rm{i}\sqrt{15}}{8}-\frac{\rm{e}^{r \rm{i}}+\rm{e}^{s \rm{i}}}{2}\\[ 6 pt]
	\frac{	-\rm{e}^{r \rm{i}}+\rm{e}^{s \rm{i}}}{2}\\[ 6 pt]
	-\frac{1}{2}\\ \end{array}\right),$$
	and  $(z_1,z_2)=(\rm{e}^{r \rm{i}},\rm{e}^{s \rm{i}}) \in S^{1}\times S^1$.

	Note that $\langle V,V \rangle = V^* H V$ is $$3-2\sin(r)\sin(s)-2\cos(r)\cos(s)+2\cos(r)+2\cos(s).$$
	Take a sample point $r=s=\pi$, then $V=V(-1,-1) \in  {\bf H}^2_{\mathbb{C}}$. So the intersection of these two isometric spheres is not empty, then it is a  Giraud  disk.
	
	\begin{figure}
		\begin{center}
			\begin{tikzpicture}
			\node at (0,0) {\includegraphics[width=10cm,height=6cm]{{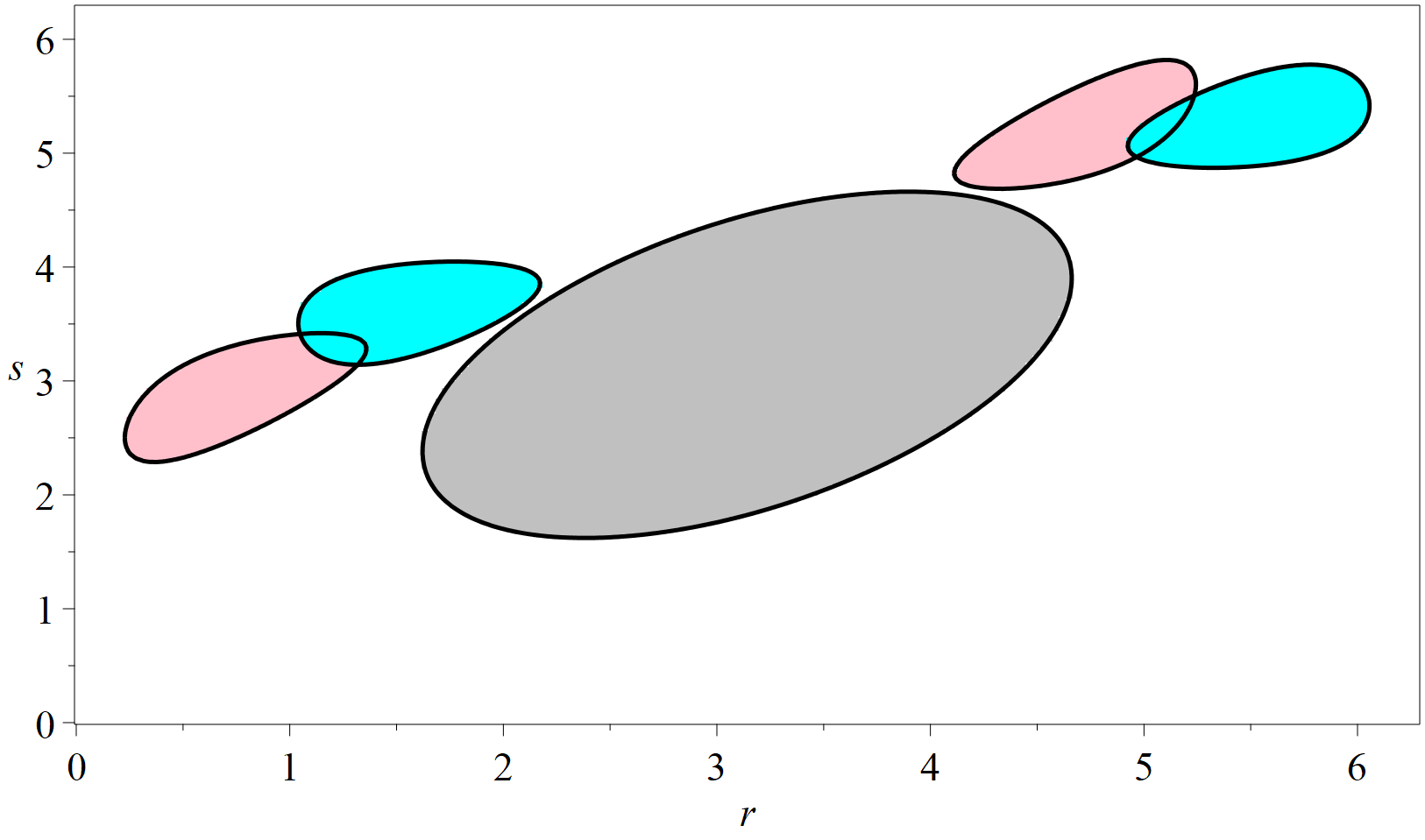}}};
			\end{tikzpicture}
		\end{center}
		\caption{The intersection of the isometric spheres  $I(C)$ and $I(A^{-1}CA)$ is a Giraud disk, which is disjoint from other isometric spheres $I(g)$ for $g \in R$.}
		\label{figure:cCdisk}
	\end{figure}

To show this disk is disjoint from the isometric spheres  $I(A^{k}CBCA^{-k})$, $I(A^{k}C^{-1}BCA^{-k})$, $I(A^{k}CBC^{-1}A^{-k})$ and $I(A^{k}C^{-1}BC^{-1}A^{-k})$ for all $k \in \mathbb{Z}$. From  propositions  in Subsection \ref{subsec:intersection}, we only need show
\begin{itemize}
\item  the intersection  $I(C) \cap I(A^{-1}CA) \cap  {\bf H}^2_{\mathbb{C}}$ and $I(CBC)$ is empty;
\item  the intersection  $I(C) \cap I(A^{-1}CA) \cap  {\bf H}^2_{\mathbb{C}}$ and $I(C^{-1}BC)$ is empty.
\end{itemize}

	Using Maple, from our parameterization of $I(C) \cap I(A^{-1}CA)$ above,  the maximum of $$|\langle V, q_{\infty}\rangle|^2-|\langle V, (CBC)^{-1}(q_{\infty})\rangle|^2$$ with the condition $\langle V, V \rangle \leq 0$ is -0.689216 numerically. So $I(C) \cap I(A^{-1}CA)$ is entirely in the exterior of  $I(CBC) \cap {\bf H}^2_{\mathbb{C}}$.
		Similarly, we have  $I(C) \cap I(A^{-1}CA)$ lies entirely in the exterior of $ I(C^{-1}BC) \cap {\bf H}^2_{\mathbb{C}}$.
	
		See Figure 	\ref{figure:cCdisk} for the  Giraud disk $I(C) \cap I(A^{-1}CA) \cap{\bf H}^2_{\mathbb{C}}$, it is the gray colored region in Figure 	\ref{figure:cCdisk}. Where the two cyan  disks is the region in the Giraud torus  where  $|\langle V, q_{\infty}\rangle|^2-|\langle V, (CBC)^{-1}(q_{\infty})\rangle|^2 >0$; and the two pink disks is the region in the Giraud torus  where   $|\langle V, q_{\infty}\rangle|^2-|\langle V, (C^{-1}BC)^{-1}(q_{\infty})\rangle|^2 >0$.
\end{proof}

\begin{prop}\label{prop:cBcandCintersecCBinverseC}
	The isometric sphere  $I(C^{-1}BC^{-1})$ intersects the  isometric sphere $I(C)$ in a Giraud disk. This disk also intersects with the isometric sphere  $I(CBC^{-1})$.  More precisely, the triple intersection $I(C^{-1}BC^{-1}) \cap I(C) \cap I(CBC^{-1})$ is a union of  two crossed straight segments, the ridge   $s(C^{-1}BC^{-1}) \cap s(C)$ is topologically the union of two sectors.
	
\end{prop}
\begin{proof}In Equation (\ref{equaation:girauddisk}), we take ${\bf q}=C^{-1}(q_{\infty})$,  ${\bf r}=CBC(q_{\infty})$ and  ${\bf p}=q_{\infty}$, then we can parameterize the intersection of the isometric spheres  $I(C^{-1}BC^{-1})$ and $I(C)$  by $V=V(z_1,z_2)$ with $\langle V,V \rangle <0$.
	Where
	
	$$V=\left(\begin{array}{c}
	-\frac{33}{16}+\frac{5 \sqrt{15}\rm{i}}{16}-\frac{\rm{e}^{r \rm{i}}}{2}+\rm{e}^{s \rm{i}}(-\frac{1}{2}-\frac{\sqrt{15} \rm{i}}{2})\\[ 6 pt]
	-\frac{3}{4}+\frac{\sqrt{15}\rm{i}}{4}-\frac{\rm{e}^{r \rm{i}}}{2}+\rm{e}^{s \rm{i}}\\[ 6 pt]
	-\frac{3}{4}-\frac{\sqrt{15} \rm{i}}{4}\\ \end{array}\right),$$
	with  $(z_1,z_2)=(\rm{e}^{r \rm{i}},\rm{e}^{s \rm{i}}) \in S^{1}\times S^1$.
	
	Note that $\langle V,V \rangle = V^{*}H V$ is $$(-\cos(r)+3)\cos(s)-\sin(s)\sin(r)+\frac{3\cos(r)}{2}+\frac{7}{2}.$$
	Take the sample point $r=s=\pi$, then $V=V(-1,-1) \in  {\bf H}^2_{\mathbb{C}}$. So the intersection of these two isometric spheres is not empty, then it is a Giraud disk.
	See Figure 	\ref{figure:inverseCBinverseCandCintersectCBinverseC} for this disk. The region bounded by the circle is the Giraud disk $I(C^{-1}BC^{-1}) \cap I(C)$.

	% \cap  \partial  {\bf H}^2_{\mathbb{C}}
		% $$(-3\cos(s)-3)\cos(r)-3\sin(s)\sin(r)-3\cos(s)-3 = 0$$
		
	We now consider the intersection of the isometric sphere  $I(CBC^{-1})$ with this Giraud disk,  that is,  we consider $V$ with  $$|\langle V, q_{\infty} \rangle |=|\langle V, (CBC^{-1})^{-1}(q_{\infty}) \rangle |.$$
	Which is equivalent to 
	$$\cos(s-r)-\cos(r)+\cos(s)+1 = 0.$$
	 The solutions are $r= \pi$, $s= \pi$ or $r+s=\pi$.
	 It is easy to see that when  $r+s=\pi$, $\langle V,V \rangle>0$, so  the points $V$ corresponding to the straight segment $r+s=\pi$
	 lie entirely outside the complex hyperbolic space.
	 It is easy to see  that the two crossed straight segments
	 $$\left\{r = \pi, s \in \left[\frac{2\pi}{3},\frac{4 \pi}{3}\right]\right\},\qquad  \left \{s = \pi,  r\in \left[\pi-\arctan(2\sqrt{6}),\pi+\arctan(2\sqrt{6})\right] \right\}$$
	 decompose the Giraud disk
	 into four sectors with a common vertex. Exactly one pair of the four sectors lie in the exterior of the isometric
	 sphere  $I(CBC^{-1})$. Therefore, the ridge   $s(C^{-1}BC^{-1}) \cap s(C)$ is a union of two sectors.
	 
See Figure  \ref{figure:inverseCBinverseCandCintersectCBinverseC} for the triple intersection $I(C^{-1}BC^{-1}) \cap I(C) \cap I(CBC^{-1})$ and the ridge 	$s(C^{-1}BC^{-1}) \cap s(C)$. The pink region in Figure \ref{figure:inverseCBinverseCandCintersectCBinverseC}  is given by  $$|\langle V, q_{\infty} \rangle |-|\langle V, (CBC^{-1})^{-1}(q_{\infty}) \rangle |<0,$$ so this region lies in the exterior of the isometric sphere $I(CBC^{-1})$.

	 	\begin{figure}
	 	\begin{center}
	 		\begin{tikzpicture}
	 		\node at (0,0) {\includegraphics[width=7cm,height=7cm]{{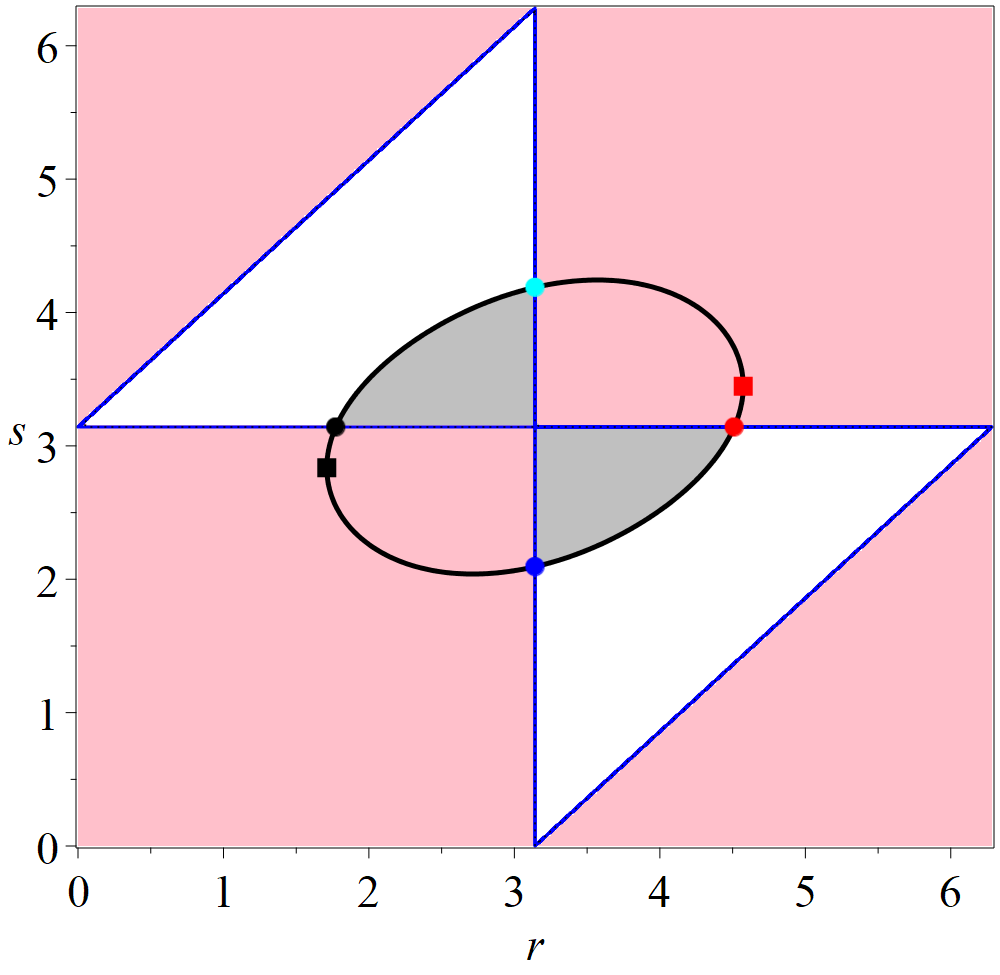}}};
	 		%	\coordinate [label=below:$CBC(u_{3})$] (S) at (-0.7,-3.2);
	 		\end{tikzpicture}
	 	\end{center}
	 	\caption{The intersection of the isometric spheres  $I(C^{-1}BC^{-1})$ and $I(C)$ is a Giraud disk, which is the disk  bounded by the circle.  The pink region is in the exterior of the isometric  $I(CBC^{-1})$.  So the two pink colored sectors inside the circle is the ridge  $s(C^{-1}BC^{-1}) \cap s(C)$.}
	 	\label{figure:inverseCBinverseCandCintersectCBinverseC}
	 \end{figure}

%	 The region bounded by the circle is the Giraud disk $I(C^{-1}BC^{-1}) \cap I(C) \cap \partial  {\bf H}^2_{\mathbb{C}}$.
	 
	 The proof of Proposition  \ref{prop:cBcandCintersecCBinverseC} is now complete. For the application of the Poincar\'e polyhedron theorem later, we make explicit parameterization of $$I(C^{-1}BC^{-1}) \cap I(C) \cap  I(CBC^{-1}) \cap \partial  {\bf H}^2_{\mathbb{C}}.$$
	 
	First we  take four points
	 \begin{enumerate}

	 	\item
	 $(r,s)=(\pi,\frac{2\pi}{3})$ in our  parameterization of $I(C^{-1}BC^{-1}) \cap I(C)$, which is $CBC(u_2)$, it is the blue solid circle marked point in Figure 	\ref{figure:inverseCBinverseCandCintersectCBinverseC};
\item 	$(r,s)=(\pi,\frac{4\pi}{3})$ in our   parameterization of $I(C^{-1}BC^{-1}) \cap I(C)$, which is $CBC(u_1)$, it is the cyan solid circle marked  point in Figure 	\ref{figure:inverseCBinverseCandCintersectCBinverseC};
\item 	$(r,s)=(\pi-\arctan(2\sqrt{6}),\pi)$ in our  parameterization of $I(C^{-1}BC^{-1}) \cap I(C)$, which is $CBC(u_3)$, it is the black solid circle marked  point in Figure 	\ref{figure:inverseCBinverseCandCintersectCBinverseC};
\item 	$(r,s)=(\pi+\arctan(2\sqrt{6}),\pi)$ in our   parameterization of $I(C^{-1}BC^{-1}) \cap I(C)$, which is $CBC(u_4)$, it is the red solid circle marked  point in Figure 	\ref{figure:inverseCBinverseCandCintersectCBinverseC}.
	\end{enumerate}

%there are four points:
%\begin{equation}\label{eq:inverseCBinverseCandC1}
%f_1=\arctan(-\frac{(-2\cos(s1)+6)(3\cos(s1)^2-2\cos(s1)+\sqrt{9\cos(s1)^4+66\cos(s%1)^3-66\cos(s1)-9}-21}{8\sin(s1)(3\cos(s1)-5)}+\frac{3\cos(s1)+7}{2\sin(s1)}, -\frac{3\cos(s1)^2-2\cos(s1)+\sqrt{(9\cos(s1)^4+66\cos(s1)^3-66\cos(s1)-9}-21}{4(3\cos(s1)-5)}
%\end{equation}

We remind the reader for two real arguments  $x,y$, $\arctan(y,x)$ is the principal value of the  argument of the complex number  $x+y \cdot \rm{i}$, so $-\pi < \arctan(y,x) \leq \pi$. This function is extended to complex arguments by the formula  $$\arctan(x,y)=- \rm{i} \cdot  \ln \left(\frac{x+ y \cdot \rm{i}}{\sqrt{x^2+y^2}}\right).$$
 
Now we  solve the equation  $\langle V,V \rangle=0$. We take $f_1$ be the function $\arctan(D,E)$, where

\begin{flalign} \label{eq:inverseCBinverseCandC1partA}
\nonumber &
D=\frac{(\cos(r)-3) \cdot (3\cos(r)^2-2\cos(r)-21+\sqrt{9\cos(r)^4+66\cos(r)^3-66\cos(r)-9})}{4\sin(r)(3\cos(r)-5)}
& \\ &
+\frac{3\cos(r)+7}{2\sin(r)}& \nonumber
\end{flalign}
and
\begin{flalign}
\nonumber &
E=\frac{3\cos(r)^2-2\cos(r)-21+\sqrt{9\cos(r)^4+66\cos(r)^3-66\cos(r)-9}}{20-12\cos(r)}.& \nonumber
\end{flalign}
We also take $f_2$ be the function $\arctan(F,G)$, where
\begin{flalign}
\nonumber &
F=\frac{(\cos(r)-3)\cdot (3\cos(r)^2-2\cos(r)-21-\sqrt{9\cos(r)^4+66\cos(r)^3-66\cos(r)-9})}{4\sin(r)(3\cos(r)-5)}
& \\ &
+\frac{3\cos(r)+7}{2\sin(r)}& \nonumber
\end{flalign}
and
\begin{flalign}
\nonumber &
G=\frac{3\cos(r)^2-2\cos(r)-21-\sqrt{9\cos(r)^4+66\cos(r)^3-66\cos(r)-9}}{20-12\cos(r)}.& \nonumber
\end{flalign}

We now take four arcs in the triple  intersection of  $I(C^{-1}BC^{-1})$, $I(C)$ and $\partial  {\bf H}^2_{\mathbb{C}}$:
\begin{itemize}
	\item
 $\mathcal{C}_{C^{-1}BC^{-1}, C, 1}$ be the curve with $V=V(r,s)$,  where  $s=f_1(r)$ and $$r \in \left[\pi, \pi-\arctan \left(\frac{2\sqrt{-56+22\sqrt{7}}}{-11+4\sqrt{7}}\right)\right];$$
	\item  $\mathcal{C}_{C^{-1}BC^{-1}, C, 2}$ be the curve with $V=V(r,s)$,   where $s=f_2(r)$ and $$r \in \left[\pi+\arctan(2\sqrt{6}), \pi-\arctan\left(\frac{2\sqrt{-56+22\sqrt{7}}}{-11+4\sqrt{7}}\right)\right];$$
	\item  $\mathcal{C}_{C^{-1}BC^{-1}, C, 3}$ be the curve with $V=V(r,s)$,   where $s=f_1(r)$ and $$r \in \left[ \pi+\arctan\left(\frac{2\sqrt{-56+22\sqrt{7}}}{-11+4\sqrt{7}}\right), \pi \right];$$
	\item  $\mathcal{C}_{C^{-1}BC^{-1}, C, 4}$ be the curve with $V=V(r,s)$,   where $s=f_2(r)$ and $$r \in \left[ \pi+\arctan\left(\frac{2\sqrt{-56+22\sqrt{7}}}{-11+4\sqrt{7}}\right), \pi-\arctan(2\sqrt{6})\right].$$
\end{itemize}

	Then  $\mathcal{C}_{C^{-1}BC^{-1}, C, 1}$ is an arc with one end point $CBC(u_1)$,  $\mathcal{C}_{C^{-1}BC^{-1}, C, 2}$ is an arc with one end point $CBC(u_4)$.  $\mathcal{C}_{C^{-1}BC^{-1}, C, 1}$ and  $\mathcal{C}_{C^{-1}BC^{-1}, C, 2}$ have a common end point with $$(r,s)=\left(\pi-\arctan(\frac{2\sqrt{-56+22\sqrt{7}}}{-11+4\sqrt{7}}\right), \pi+\arctan\left(\frac{(1+\sqrt{7})\sqrt{-56+22\sqrt{7}}}{-4+8\sqrt{7}}\right).$$
	This point is the red solid square marked point in   Figure  \ref{figure:inverseCBinverseCandCintersectCBinverseC}.
Then the union of  $\mathcal{C}_{C^{-1}BC^{-1}, C, 1}$ and  $\mathcal{C}_{C^{-1}BC^{-1}, C, 2}$ is an arc connecting  $CBC(u_1)$ and $CBC(u_4)$.
Similarly,  $\mathcal{C}_{C^{-1}BC^{-1}, C, 3}$ and  $\mathcal{C}_{C^{-1}BC^{-1}, C, 4}$ are two arcs with a common end point the black solid square marked point in   Figure  \ref{figure:inverseCBinverseCandCintersectCBinverseC}, and 
the union of  $\mathcal{C}_{C^{-1}BC^{-1}, C, 3}$ and  $\mathcal{C}_{C^{-1}BC^{-1}, C, 4}$ is an arc connecting  $CBC(u_2)$ and $CBC(u_3)$.

 Moreover,  it can be showed that the two arcs $$\mathcal{C}_{C^{-1}BC^{-1}, C, 1} \cup\mathcal{C}_{C^{-1}BC^{-1}, C, 2}$$ and $$\mathcal{C}_{C^{-1}BC^{-1}, C, 3} \cup\mathcal{C}_{C^{-1}BC^{-1}, C, 4}$$  are exactly the part of the triple intersection  $$I(C^{-1}BC^{-1}) \cap I(C) \cap \partial {\bf H}^2_{\mathbb{C}}$$ which are in the exterior of the isometric sphere  $I(CBC^{-1})$.

\end{proof}

\begin{prop}\label{prop:cBcandCBinverseCintersectionC}
	The isometric sphere $I(C^{-1}BC^{-1})$ intersects the  isometric sphere $I(CBC^{-1})$ in a Giraud  disk. This disk also intersects with the isometric sphere  $I(C)$.	Moreover, the ridge  $s(C^{-1}BC^{-1}) \cap s(CBC^{-1})$ is topologically the union of two sectors.
	
\end{prop}
\begin{proof}In Equation (\ref{equaation:girauddisk}), we take ${\bf q}=CBC(q_{\infty})$,  ${\bf r}=CBC^{-1}(q_{\infty})$ and ${\bf p}=q_{\infty}$, then we can parameterize the intersection of the isometric spheres  $I(C^{-1}BC^{-1})$ and $I(CBC^{-1})$  by $V=V(z_1,z_2)$ with $\langle V,V \rangle <0$.
	Where
	$$V=\left(\begin{array}{c}
	\frac{25}{16}-\frac{13\sqrt{15}\rm{i}}{16}+\frac{\sqrt{15}(\rm{i}\cos(r)-\sin(r))}{2}+	\rm{e}^{s \rm{i}}(-\frac{1}{2}-\frac{\sqrt{15} \rm{i}}{2})\\[ 6 pt]
	\frac{7}{4}-\frac{\sqrt{15} \rm{i}}{4}-\frac{3\rm{e}^{r \rm{i}}}{2}+
\rm{e}^{s \rm{i}}	\\[ 6 pt]
	\frac{3}{4}+\frac{\sqrt{15} \rm{i}}{4}\\ \end{array}\right),$$
	and $(z_1,z_2)=(\rm{e}^{r \rm{i}},\rm{e}^{s \rm{i}}) \in S^{1}\times S^1$.

	Note that $\langle V,V \rangle = V^{*}H V$ is
\begin{equation} \label{equa:cBcandCBinverseC} (-3\cos(r)-1)\cos(s)-3\sin(s)\sin(r)-\frac{3\cos(r)}{2}+\frac{7}{2}.
\end{equation}
		Take the sample point  $r=s=0$, then $V=V(1,1) \in  {\bf H}^2_{\mathbb{C}}$. So the intersection of these two isometric spheres is not empty, then it is Giraud  a disk.

	We consider the intersection of the isometric sphere  $I(C)$ with this Giraud disk,  that is $V$ as above with $|\langle V, q_{\infty} \rangle |=|\langle V, C^{-1}(q_{\infty})\rangle |$. It is equivalent to
		$$\cos(s-r)+\cos(r)-\cos(s)-1 = 0.$$ 
	The solutions are $r= 0$, $s= \pi$ or $r=s$.
	It is easy to see that when  $s=\pi$, $\langle V,V \rangle>0$, so  the points $V$ corresponding to the straight segment $s=\pi$
	lie entirely outside the complex hyperbolic space.
	
	%$$(3\cos(s)+3)\cos(r)+3\sin(s)\sin(r)-3\cos(s)-3 = 0$$
	
	It is easy to see  that the two crossed straight segments
	$$\left\{r = 0, s\in \left[-\frac{\pi}{3},\frac{\pi}{3}\right]\right\}, \quad \left\{r=s  \in \left[-\arctan\left(2\sqrt{6}\right),  \arctan\left(2\sqrt{6}\right)\right] \right\}$$
	decompose the Giraud disk
	into four sectors with one common vertex. Exactly one pair of the four sectors lie in the exterior the isometric
	sphere  $I(C)$. Therefore, the ridge $s(C^{-1}BC^{-1}) \cap s(CBC^{-1})$ is a union of two sectors.

See Figure 	\ref{figure:cBcandCBcinterC}, where the disk bounded by the circle is the Giraud disk $$I(C^{-1}BC^{-1}) \cap I(CBC^{-1}) \cap \partial {\bf H}^2_{\mathbb{C}}.$$   The pink region is where $$|\langle V, q_{\infty} \rangle |<|\langle V, C^{-1}(q_{\infty})\rangle |,$$ so the two pink colored sectors  inside the circle  is the ridge $s(C^{-1}BC^{-1}) \cap s(CBC^{-1})$.

The proof of Proposition  \ref{prop:cBcandCBinverseCintersectionC} is now complete. For the application of the Poincar\'e polyhedron theorem later, we make explicit parametrization of $$I(C^{-1}BC^{-1}) \cap I(CBC^{-1}) \cap  I(C) \cap \partial  {\bf H}^2_{\mathbb{C}}$$ in our parametrization of $I(C^{-1}BC^{-1}) \cap I(CBC^{-1})$.
		
 First we take four points:
	\begin{enumerate}
		\item
		$(r,s)=(0,\frac{\pi}{3})$ in our parametrization of $I(C^{-1}BC^{-1}) \cap I(CBC^{-1})$, which is $CBC(u_1)$. It is the cyan solid circle labeled point in Figure 	\ref{figure:cBcandCBcinterC};
		\item 	$(r,s)=(0,-\frac{\pi}{3})$ in our parametrization of $I(C^{-1}BC^{-1}) \cap I(CBC^{-1})$, which is $CBC(u_2)$. It is the blue  solid circle labeled  point in Figure 	\ref{figure:cBcandCBcinterC};
		\item 	$(r,s)=(\arctan(2\sqrt{6}),\arctan(2\sqrt{6}))$ in our parametrization of $I(C^{-1}BC^{-1}) \cap I(CBC^{-1})$, which is $CBC(u_4)$. It is the red  solid circle labeled  point in Figure 	 \ref{figure:cBcandCBcinterC};
		\item 	$(r,s)=(-\arctan(2\sqrt{6}),-\arctan(2\sqrt{6}))$ in our parametrization of $I(C^{-1}BC^{-1}) \cap I(CBC^{-1})$, which is $CBC(u_3)$. It is the black  solid circle labeled  point in Figure 	 \ref{figure:cBcandCBcinterC}.
	\end{enumerate}
	
	\begin{figure}
		\begin{center}
			\begin{tikzpicture}
			\node at (0,0) {\includegraphics[width=10cm,height=6cm]{{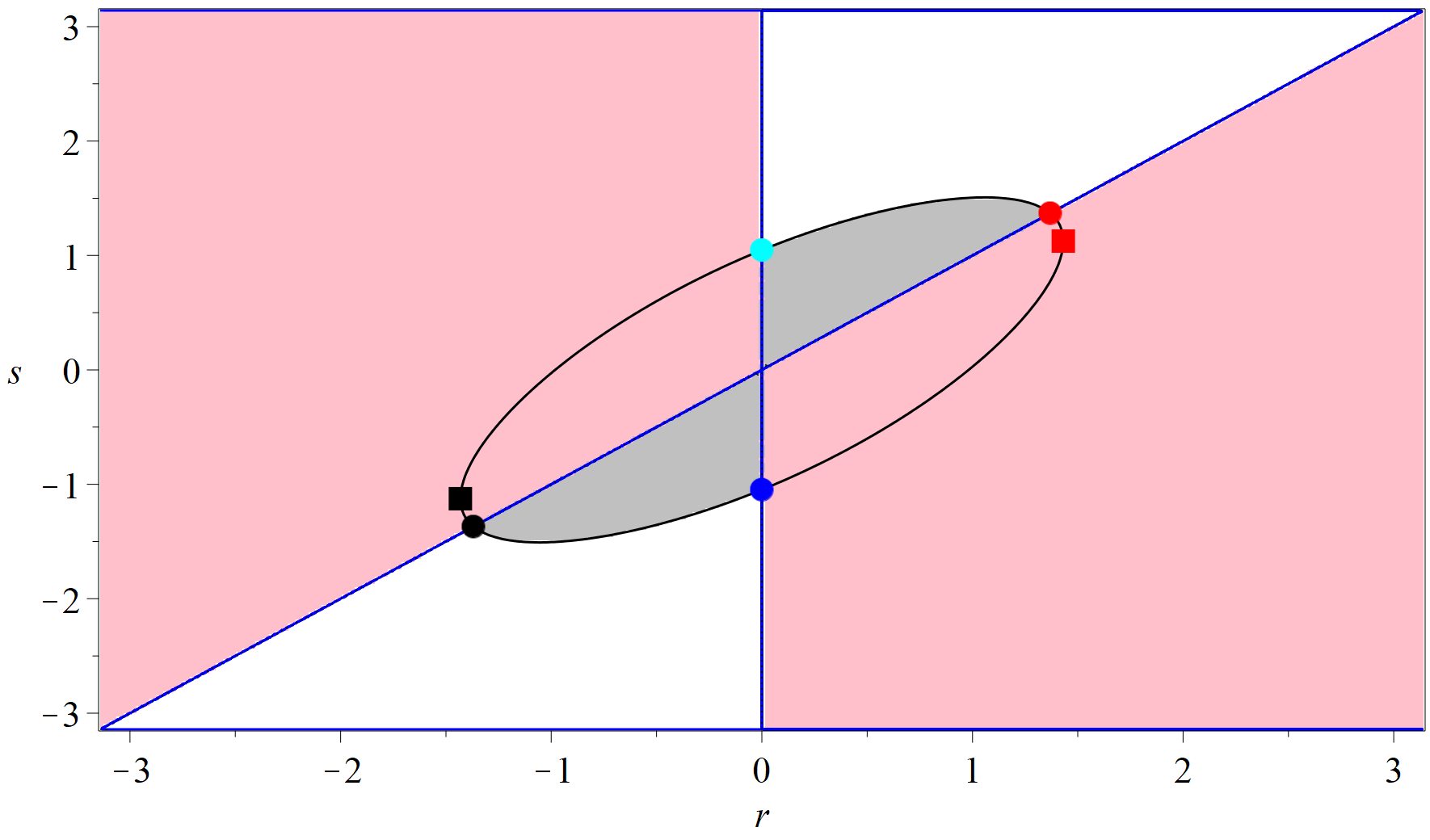}}};
			\end{tikzpicture}
		\end{center}
		\caption{The intersection of the isometric spheres  $I(C^{-1}BC^{-1})$ and $I(CBC^{-1})$ is a Giraud disk. The pink region lies in the exterior of the isometric sphere  $I(C)$. So the pink sectors inside the circle  is the ridge  $s(C^{-1}BC^{-1}) \cap s(CBC^{-1})$.}
		\label{figure:cBcandCBcinterC}
	\end{figure}

	Moreover, solve the equation of $\langle V,V \rangle=0$.  We get $s=g_1(r)$ or $s=g_2(r)$, where $g_{i}$ has similar form as the function $f_{i}$ in Proposition \ref{prop:cBcandCintersecCBinverseC}, but they are more involved. We omits the detailed terms of $g_{i}$ but only note that $$g_{1}\left(-\arctan\left(\frac{(12+8\sqrt{7})\sqrt{48\sqrt{7}-126}}{-66+32\sqrt{7}}\right)\right)=-\arctan\left(\frac{\sqrt{48\sqrt{7}-126}}{8-3\sqrt{7}}\right)$$ and $$g_{2}\left(\arctan\left(\frac{(12+8\sqrt{7})\sqrt{48\sqrt{7}-126}}{-66+32\sqrt{7}}\right)\right)=\arctan\left(\frac{\sqrt{48\sqrt{7}-126}}{8-3\sqrt{7}}\right).$$
	Note that $\pm \arctan\left(\frac{\sqrt{48\sqrt{7}-126}}{8-3\sqrt{7}}\right)$ are the maximum and minimum of $s$ with the condition $\langle V,V \rangle =0$ in
	(\ref{equa:cBcandCBinverseC}), see the red and black solid square labeled points in Figure 	\ref{figure:cBcandCBcinterC}.

	We now take four arcs in the  triple  intersection of   $I(C^{-1}BC^{-1})$,  $I(CBC^{-1})$ and $\partial  {\bf H}^2_{\mathbb{C}}$:
\begin{itemize}		
\item
$\mathcal{C}_{C^{-1}BC^{-1}, CBC^{-1}, 1}$ be the curve with $V=V(r,s)$,   $s=g_1(r)$ and $$r \in \left[0,  \arctan\left(\frac{2\sqrt{-56+22\sqrt{7}}}{11-4\sqrt{7}}\right)\right];$$
\item  $\mathcal{C}_{C^{-1}BC^{-1}, CBC^{-1}, 2}$ be the curve with $V=V(r,s)$,   $s=g_2(r)$ and $$r \in \left[\arctan\left(\frac{2\sqrt{-56+22\sqrt{7}}}{11-4\sqrt{7}}\right),\arctan\left(2\sqrt{6}\right)\right];$$
\item  $\mathcal{C}_{C^{-1}BC^{-1},CBC^{-1}, 3}$ be the curve with $V=V(r,s)$,   $s=g_1(r)$ and $$r \in \left[-\arctan\left(\frac{2\sqrt{-56+22\sqrt{7}}}{11-4\sqrt{7}}\right),0\right];$$	
\item  $\mathcal{C}_{C^{-1}BC^{-1}, CBC^{-1}, 4}$ be the curve with $V=V(r,s)$,   $s=g_2(r)$ and $$r \in \left[ -\arctan\left(2\sqrt{6}\right),-\arctan\left(\frac{2\sqrt{-56+22\sqrt{7}}}{11-4\sqrt{7}}\right)\right].$$
\end{itemize}

Then  $\mathcal{C}_{C^{-1}BC^{-1}, CBC^{-1}, 1}$ is an arc with one end point $CBC(u_2)$,  $\mathcal{C}_{C^{-1}BC^{-1}, CBC^{-1}, 2}$ is an arc with one end point $CBC(u_4)$.  $\mathcal{C}_{C^{-1}BC^{-1}, CBC^{-1}, 1}$ and  $\mathcal{C}_{C^{-1}BC^{-1},  CBC^{-1}, 2}$ have a common end point with $$(r,s)=\left(\arctan\left(\frac{2\sqrt{-56+22\sqrt{7}}}{11-4\sqrt{7}}\right), \arctan\left(\frac{(1+\sqrt{7})\sqrt{-56+22\sqrt{7}}}{-8+4\sqrt{7}}\right)\right).$$	Then the union of  $\mathcal{C}_{C^{-1}BC^{-1}, CBC^{-1}, 1}$ and  $\mathcal{C}_{C^{-1}BC^{-1}, CBC^{-1}, 2}$ is an arc connecting  $CBC(u_2)$ and $CBC(u_4)$.
Similarly, the union of  $\mathcal{C}_{C^{-1}BC^{-1}, CBC^{-1}, 3}$ and  $\mathcal{C}_{C^{-1}BC^{-1}, CBC^{-1}, 1}$ is an arc connecting  $CBC(u_1)$ and $CBC(u_3)$.
		
Moreover,  it can be showed that the two arcs $$\mathcal{C}_{C^{-1}BC^{-1}, CBC^{-1}, 1} \cup\mathcal{C}_{C^{-1}BC^{-1}, CBC^{-1}, 2}$$ and $$\mathcal{C}_{C^{-1}BC^{-1}, CBC^{-1}, 3} \cup\mathcal{C}_{C^{-1}BC^{-1}, CBC^{-1}, 4}$$  are exactly the part of the triple intersection $$I(C^{-1}BC^{-1}) \cap I(CBC^{-1}) \cap \partial {\bf H}^2_{\mathbb{C}}$$ which is in the exterior of the isometric sphere  $I(C)$.

\end{proof}

\begin{prop}\label{prop:CBcandCintersectioninverseCBinverseC}
	The isometric sphere $I(CBC^{-1})$ intersects the  isometric sphere  $I(C)$ in a Giraud disk. This disk also intersects with the isometric sphere  $I(C^{-1}BC^{-1})$.	Moreover, the ridge   $s(CBC^{-1})\cap s(C)$ is topologically the union of two sectors.
	
\end{prop}
\begin{proof}In Equation (\ref{equaation:girauddisk}), we take ${\bf q}=CBC^{-1}(q_{\infty})$,  ${\bf r}=C^{-1}(q_{\infty})$ and ${\bf p}=q_{\infty}$, then we can parameterize the intersection of the isometric spheres  $I(C^{-1}BC^{-1})$ and $I(CBC^{-1})$  by $V=V(z_1,z_2)$ with $\langle V,V \rangle <0$.
	Where
		$$V=\left(\begin{array}{c}
	\frac{25}{16}-\frac{5\sqrt{15}\rm{i}}{16}-\frac{\rm{e}^{r \rm{i}}}{2}+	\left(-\rm{i}\cos(s)+\sin(s)\right)\frac{\sqrt{15}}{2}\\[ 6 pt]
	\frac{1}{4}-\frac{\sqrt{15} \rm{i}}{4}	-\frac{\rm{e}^{r \rm{i}}}{2}+\frac{3\rm{e}^{s \rm{i}}}{2}	\\[ 6 pt]
	\frac{3}{4}+\frac{\sqrt{15} \rm{i}}{4}\\ \end{array}\right),$$
	and $(z_1,z_2)=(\rm{e}^{r \rm{i}},\rm{e}^{s \rm{i}}) \in S^{1}\times S^1$.

	Note that $\langle V,V \rangle = V^{*}H V$ is
	$$\frac{(-3\cos(s)-2)\cos(r)}{2}-\frac{3\sin(s)\sin(r)}{2}-3\cos(s)+\frac{7}{2}.$$
Take a sample point $r=s=0$, then $V=V(1,1) \in  {\bf H}^2_{\mathbb{C}}$. So the intersection of these two isometric spheres is not empty, then it is a Giraud disk.
	
%	$$(3\cos(s)-3)\cos(r)+3\sin(s)\sin(r)+3\cos(s)-3 = 0$$
We consider the intersection of the isometric sphere  $I(C^{-1}BC^{-1})$ with this Giraud disk, that is $$|\langle V, q_{\infty} \rangle |=|\langle V, C^{-1}BC^{-1}(q_{\infty}) \rangle.$$ It   is equivalent to 
	$$\cos(s-r)-\cos(r)+\cos(s)-1 = 0.$$	The solutions are $r= \pi$, $s=0$ or $r=s$.
	It is easy to see that when  $r=\pi$, $\langle V,V \rangle>0$, so  the points $ V$ corresponding to the straight segment $r=\pi$
	lie entirely outside the complex hyperbolic space.
	
	It is easy to see  that the two crossed straight segments
	$$\left\{s = 0, r\in \left[-\arctan\left(2\sqrt{6}\right),  \arctan\left(2\sqrt{6}\right)\right]\right\}, \quad \left\{r=s  \in \left[-\frac{\pi}{3},\frac{ \pi}{3}\right]\right\}$$
	decompose the Giraud disk
	into four sectors with one common vertex. Exactly one pair of the four sectors lie in the exterior of  the isometric
	sphere  $I(C^{-1}BC^{-1})$. Therefore, the ridge $s(CBC^{-1})\cap s(C)$  is a union of two sectors.

	See Figure 	\ref{figure:CBinverseCandCinterinverseCBinverseC} for this Giraud disk. The  region bounded by the circle is the Giraud disk. The pink region lies in the exterior of $I(C^{-1}BC^{-1})$, so the two pink sectors  inside the circle is the ridge $s(CBC^{-1})\cap s(C)$.
	
	The proof of Proposition  \ref{prop:CBcandCintersectioninverseCBinverseC} is now complete. For the application of the Poincar\'e polyhedron theorem later, we make explicit parameterization of $$I(CBC^{-1}) \cap I(C) \cap  I(C^{-1}BC^{-1}) \cap \partial  {\bf H}^2_{\mathbb{C}}$$ in our parameterization of $I(CBC^{-1}) \cap I(C)$.
	
	We first take four points
	\begin{enumerate}	
		\item
		$(r,s)=(\frac{\pi}{3},\frac{\pi}{3})$ in our  parametrization of $I(CBC^{-1}) \cap I(C)$, which is $CBC(u_1)$. It is the cyan solid circle labeled point in Figure 	 \ref{figure:CBinverseCandCinterinverseCBinverseC};
		\item 	$(r,s)=(-\frac{\pi}{3},-\frac{\pi}{3})$ in our  parameterzation of $I(CBC^{-1}) \cap I(C)$, which is $CBC(u_2)$. It  is the blue solid circle labeled point in Figure 	 \ref{figure:CBinverseCandCinterinverseCBinverseC};
		\item 	$(r,s)=(\arctan(2\sqrt{6}),0)$ in our  parametrization of $I(CBC^{-1}) \cap I(C)$, which is $CBC(u_4)$.  It is the red solid circle labeled point in Figure 	 \ref{figure:CBinverseCandCinterinverseCBinverseC};
		\item 	$(r,s)=(-\arctan(2\sqrt{6}),0)$ in our  parametrization of $I(CBC^{-1}) \cap I(C)$, which is $CBC(u_3)$. It is the black solid circle labeled point in Figure 	 \ref{figure:CBinverseCandCinterinverseCBinverseC}.
	\end{enumerate}
	
	\begin{figure}
		\begin{center}
			\begin{tikzpicture}
			\node at (0,0) {\includegraphics[width=10cm,height=6cm]{{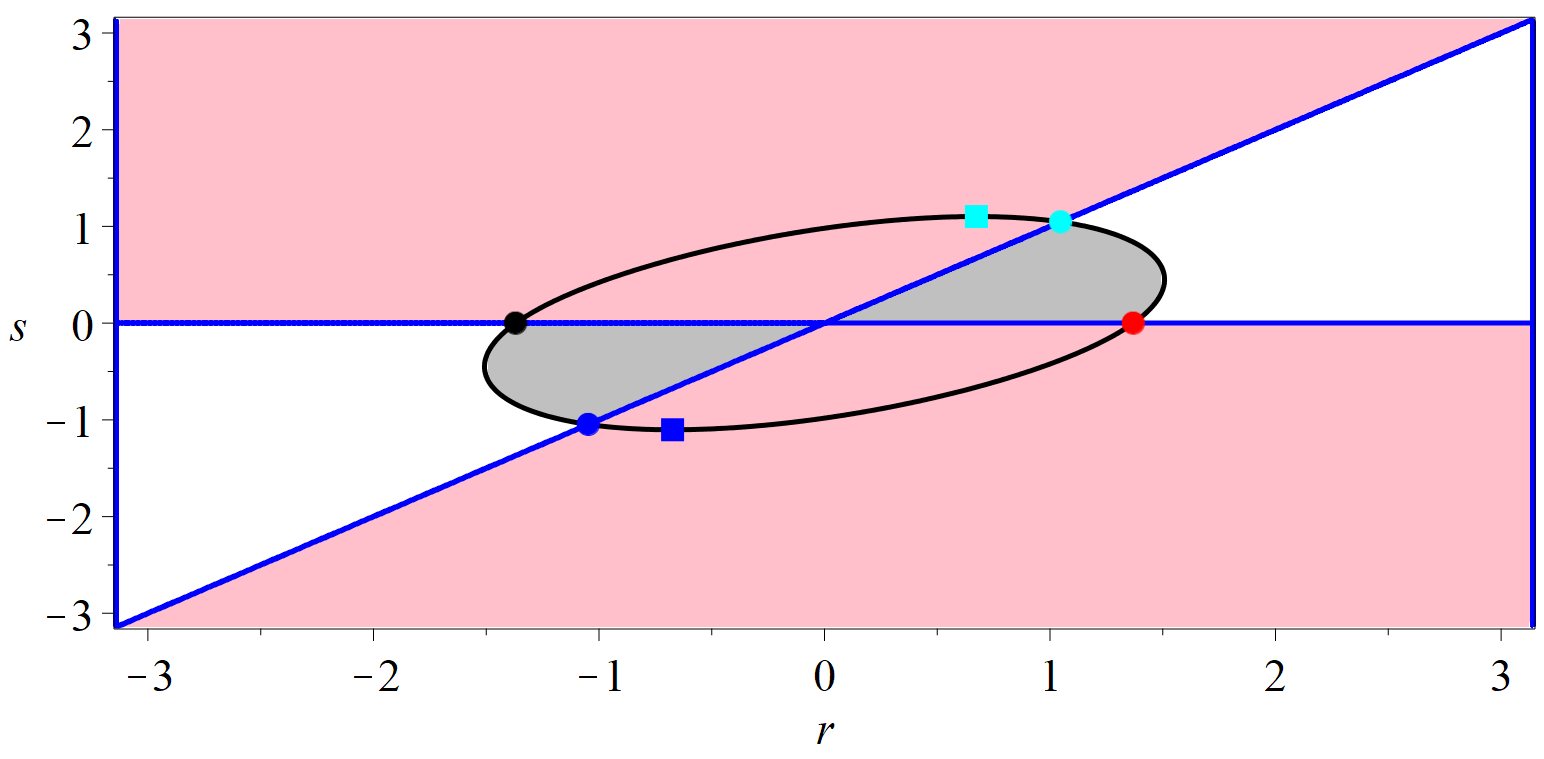}}};
			\end{tikzpicture}
		\end{center}
		\caption{The intersection of the isometric spheres $I(CBC^{-1})$ and $I(C)$ is a Giraud disk. The pink region is in the exterior of the isometric sphere  $I(C^{-1}BC^{-1})$. So the two pink colored sectors inside the circle is the ridge  $s(CBC^{-1}) \cap s(C)$.}
		\label{figure:CBinverseCandCinterinverseCBinverseC}
	\end{figure}

	Moreover, solve the equation of $\langle V,V \rangle=0$. We get $r=g_1(s)$ or $r=g_2(s)$, where $g_{i}$ has similar form as the function $f_{i}$ in Proposition \ref{prop:cBcandCBinverseCintersectionC}, but they are much  more involved. We omits the detailed terms of $g_{i}$ but only note that $g_{1}\left(\frac{\pi}{3}\right)=\arctan\left(\frac{5\sqrt{3}}{37}\right)$ and $g_{2}\left(-\frac{\pi}{3}\right)=-\arctan\left(\frac{5\sqrt{3}}{37}\right)$.

	We now take four arcs in the intersection of the isometric spheres of $I(CBC^{-1})$, $I(C)$  and $\partial  {\bf H}^2_{\mathbb{C}}$:
	
	\begin{itemize}
		
		\item
		$\mathcal{C}_{CBC^{-1}, C, 1}$ be the curve with $V=V(r,s)$,   where $r=g_1(s)$ and $$s \in \left[0, \arctan\left(\frac{\sqrt{-14+8\sqrt{7}}}{4-\sqrt{7}}\right)\right].$$
		
		\item  $\mathcal{C}_{CBC^{-1}, C, 2}$ be the curve with $V=V(r,s)$,   where $r=g_2(s)$ and $$s \in \left[\frac{\pi}{3}, \arctan\left(\frac{\sqrt{-14+8\sqrt{7}}}{4-\sqrt{7}}\right)\right].$$
		
		\item  $\mathcal{C}_{CBC^{-1}, C, 3}$ be the curve with  $V=V(r,s)$,   where $r=g_1(s)$ and $$s \in \left[- \arctan\left(\frac{\sqrt{-14+8\sqrt{7}}}{4-\sqrt{7}}\right),0\right].$$
		
		\item  $\mathcal{C}_{CBC^{-1}, C, 4}$ be the curve with $V=V(r,s)$,     where $r=g_2(s)$ and $$s \in \left[- \arctan\left(\frac{\sqrt{-14+8\sqrt{7}}}{4-\sqrt{7}}\right),-\frac{\pi}{3}\right].$$
		\end{itemize}	
	
		Then  $\mathcal{C}_{CBC^{-1}, C, 1}$ is an arc with one end point $CBC(u_3)$,  $\mathcal{C}_{CBC^{-1}, C, 2}$ is an arc with one end point $CBC(u_1)$.  $\mathcal{C}_{CBC^{-1}, C, 1}$ and  $\mathcal{C}_{CBC^{-1}, C, 2}$ have a common end point with $$(r,s)=\left(\arctan\left(\frac{(1+2\sqrt{7})\sqrt{-14+8\sqrt{7}}}{-8+11\sqrt{7}}\right),  \arctan\left(\frac{\sqrt{-14+8\sqrt{7}}}{4-\sqrt{7}}\right)\right).$$ That is the cyan solid square labeled point in  Figure 	\ref{figure:CBinverseCandCinterinverseCBinverseC}.
		Then the union of  $\mathcal{C}_{CBC^{-1}, C, 1}$ and  $\mathcal{C}_{CBC^{-1}, C, 2}$ is an arc connecting  $CBC(u_3)$ and $CBC(u_1)$.
		Similarly, the union of  $\mathcal{C}_{CBC^{-1}, C, 3}$ and  $\mathcal{C}_{CBC^{-1}, C, 4}$ is an arc connecting  $CBC(u_2)$ and $CBC(u_4)$.
		
		Moreover,  it can be showed that the two arcs $$\mathcal{C}_{CBC^{-1}, C, 1} \cup\mathcal{C}_{CBC^{-1}, C, 2}$$ and $$\mathcal{C}_{CBC^{-1}, C, 3} \cup\mathcal{C}_{CBC^{-1}, C, 4}$$  are exactly the part of the triple intersection  $I(CBC^{-1}) \cap I(C) \cap \partial {\bf H}^2_{\mathbb{C}}$ which are in the exterior of the isometric sphere  $I(C^{-1}BC^{-1})$.
		
\end{proof}

		We note that when $(r,s)=(\pi,\pi)$  in our parametrization   of $I(C^{-1}BC^{-1})\cap I(C)$, $(r,s)=(0,0)$  in our parametrization of $I(C^{-1}BC^{-1})\cap I(CBC^{-1})$, and  $(r,s)=(0,0)$  in our parametrization  of $I(CBC^{-1})\cap I(C)$, we get the same point in  ${\bf H}^2_{\mathbb{C}}$ with coordinates (up to sign):
			$$V_{C^{-1}BC^{-1}, CBC^{-1},C}=\left(\begin{array}{c}
		\frac{17}{16}-\frac{13\sqrt{15}\rm{i}}{16}\\[ 6 pt]
		\frac{5}{4}-\frac{\sqrt{15} \rm{i}}{4}	\\[ 6 pt]
		\frac{3}{4}+\frac{\sqrt{15} \rm{i}}{4}\\ \end{array}\right).$$

Moreover, consider our parametrization  of $I(C^{-1}BC^{-1})\cap  I(C)$ when $r= \pi$, $s \in [\pi, \frac{4 \pi}{3}]$. By direct calculation, this arc is identified with the arc  in our parametrization  of $I(C^{-1}BC^{-1})\cap  I(CBC^{-1})$ when $r=0$ and  $s \in [0, \frac{\pi}{3}]$, and it is identified with the arc  in our parametrization  of $I(CBC^{-1})\cap  I(C)$ when $r=0$ and $s\in [0, \frac{\pi}{3}]$. We denote  this arc by $[V_{C^{-1}BC^{-1}, CBC^{-1},C},CBC(u_1)]$. Similarly, the arcs in Figures 	\ref{figure:inverseCBinverseCandCintersectCBinverseC}, \ref{figure:cBcandCBcinterC}, \ref{figure:CBinverseCandCinterinverseCBinverseC} with center $V_{C^{-1}BC^{-1}, CBC^{-1},C}$ and  ends points  $CBC(u_2)$, $CBC(u_3)$, $CBC(u_4)$ are identified respectively.
So the union of the crossed  lines in Figures 	\ref{figure:inverseCBinverseCandCintersectCBinverseC},  \ref{figure:cBcandCBcinterC}, \ref{figure:CBinverseCandCinterinverseCBinverseC}   with ends points $CBC(u_1)$,  $CBC(u_2)$, $CBC(u_3)$, $CBC(u_4)$
 is exactly   $I(C^{-1}BC^{-1})\cap  I(C) \cap I(CBC^{-1})$.

	\begin{prop}\label{prop:3-sidecBc}	
	
	The side  $s(C^{-1}BC^{-1})$  is 3-ball in $ {\bf H}^2_{\mathbb{C}} \cup \partial {\bf H}^2_{\mathbb{C}}$. Moreover,   $s(C^{-1}BC^{-1}) \cap \partial {\bf H}^2_{\mathbb{C}}$ is a disk, $s(C^{-1}BC^{-1}) \cap  {\bf H}^2_{\mathbb{C}}$ is a disk consists of four sectors.
\end{prop}

\begin{proof}The side  $s(C^{-1}BC^{-1})$
	is contained in the isometric sphere  $I(C^{-1}BC^{-1})$. From above propositions,
	$s(C^{-1}BC^{-1})$
	intersects possibly with the sides $s(C)$ and  $s(CBC^{-1})$. From above calculations, the union of $s(C^{-1}BC^{-1})\cap s(C)$ and $s(C^{-1}BC^{-1})\cap s(CBC^{-1})$   is a disk $$(s(C^{-1}BC^{-1})\cap s(C)) \cup (s(C^{-1}BC^{-1})\cap s(CBC^{-1})).$$  The boundary of $(s(C^{-1}BC^{-1})\cap s(C)) \cup (s(C^{-1}BC^{-1})\cap s(CBC^{-1}))$ is a simple closed curve in the spinal sphere of  $I(C^{-1}BC^{-1})$, so it separates this spinal sphere into two disks. One of these two disks together with $(s(C^{-1}BC^{-1})\cap s(C)) \cup (s(C^{-1}BC^{-1})\cap s(CBC^{-1}))$ co-bound a 3-ball in $I(C^{-1}BC^{-1})$, which is the 3-side $s(C^{-1}BC^{-1})$.
	See Figure  \ref{figure:inverseCBinverseCridge} for the side $s(C^{-1}BC^{-1})$.
\end{proof}

	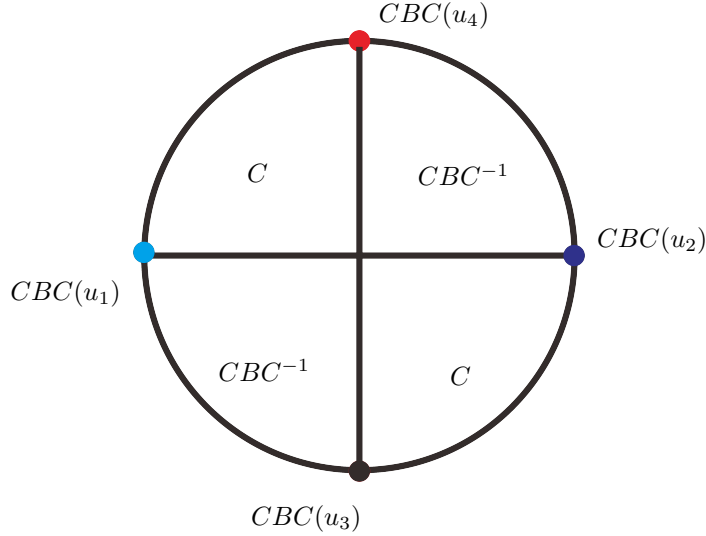
\begin{figure}
	\begin{center}
		\begin{tikzpicture}
		\node at (0,0) {\includegraphics[width=6cm,height=6cm]{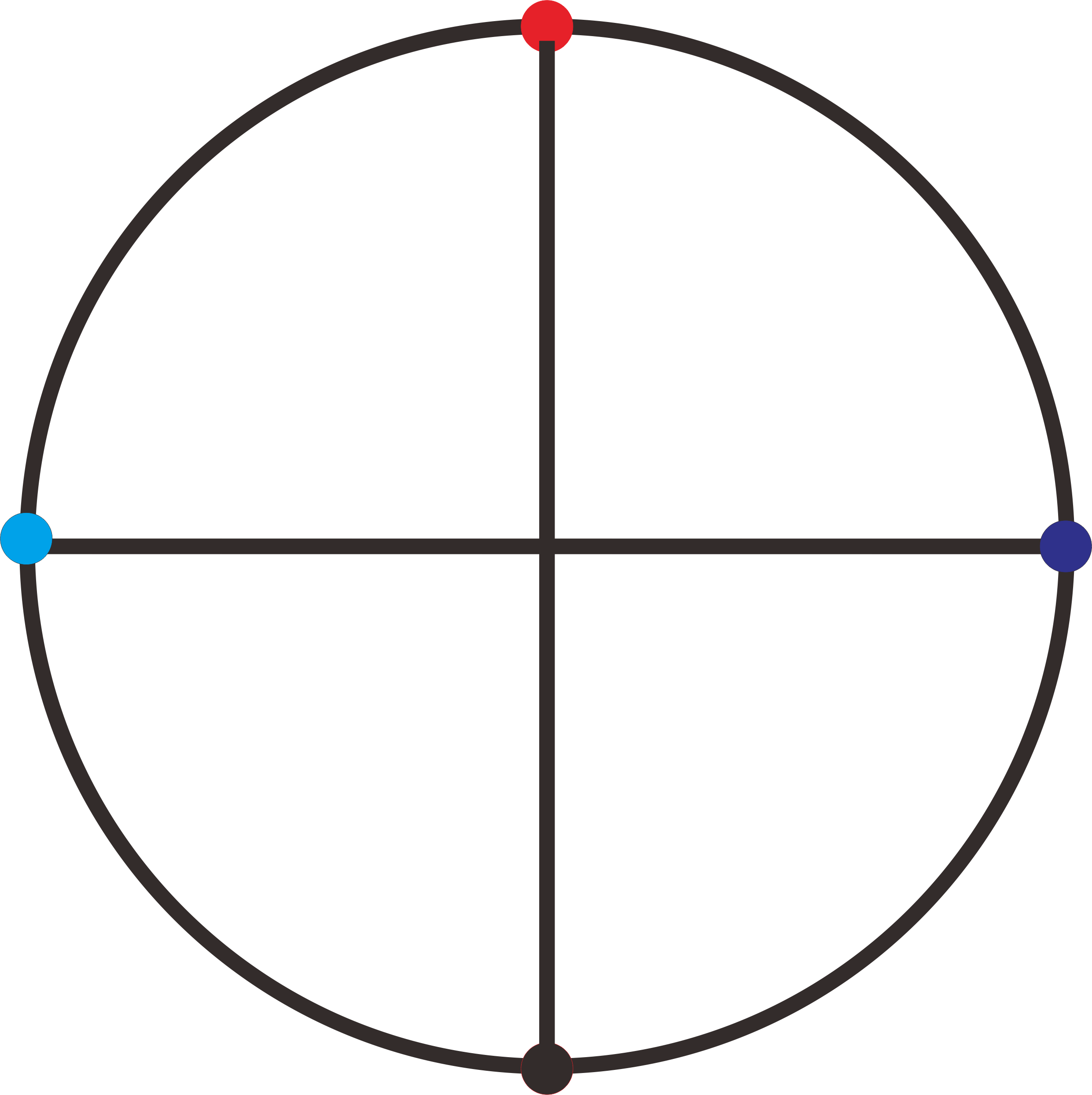}};

		\coordinate [label=below:$CBC(u_{1})$] (S)at(-3.9,-0.2);

		\coordinate [label=below:$CBC(u_{2})$] (S) at (3.9,0.5);

		\coordinate [label=below:$CBC(u_{3})$] (S) at (-0.7,-3.2);
		
		\coordinate [label=below:$CBC(u_{4})$] (S) at (1.0,3.5);
		
		\coordinate [label=below:$C$] (S)at(-1.35,1.35);
		\coordinate [label=below:$C$] (S)at(1.35,-1.35);
		\coordinate [label=below:$CBC^{-1}$] (S)at(1.39,1.35);
		\coordinate[label=below:$CBC^{-1}$](S)at(-1.259,-1.25);			
		\end{tikzpicture}
	\end{center}
	\caption{The  side  $s(C^{-1}BC^{-1})$ is a 3-ball in $I(C^{-1}BC^{-1})$, the outer side of the disk in  Figure \ref{figure:inverseCBinverseCridge} is the disk  $s(C^{-1}BC^{-1}) \cap \partial  {\bf H}^2_{\mathbb{C}}$.  The union of the two bisectors labeled by $C$ is the ridge $s(C^{-1}BC^{-1})\cap s(C)$. Similarly, the union of the two bisectors labeled by $CBC^{-1}$ is the ridge $s(C^{-1}BC^{-1})\cap s(CBC^{-1})$.}
	\label{figure:inverseCBinverseCridge}
\end{figure}

%Consider the $AC$-action on the side  $s(C^{-1}BC^{-1})$,  WE CAN NOT USE AC ACTION????

Similarly, we have

	\begin{prop}\label{prop:3-sideACBCa}	
	The side  $s(ACBCA^{-1})$  is 3-ball in $ {\bf H}^2_{\mathbb{C}} \cup \partial {\bf H}^2_{\mathbb{C}}$. Moreover,   $s(ACBCA^{-1}) \cap \partial {\bf H}^2_{\mathbb{C}}$ is a disk, $s(ACBCA^{-1}) \cap  {\bf H}^2_{\mathbb{C}}$ is a disk consists of four sectors.
\end{prop}

	\begin{prop}\label{prop:3-sideCBc}	
	
	The side  $s(CBC^{-1})$  is a solid light cone in $ {\bf H}^2_{\mathbb{C}} \cup \partial {\bf H}^2_{\mathbb{C}}$. Moreover,  $s(CBC^{-1}) \cap \partial {\bf H}^2_{\mathbb{C}}$ consists of two disjoint disks, the boundary of 
	$s(CBC^{-1}) \cap  {\bf H}^2_{\mathbb{C}}$ is a  light cone consisting of  four sectors.

\end{prop}

\begin{proof}The side  $s(CBC^{-1})$
	is contained in the isometric sphere  $I(CBC^{-1})$.	From above propositions,
we need only to consider the intersections of 	$s(CBC^{-1})$ with the sides $s(C)$ and  $s(C^{-1}BC^{-1})$. From above calculations, the union of $s(CBC^{-1})\cap s(C)$ and $s(CBC^{-1})\cap s(C^{-1}BC^{-1})$   are two disks  pinching at the point $V_{C^{-1}BC^{-1}, CBC^{-1},C}$. We denote these disks by $E_1$ and $E_2$ respectively.
	$\partial (E_1 \cup E_2)$  are two disjoint simple closed curves in the spinal sphere of $CBC^{-1}$, so they separate this spinal sphere into one annulus and two disks, say $D_1$ and $D_2$. We may assume $E_1 \cup D_1$ and  $E_2 \cup D_2$ are two spheres, so each of them bounds a 3-ball in the isometric sphere $I(CBC^{-1})$. These two 3-balls pinch at the point $V_{C^{-1}BC^{-1}, CBC^{-1},C}$, it is topologically a solid light cone.
	See Figure  	\ref{figure:CBinverseCridge} for the side $s(CBC^{-1})$.
\end{proof}

	\begin{figure}
	\begin{center}
		\begin{tikzpicture}
		\node at (0,0) {\includegraphics[width=12cm,height=6cm]{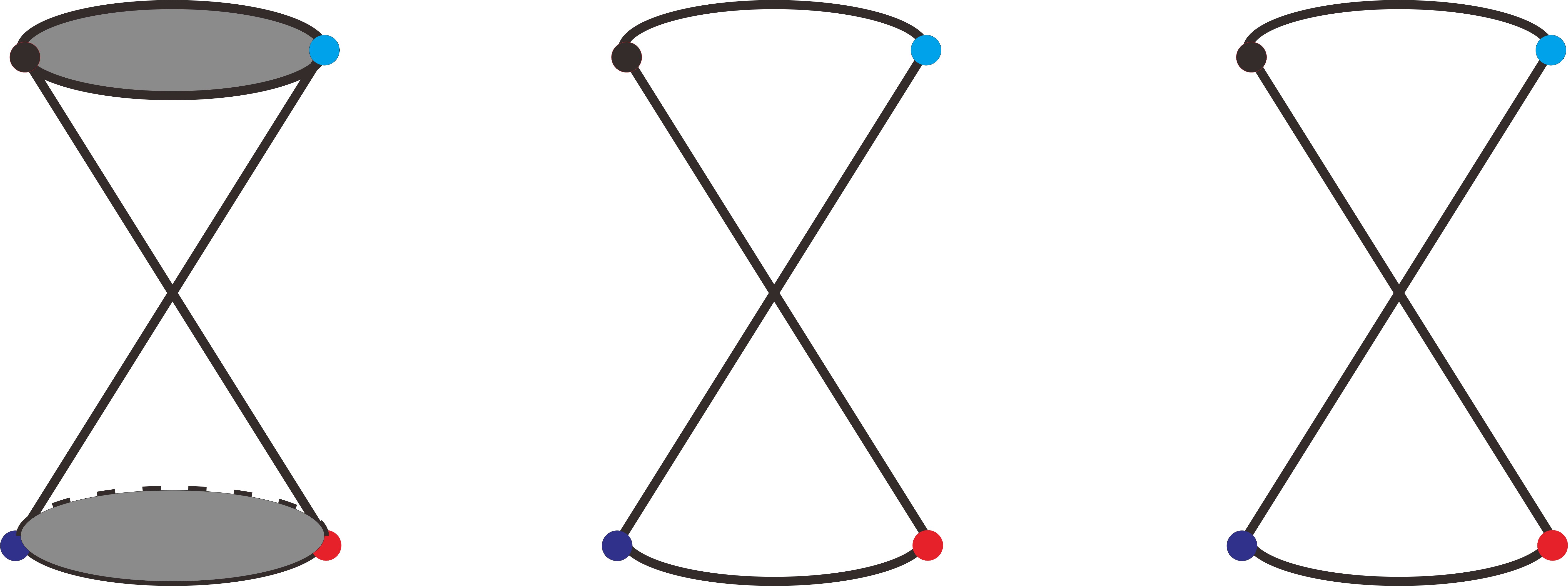}};
		
		\draw[->] (-3.4,2.6)--(-3,3.1);
		\coordinate [label=below:$CBC(u_{1})$] (S)at(-2.9,3.6);
		
		\draw[->] (-5.85,-2.8)--(-5.7,-3.3);
		\coordinate [label=below:$CBC(u_{2})$] (S) at (-5.1,-3.4);
		
		\draw[->] (-5.85,2.6)--(-5.7,3.1);
		\coordinate [label=below:$CBC(u_{3})$] (S) at (-5.1,3.6);
		
		\draw[->] (-3.35,-2.8)--(-2.5,-3.3);
		\coordinate [label=below:$CBC(u_{4})$] (S) at (-2.1,-3.4);
		
		\end{tikzpicture}
	\end{center}
	\caption{The left is the side $s(CBC^{-1})$, which is a solid lightcone, the two gray colored disks are $s(CBC^{-1}) \cap  \partial {\bf H}^2_{\mathbb{C}}$; the center is the ridge  $s(CBC^{-1})\cap s(C)$;  the right is the ridge $s(CBC^{-1})\cap s(C^{-1}BC^{-1})$. We omit the labels of $CBC(u_{i})$ for $i=1,2,3,4$ in the center and the right subfigures.}
	\label{figure:CBinverseCridge}
\end{figure}
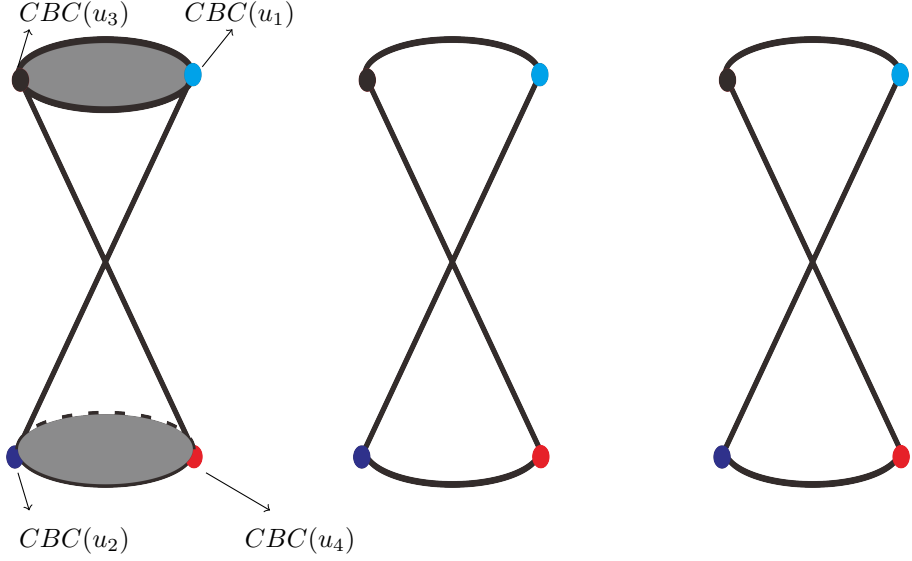

%Consider the $AC$-action on the side  $s(CBC^{-1})$, WE CAN NOT USE AC ACTION ???? 

Similarly, we have
\begin{prop}\label{prop:3-sideAcBCa}	
	The side  $s(AC^{-1}BCA^{-1})$  is a solid light cone in $ {\bf H}^2_{\mathbb{C}} \cup \partial {\bf H}^2_{\mathbb{C}}$. Moreover,  $s(AC^{-1}BCA^{-1}) \cap \partial {\bf H}^2_{\mathbb{C}}$ consists of two disjoint disks,
	$\partial s(AC^{-1}BCA^{-1}) \cap  {\bf H}^2_{\mathbb{C}}$ is a  light cone consisting of four sectors.	
\end{prop}

\begin{prop}\label{prop:3-sideC}
	
	The side  $s(C)$  is a 3-ball in $ {\bf H}^2_{\mathbb{C}} \cup \partial {\bf H}^2_{\mathbb{C}}$. Moreover,  $s(C) \cap \partial {\bf H}^2_{\mathbb{C}}$ is a 4-holed 2-sphere, and the boundary of  $s(C) \cap {\bf H}^2_{\mathbb{C}}$ is a disjoint union of four disks:
	
		\begin{itemize}
		
		\item  $s(C) \cap S(ACA^{-1})$ is a Giraud disk;

			\item  $s(C) \cap S(A^{-1}CA)$ is a Giraud disk;
			
			\item  $s(C) \cap (S(CBC^{-1})\cup S(C^{-1}BC^{-1}))$ is a union of four sectors;
			
				\item  $s(C) \cap (S(ACBCA^{-1})\cup S(AC^{-1}BCA^{-1}))$ is a union of four sectors.
			\end{itemize}
		\end{prop}
	\begin{proof}The side  $s(C)$
		is contained in the isometric sphere  $I(C)$.
			From above propositions,
		we need only  to consider the intersections of 	$s(C)$
	 with the sides $s(ACA^{-1})$, $s(A^{-1}CA)$,  $s(C^{-1}BC^{-1})$, $s(CBC^{-1})$, $s(ACBCA^{-1})$ and  $s(AC^{-1}BCA^{-1})$.
			From above calculations, the union of $s(ACA^{-1})\cap s(C)$ and $s(A^{-1}CA)\cap s(C)$
		are two disjoint Giraud disks, which are disjoint from other sides.
		The union of two ridges $s(C)\cap s(C^{-1}BC^{-1})$ and $s(C)\cap s(CBC^{-1})$ is a disk in $I(C) \cap {\bf H}^2_{\mathbb{C}}$. Similarly, the union of two ridges $s(C)\cap s(AC^{-1}BCA^{-1})$ and $s(C)\cap s(AC^{-1}BC^{-1}A^{-1})$ is  also a disk in $I(C) \cap {\bf H}^2_{\mathbb{C}}$.  These four disks are disjoint in $I(C)$, they together separate a 3-ball in $I(C)$, which is $s(C)$, and  $s(C) \cap \partial {\bf H}^2_{\mathbb{C}}$ is a 4-holed 2-sphere.
		See Figure \ref{figure:Cridge} for the side  $s(C)$.
	\end{proof}

%\begin{figure}
%	\begin{center}
%		\begin{tikzpicture}
%		\node at (0,0) %{\includegraphics[width=12.5cm,height=6cm]{Cridge.png}};
%		\coordinate [label=below:$A(u_{1})$] (S) at (3.6,0.5);
%		\coordinate [label=below:$A(u_{2})$] (S) at (0.7,2.7);
%		\coordinate [label=below:$A(u_{3})$] (S) at (0.7,0.5);
%		\coordinate [label=below:$A(u_{4})$] (S) at (3.6,2.7);
		
%		\coordinate [label=below:$CBC(u_{1})$] (S)at(-3.9,-0.2);
%		\coordinate [label=below:$CBC(u_{2})$] (S) at (-0.1,-2.4);
%		\coordinate [label=below:$CBC(u_{3})$] (S) at (-3.9,-2.4);
%		\coordinate [label=below:$CBC(u_{4})$] (S) at (-0.1,-0.2);
		
%		\draw[->] (2.99,1.55)--(1.3,3.2);
%		\draw[->] (1.7,1.55)--(1.3,3.2);
%		\coordinate [label=left:$AC^{-1}BCA^{-1}$] (S) at (2.0,3.5);
		
%		\draw[->] (2.4,2.4)--(3.3,3.2);
%		\draw[->] (2.3,0.75)--(3.3,3.2);
%		\coordinate [label=left:$ACBCA^{-1}$] (S) at (4.85,3.5);
%		
%		\draw[->] (-2.9,-1.55)--(-1,-3.2);
%		\draw[->] (-1.5,-1.55)--(-1,-3.2);
%		\coordinate [label=left:$CBC^{-1}$] (S) at (-0.0,-3.6);
%		
%		\draw[->] (-2.3,-2.3)--(-2.9,-3.2);
%		\draw[->] (-2.1,-0.8)--(-2.9,-3.2);
%		\coordinate [label=left:$ C^{-1}BC^{-1}$] (S) at (-2.5,-3.6);
%		\end{tikzpicture}
%	\end{center}
%	\caption{The side $s(C)$.  The left and the right round disks are the %ridges $s(C)\cap s(A^{-1}CA)$ and $s(C)\cap s(ACA^{-1})$ respectively. The %union of the two bisectors labeled by $CBC^{-1}$ is the ridge $s(C) \cap %s(CBC^{-1})$. Similarly we have the ridges $s(C) \cap s(C^{-1}BC^{-1})$, %$s(C) \cap s(AC^{-1}BCA^{-1})$ and $s(C) \cap s(ACBCA^{-1})$. }
%	\label{figure:Cridge}
%\end{figure}

\begin{figure}
	\begin{center}
		\begin{tikzpicture}
		\node at (0,0) {\includegraphics[width=12.5cm,height=6cm]{{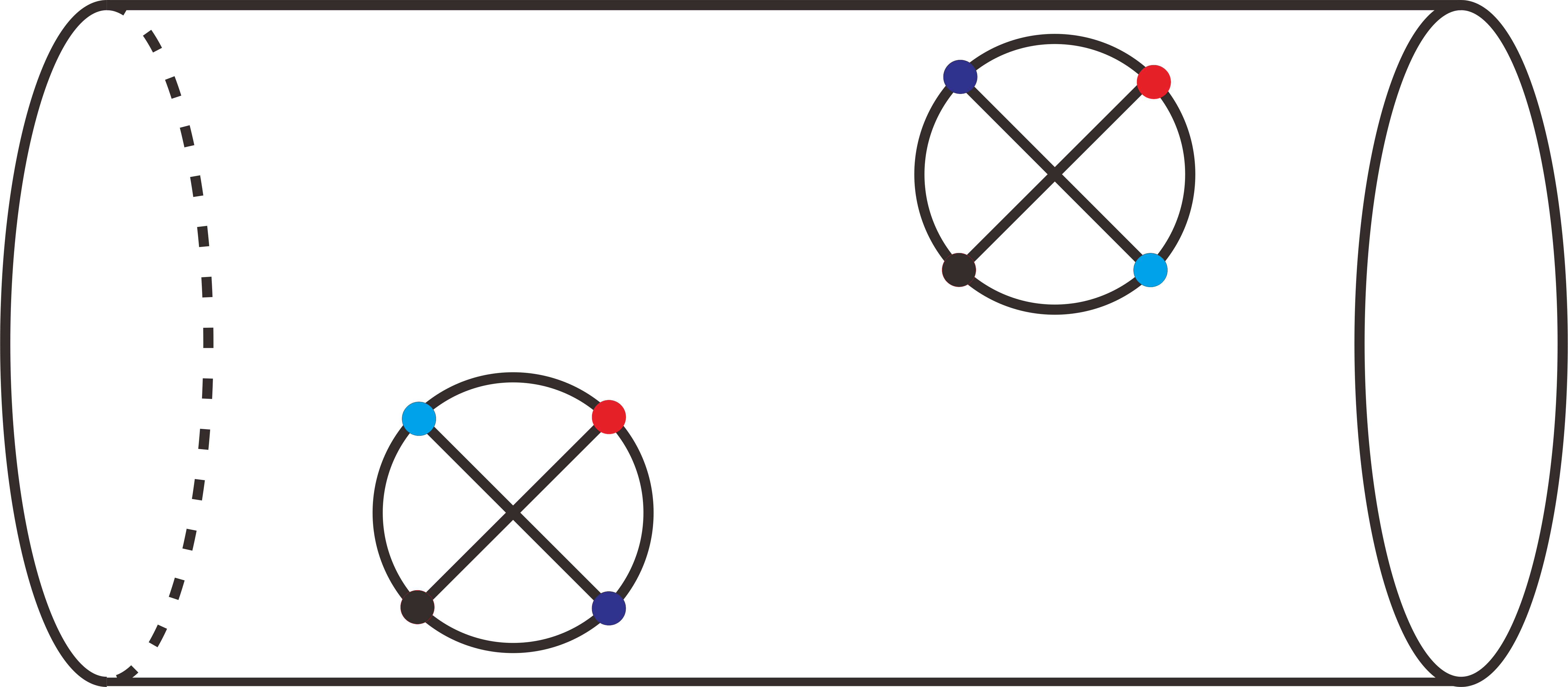}}};
		\coordinate [label=below:$A(u_{1})$] (S) at (3.6,0.5);
		\coordinate [label=below:$A(u_{2})$] (S) at (0.7,2.7);
		\coordinate [label=below:$A(u_{3})$] (S) at (0.7,0.5);
		\coordinate [label=below:$A(u_{4})$] (S) at (3.6,2.7);
		
		\coordinate [label=below:$CBC(u_{1})$] (S)at(-3.8,-0.2);
		\coordinate [label=below:$CBC(u_{2})$] (S) at (-0.1,-2.4);
		\coordinate [label=below:$CBC(u_{3})$] (S) at (-3.9,-2.4);
		\coordinate [label=below:$CBC(u_{4})$] (S) at (-0.1,-0.2);
		
		\draw[->] (2.99,1.55)--(1.3,3.2);
		\draw[->] (1.7,1.55)--(1.3,3.2);
		\coordinate [label=left:$AC^{-1}BCA^{-1}$] (S) at (2.0,3.5);
		
		\draw[->] (2.4,2.4)--(3.3,3.2);
		\draw[->] (2.3,0.75)--(3.3,3.2);
		\coordinate [label=left:$ACBCA^{-1}$] (S) at (4.85,3.5);
		
		\draw[->] (-2.9,-1.55)--(-1,-3.2);
		\draw[->] (-1.5,-1.55)--(-1,-3.2);
		\coordinate [label=left:$CBC^{-1}$] (S) at (-0.0,-3.6);
		
		\draw[->] (-2.3,-2.3)--(-2.9,-3.2);
		\draw[->] (-2.1,-0.8)--(-2.9,-3.2);
		\coordinate [label=left:$ C^{-1}BC^{-1}$] (S) at (-2.5,-3.6);
		\end{tikzpicture}
	\end{center}
	\caption{The side $s(C)$.  The left and the right round disks are the ridges $s(C)\cap s(A^{-1}CA)$ and $s(C)\cap s(ACA^{-1})$ respectively. The union of the two bisectors labeled by $CBC^{-1}$ is the ridge $s(C) \cap s(CBC^{-1})$. Similarly we have the ridges $s(C) \cap s(C^{-1}BC^{-1})$, $s(C) \cap s(AC^{-1}BCA^{-1})$ and $s(C) \cap s(ACBCA^{-1})$. }
	\label{figure:Cridge}
\end{figure}
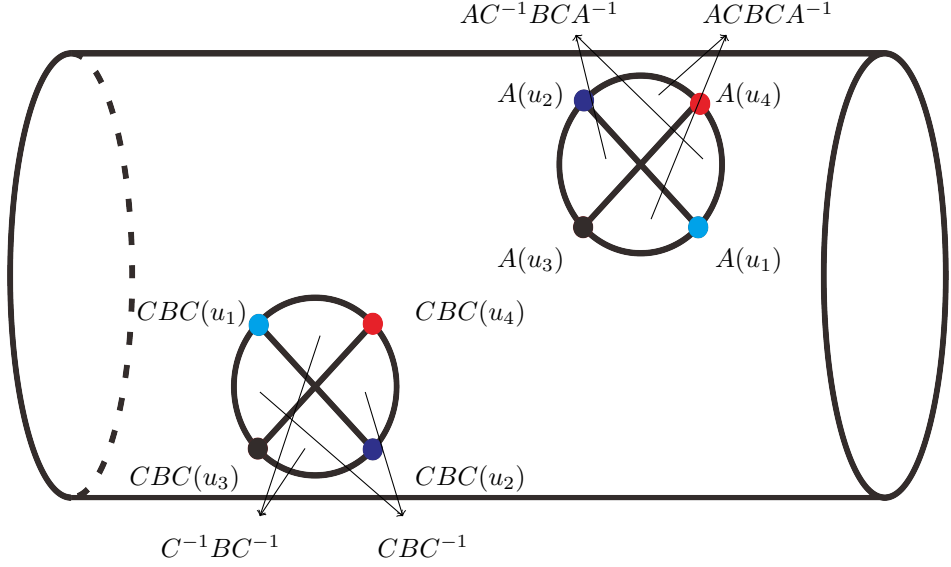

\subsection{Using the  Poincar\'e polyhedron theorem}\label{subsec:poincare2dim}
	Since $(AC)^2=id$, then   $I(A^{-1}CA)=I(C)$. We have the  side pairing maps as follows:
		\begin{itemize}
		
		\item   $A \cdot A^{k}CA^{-k}: s(A^{k}CA^{-k}) \rightarrow s(A^{k}CA^{-k})$;

			\item   $A^{k}CBCA^{-k}: s(A^{k}CBCA^{-k}) \rightarrow s(A^{k}C^{-1}BC^{-1}A^{-k})$;
			
				\item   $A^{k}C^{-1}BCA^{-k}: s(A^{k}C^{-1}BCA^{-k}) \rightarrow s(A^{k}C^{-1}BCA^{-k})$;
		
		\item   $A^{k}CBC^{-1}A^{-k}: s(A^{k}CBC^{-1}A^{-k}) \rightarrow s(A^{k}CBC^{-1}A^{-k})$.
		
			\end{itemize}

			\begin{prop}\label{prop:sidepairingC}
				
			The side pairing map  $A^{k+1}CA^{-k}$  is a self-homeomorphism of $s(A^{k}CA^{-k})$:
			\begin{enumerate}
				
				\item \label{item:Cridge1} $A^{k+1}CA^{-k}$ exchanges the ridges $s(A^{k}CA^{-k}) \cap s(A^{k-1}CA^{-k+1})$
			and   $s(A^{k}CA^{-k}) \cap s(A^{k+1}CA^{-k-1})$;	
					\item   \label{item:Cridge2} $A^{k+1}CA^{-k}$ sends the ridge $s(A^{k}CA^{-k}) \cap s(A^{k+1}CBCA^{-k-1})$
					to  the ridge $s(A^{k}CA^{-k}) \cap s(A^{k}CBC^{-1}A^{-k})$;
					
					\item  \label{item:Cridge2} $A^{k+1}CA^{-k}$ sends the  ridge $s(A^{k}CA^{-k}) \cap s(A^{k+1}C^{-1}BCA^{-k-1})$
					to the ridge $s(A^{k}CA^{-k}) \cap s(A^{k}C^{-1}BC^{-1}A^{-k})$.			
			\end{enumerate}

			\end{prop}

		\begin{proof}We just prove the case $k=0$, the other cases are similar.
			
			  The ridge $s(C) \cap s(ACA^{-1})$	is
			defined by the triple equality
		$$\langle z, q_{\infty}\rangle =\langle z, C^{-1}(q_{\infty})\rangle=\langle z, AC^{-1}A^{-1}(q_{\infty})\rangle.$$
From 	$AC$'s action  on the set 	$$\{q_{\infty}, C^{-1}(q_{\infty}), AC^{-1}A^{-1}(q_{\infty})\},$$ we get the set
	$$\{C^{-1}(q_{\infty}), q_{\infty},AC^{-1}A^{-1}(q_{\infty})\}.$$
So  $AC$ maps  $s(C) \cap s(ACA^{-1})$ to  $s(C) \cap s(A^{-1}CA)$.

The ridge $s(C) \cap s(ACBCA^{-1})$
is
defined by the triple equality
$$\langle z, q_{\infty}\rangle =\langle z, C^{-1}(q_{\infty})\rangle=\langle z, AC^{-1}BC^{-1}A^{-1}(q_{\infty})\rangle.$$
From $AC$'s action on the set 	$$\{q_{\infty},C^{-1}(q_{\infty}), AC^{-1}BC^{-1}A^{-1}(q_{\infty})\},$$ we get the set
$$\{C^{-1}(q_{\infty}), q_{\infty},CBC^{-1}(q_{\infty})\}.$$
So  $AC$ maps  $s(C) \cap s(ACBCA^{-1})$ to  $s(C) \cap s(CBC^{-1})$.

	The ridge $s(C) \cap s(AC^{-1}BCA^{-1})$
	is
	defined by the triple equality
	$$\langle z, q_{\infty}\rangle =\langle z, C^{-1}(q_{\infty})\rangle=\langle z, AC^{-1}BCA^{-1}(q_{\infty})\rangle.$$
	From $AC$'s action on the set 	$$\{q_{\infty},C^{-1}BCA^{-1}(q_{\infty}), AC^{-1}A^{-1}(q_{\infty})\},$$ we get the set
	$$\{C^{-1}(q_{\infty}), q_{\infty},CBC(q_{\infty})\}.$$
	So  $AC$ maps  $s(C) \cap s(AC^{-1}BCA^{-1})$ to  $s(C) \cap s(CBCA)$. The reader can compare to Figure 	\ref{figure:Cridge}.
	\end{proof}

Similarly, we have following three propositions, we omits the routine  proof.
	\begin{prop}\label{prop:sidepairingCBC}
	
	The side pairing map  $A^{k}CBCA^{-k}$  is a homeomorphism from $s(A^{k}CBCA^{-k})$ to $s(A^{k}C^{-1}BC^{-1}A^{-k})$:
	\begin{enumerate}
		
		\item \label{item:Cridge1}  $A^{k}CBCA^{-k}$ sends the ridge $s(A^{k}CBCA^{-k}) \cap s(A^{k}C^{-1}BCA^{-k})$
		to the ridge   $s(A^{k}C^{-1}BC^{-1}A^{-k}) \cap s(A^{k}CA^{-k})$;

		\item   \label{item:Cridge2} $A^{k}CBCA^{-k}$ sends the ridge $s(A^{k}CBCA^{-k}) \cap s(A^{k-1}CA^{-k+1})$
		to  the ridge $s(A^{k}C^{-1}BC^{-1}A^{-k}) \cap s(A^{k}CBC^{-1}A^{-k})$;
		
	\end{enumerate}
		
\end{prop}

\begin{prop}\label{prop:sidepairingcBC}
	
	The side pairing map  $A^{k}C^{-1}BCA^{-k}$  is a self-homeomorphism of $s(A^{k}C^{-1}BCA^{-k})$. Moreover, $A^{k}C^{-1}BCA^{-k}$  exchanges the ridges $s(A^{k}C^{-1}BCA^{-k}) \cap s(A^{k-1}CA^{-k+1})$
		and    $s(A^{k}C^{-1}BCA^{-k}) \cap s(A^{k}CBCA^{-k})$.
	
\end{prop}

\begin{prop}\label{prop:sidepairingCBc}
	
	The side pairing map  $A^{k}CBC^{-1}A^{-k}$  is a self-homeomorphism of $s(A^{k}CBC^{-1}A^{-k})$. Moreover, $A^{k}CBC^{-1}A^{-k}$  exchanges the ridges $s(A^{k}CBC^{-1}A^{-k}) \cap s(A^{k}CA^{-k})$
	and    $s(A^{k}CBC^{-1}A^{-k}) \cap s(A^{k}C^{-1}BC^{-1}A^{-k})$.
	
\end{prop}

{\bf Proof of Theorem  \ref{thm:2dimford}.} After above propositions, we show

{\bf Local tessellation.}   We prove the tessellation around the sides and ridges of $D_{\Gamma}$.

 First, for example, since $A^{k+1}CA^{-k}$  is a self-homeomorphism of $s(A^{k}CA^{-k})$, $A^{k+1}CA^{-k}$
sends the exterior of $I(A^{k+1}CA^{-k})$ to the interior of $I(A^{k+1}CA^{-k})$, we see that
$D_{\Gamma}$ and $A^{k+1}CA^{-k}(D_{\Gamma})$
 have disjoint interiors and cover a neighborhood  of
each point in the interior of the side $s(A^{k}CA^{-k})$.  The cases of the other 3-sides are similar.

Secondly,  we consider tessellation about ridges.

(1).  Since $s(A^{-1}CA)=s(C^{-1})$, for the ridge  $s(C) \cap s(A^{-1}CA)$,  the ridge circle is $$s(C) \cap s(A^{-1}CA) \xrightarrow{C}s(C) \cap s(A^{-1}CA).$$
Which gives the relation $C^3=id$.  By a standard argument as in \cite{ParkerWill:2017}, we have $D_{\Gamma} \cup C(D_{\Gamma}) \cup  C^{-1}(D_{\Gamma})$ will cover a small neighborhood of  $s(C) \cap s(A^{-1}CA)$.

(2). For the ridge  $s(C^{-1}BC) \cap s(C^{-1})$,  the ridge circle is
\begin{flalign}
\nonumber & s(C^{-1}BC) \cap s(C^{-1}) \xrightarrow{C^{-1}} s(C) \cap s(C^{-1}BC^{-1})\xrightarrow{C^{-1}BC^{-1}} s(CBC) \cap s(C^{-1}BC)
&  \\
&  \xrightarrow{C^{-1}BC} s(C^{-1}BC) \cap s(C).&\nonumber
\end{flalign}
Which gives the relation $ C^{-1}BC \cdot C^{-1}BC^{-1} \cdot C^{-1}=id$, that is $B^2=id$. We note that the ridge circle may have length 6 a priroi. There are two ways to show the ridge circle is length 3. The first way is taking a sample point $w$ with coordinates $$w=\left(\begin{array}{c}
\frac{\sqrt{15 \cdot 14}}{18}+ \frac{1}{3}
\\[ 6 pt]
\frac{2 \sqrt{15}}{9}+ \frac{\sqrt{14}}{6}\\[ 6 pt]
-\frac{17 \sqrt{15}}{18}+ 5\frac{\sqrt{14}}{9}\\ \end{array}\right)$$ in the Heisenberg group.
The point $w$ corresponds to $(r,s)=(0, -\arccos(\frac{5}{9}))$ in the parametrization of the intersection $I(C^{-1}BC)\cap I(C^{-1})$ in Proposition \ref{prop:CBcandCintersectioninverseCBinverseC}.
Then $$C^{-1}(w)=\left(\begin{array}{c}
\frac{\sqrt{15 \cdot 14}}{54}- \frac{1}{9}
\\[ 6 pt]
-\frac{8 \sqrt{15}}{27}+ \frac{\sqrt{14}}{18}\\[ 6 pt]
-\frac{59 \sqrt{15}}{54}-7\frac{\sqrt{14}}{9}\\ \end{array}\right) $$ in the Heisenberg group, and
$$C^{-1}BC(w)=\left(\begin{array}{c}
-\frac{\sqrt{15 \cdot 14}}{18}- \frac{1}{3}
\\[ 6 pt]
\frac{4 \sqrt{15}}{9}- \frac{\sqrt{14}}{6}\\[ 6 pt]
-\frac{25 \sqrt{15}}{18}-5\frac{\sqrt{14}}{9}\\ \end{array}\right)$$ in the Heisenberg group.
We can plot the points $w$, $C^{-1}(w)$ and $C^{-1}BC(w)$ in the boundary of our polytope $D_{\Gamma}$, see Figure 	\ref{figure:pointw}. From which, then
the ridge circle of  $s(C^{-1}BC) \cap s(C^{-1})$ has length three.

\begin{figure}
	\begin{center}
		\begin{tikzpicture}
		\node at (0,0) {\includegraphics[width=12cm,height=8cm]{{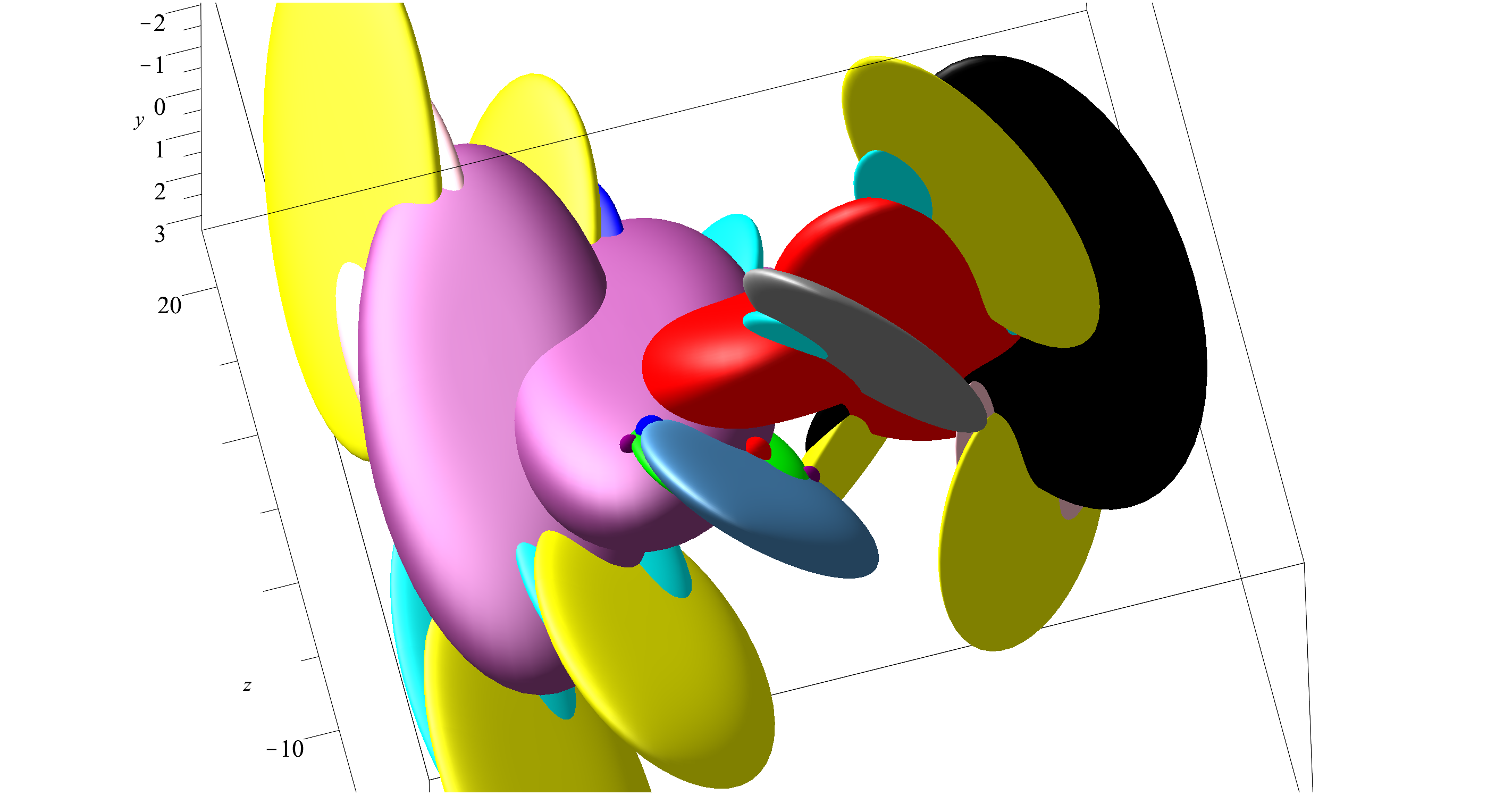}}};
		\end{tikzpicture}
	\end{center}
	\caption{One realistic view   of the ideal boundary of the Ford domain of 
	 $\rho_{(\sqrt{2},\operatorname{arccos}(-\frac{7}{8}))}(K) < \mathbf{PU}(2,1)$, which is the region outside all of the (big) spheres. The two big purple spheres are the spinal spheres of $A^{-2}CA^2$ and  $A^{-1}CA$. The two big red spheres are the spinal spheres of $C$ and  $ACA^{-1}$. The  big black sphere is the  spinal sphere of  $A^2CA^{-2}$. The  big steel-blue and gray spheres are  the  spinal spheres of $CBC$ and $ACBCA^{-1}$ respectively. The  big green and blue  spheres are  the  spinal spheres of $C^{-1}BC$ and $A^{-1}CBC^{-1}A$ respectively.   The points  $w$, $C^{-1}(w)$ and $C^{-1}BC(w)$ are the purple colored small balls in Figure  	\ref{figure:pointw}. The blue small ball is $u_2$, and  the red small ball is $u_4$. In this figure, we can only see  $w$ (near the small blue ball) and $C^{-1}BC(w)$ (near the small red ball).}
	\label{figure:pointw}
\end{figure}

The second way to show   the ridge circle of  $s(C^{-1}BC) \cap s(C^{-1})$ has length three is using all the parametrizations of the ridges $s(C^{-1}BC) \cap s(C^{-1})$, $ s(C) \cap s(C^{-1}BC^{-1})$ and $ s(CBC) \cap s(C^{-1}BC)$ in Propositions  \ref{prop:cBcandCintersecCBinverseC}, \ref{prop:cBcandCBinverseCintersectionC} and  \ref{prop:CBcandCintersectioninverseCBinverseC}, and calculating explicitly.

Then by a standard argument as in \cite{ParkerWill:2017}, we have $D_{\Gamma} \cup C(D_{\Gamma}) \cup  C^{-1}(D_{\Gamma})$ will cover a small neighborhood of  $s(C) \cap s(A^{-1}CA)$.

(3). For the ridge  $s(C^{-1}BC) \cap s(CBC)$,  the ridge circle is
\begin{flalign}
\nonumber & s(C^{-1}BC) \cap s(CBC) \xrightarrow{CBC} s(C^{-1}BC^{-1}) \cap s(C)\xrightarrow{C} s(C^{-1}) \cap s(C^{-1}BC)
&  \\
&  \xrightarrow{C^{-1}BC} s(C^{-1}BC) \cap s(CBC).&\nonumber
\end{flalign}
Which gives the relation $ CBC^{-1} \cdot C \cdot CBC=id$, that is $B^2=id$. %Take a parametrization of $I(C^{-1}BC) \cap I(CBC)$.
%$s_1=\arccos(5/9)$, $s_2=0$

%$v$ corresponds to $(s_1,s_2)=(arccos(\frac{5}{9},0))$ in the parameterization of the intersection $I(C^{-1}BC)\cap I(CBC)$ in Proposition ???
We also take a sample point $v$ with coordinates $$v=\left(\begin{array}{c}
\frac{\sqrt{15 \cdot 14}}{18}- \frac{1}{3}
\\[ 6 pt]
\frac{4\sqrt{15}}{9}+ \frac{\sqrt{14}}{6}\\[ 6 pt]
-\frac{25 \sqrt{15}}{18}+ 5\frac{\sqrt{14}}{9}\\ \end{array}\right)$$ in the Heisenberg group, which lies in  $I(C^{-1}BC) \cap I(CBC) \cap \partial  {\bf H}^2_{\mathbb{C}}$.
Then $$CBC(v)=\left(\begin{array}{c}
-\frac{\sqrt{15 \cdot 14}}{54}- \frac{1}{9}
\\[ 6 pt]
-\frac{8 \sqrt{15}}{27}- \frac{\sqrt{14}}{18}\\[ 6 pt]
-\frac{59 \sqrt{15}}{54}+7\frac{\sqrt{14}}{9}\\ \end{array}\right) $$ in the Heisenberg group,
and $$C^{-1}BC(v)=\left(\begin{array}{c}
-\frac{\sqrt{15 \cdot 14}}{18}+ \frac{1}{3}
\\[ 6 pt]
\frac{2 \sqrt{15}}{9}- \frac{\sqrt{14}}{6}\\[ 6 pt]
-\frac{17 \sqrt{15}}{18}-5\frac{\sqrt{14}}{9}\\ \end{array}\right)$$ in the Heisenberg group.
We can plot the points $v$, $CBC(v)$ and $C^{-1}BC(v)$ in the boundary of our polytope $D_{\Gamma}$, see Figure 	\ref{figure:pointv}. From which, then
the ridge circle of  $s(C^{-1}BC) \cap s(CBC)$ has length three.

The author also remarks that Figures 	\ref{figure:pointw} and 	\ref{figure:pointv} are the guides to get some  results in Subsections \ref{subsec:intersection} and \ref{subsec:ridge}. Figure 	\ref{figure:abstract}
later is a more explicit but abstract picture   of the ideal boundary of the Ford domain of 
$\rho_{(\sqrt{2},\operatorname{arccos}(-\frac{7}{8}))}(K) < \mathbf{PU}(2,1)$.

\begin{figure}
	\begin{center}
		\begin{tikzpicture}
		\node at (0,0) {\includegraphics[width=12cm,height=8cm]{{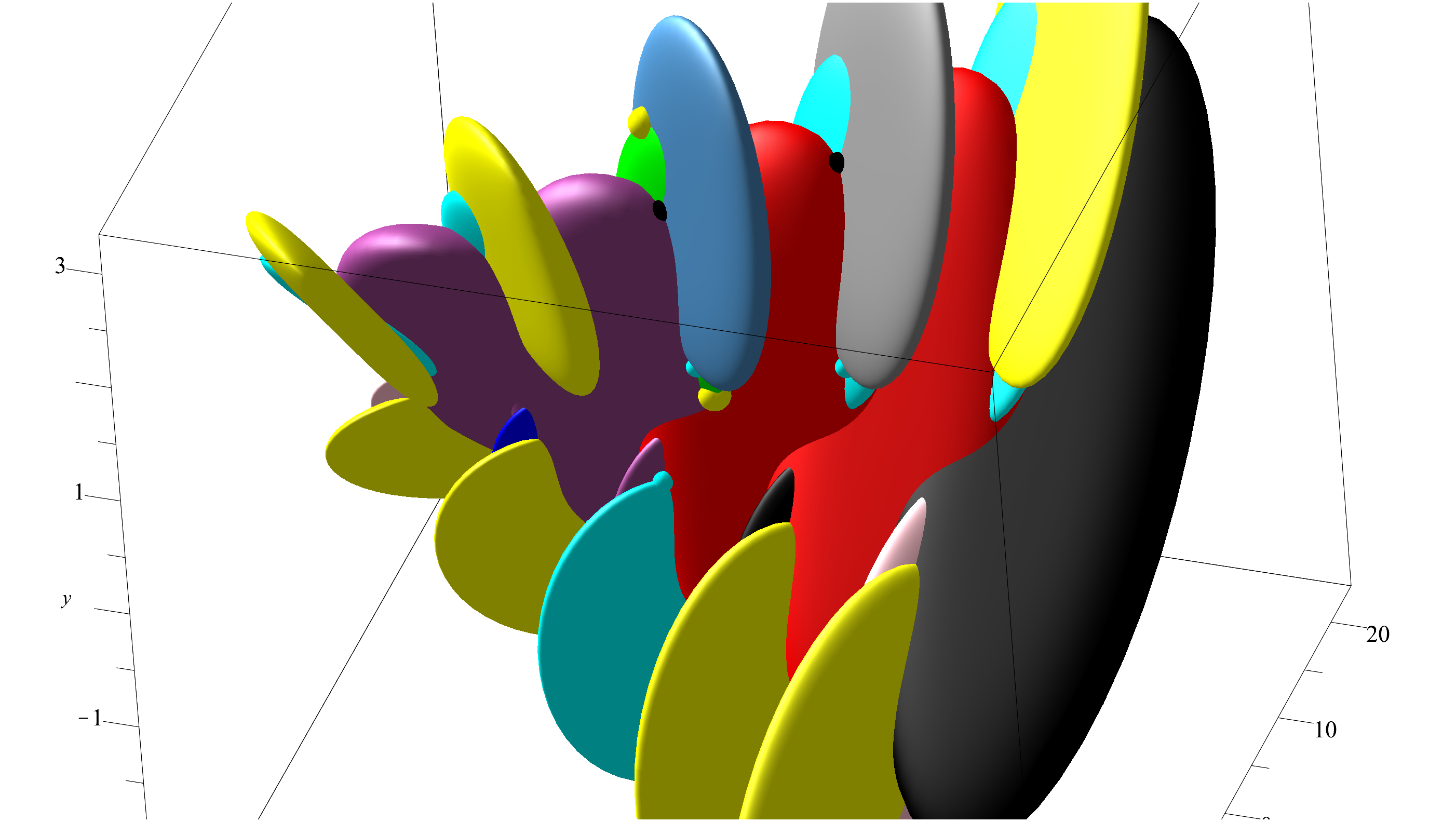}}};
		\end{tikzpicture}
	\end{center}
	\caption{Another realistic view   of the ideal boundary of the Ford domain of 
$\rho_{(\sqrt{2},\operatorname{arccos}(-\frac{7}{8}))}(K) < \mathbf{PU}(2,1)$, which is the region outside all of the (big) spheres. The two big purple spheres are the spinal spheres of $A^{-2}CA^2$ and  $A^{-1}CA$. The two big red spheres are the spinal spheres of $C$ and  $ACA^{-1}$. The  big black sphere is the  spinal sphere of  $A^2CA^{-2}$. The  big steel-blue and gray spheres are  the  spinal spheres of $CBC$ and $ACBCA^{-1}$ respectively. The  big green and blue  spheres are  the  spinal spheres of $C^{-1}BC$ and $A^{-1}CBC^{-1}A$ respectively. The points $v$, $CBC(v)$ and $C^{-1}BC(v)$  are the yellow colored small balls in Figure 	\ref{figure:pointv}. The black small balls are $u_3$ and $A(u_3)$, and  the cyan small balls is $u_1$, $CBC(u_1)$ and $A(u_1)$. In this figure, we can only see  $v$ (the top small yellow ball) and $C^{-1}BC(v)$ (the small yellow ball near to a small cyan ball).}
	\label{figure:pointv}
\end{figure}
By a standard argument as in \cite{ParkerWill:2017}, we have $D_{\Gamma} \cup C(D_{\Gamma}) \cup  C^{-1}(D_{\Gamma})$ will cover a small neighborhood of  $s(C) \cap s(A^{-1}CA)$.

(4). For the ridge  $s(CBC^{-1}) \cap s(C)$,  the ridge circle is
\begin{flalign}
\nonumber & s(CBC^{-1}) \cap s(C) \xrightarrow{C} s(C^{-1}) \cap s(CBC)\xrightarrow{CBC} s(C^{-1}BC^{-1}) \cap s(CBC^{-1})
&  \\
&  \xrightarrow{CBC^{-1}} s(CBC^{-1}) \cap s(C).&\nonumber
\end{flalign}
Which gives the relation $CBC^{-1} \cdot CBC \cdot C=id$, that is $B^2=id$.

By a standard argument as in \cite{ParkerWill:2017}, we have $D_{\Gamma} \cup C(D_{\Gamma}) \cup  C^{-1}(D_{\Gamma})$ will cover a small neighborhood of  $s(C) \cap s(A^{-1}CA)$.

(5). For the ridge  $s(C^{-1}BC^{-1}) \cap s(CBC^{-1})$,  the ridge circle is
\begin{flalign}
\nonumber & s(C^{-1}BC^{-1}) \cap s(CBC^{-1}) \xrightarrow{CBC^{-1}} s(CBC^{-1}) \cap s(C)\xrightarrow{C} s(C^{-1} \cap s(CBC)
&  \\
&  \xrightarrow{CBC} s(C^{-1}BC^{-1}) \cap s(CBC^{-1}).&\nonumber
\end{flalign}
Which gives the relation $C \cdot CBC^{-1} \cdot CBC \cdot C=id$, that is $B^2=id$.

By a standard argument as in \cite{ParkerWill:2017}, we have $D_{\Gamma} \cup C(D_{\Gamma}) \cup  C^{-1}(D_{\Gamma})$ will cover a small neighborhood of  $s(C) \cap s(A^{-1}CA)$.

%$$s(CBC^{-1}) \cap s(C) \xrightarrow{C} s(C^{-1}) \cap s(CBC)\xrightarrow{CBC} s(C^{-1}BC^{-1}) \cap s(CBC^{-1}) \xrightarrow{CBC^{-1}} s(C^{-1}BC^{-1}) \cap s(CBC^{-1}).$$

 (6). A hidden ridge.  Consider the circle $$ s(ACA^{-1}) \cap s(C) \xrightarrow{C} s(C^{-1}) \cap s(A^{-2}CA^2) \xrightarrow{A} s(C) \cap s(ACA^{-1}).$$
 Which means that $AC$ preserves $s(C)$-invariant as set but exchanges $s(A^{-1}CA) \cap s(C)$ and $s(ACA^{-1}) \cap s(C)$. In fact $AC$ preserves the isometric sphere $I(C)$ invariant but exchanges the exterior and the interior of it, so we have the relation $(AC)^2=id.$
	This ends the proof of Theorem \ref{thm:2dimford},  and  so the proof of  the first part of Theorem \ref{thm:3-mfd}.
% But note that $AC$ is $I_1I_2I_4I_1$, which is a $\mathbb{C}$-reflection about a point in ${\bf H}^2_{\mathbb{C}}$.  In particular, the $AC$-action on  $I(C)$
% is conjugate to the standard polar $\mathbb{Z}_2$-action on 3-ball?????

\section{3-manifold at infinity  of $\rho_{(\sqrt{2},\operatorname{arccos}(-\frac{7}{8}))}(K)<\mathbf{PU}(2,1)$}\label{sec:3mfd}

 %in $$R= \{A^{k}CA^{-k}, ~~~A^{k}CBCA^{-k},~~~A^{k}C^{-1}BC^{-1}A^{-k},~~~ A^{k}CBC^{-1}A^{-k},~~~A^{k}C^{-1}BCA^{-k}\}_{k \in \mathbb{Z}},$$

% $s(g)=I(g) \cap \overline {\bf H}^2_{\mathbb{C}}$

Based on results in Section \ref{sec:Ford2dim}, in particular, the combinatorial structures of ridge $s(g)$  for $g$ in the set  $R$, we study the 3-manifold  at infinity of  $\rho_{(\sqrt{2},\operatorname{arccos}(-\frac{7}{8}))}(K)<\mathbf{PU}(2,1)$ in this section.

 %the even subgroup $\langle A,B,C \rangle$ of the complex hyperbolic %reflection group  $\Gamma(\sqrt{2},\operatorname{arccos}(-\frac{7}{8}) < %\mathbf{PU}(2,1)$ in this section.

\begin{figure}
	\begin{center}
		\begin{tikzpicture}
		\node at (0,0) {\includegraphics[width=12.5cm,height=6cm]{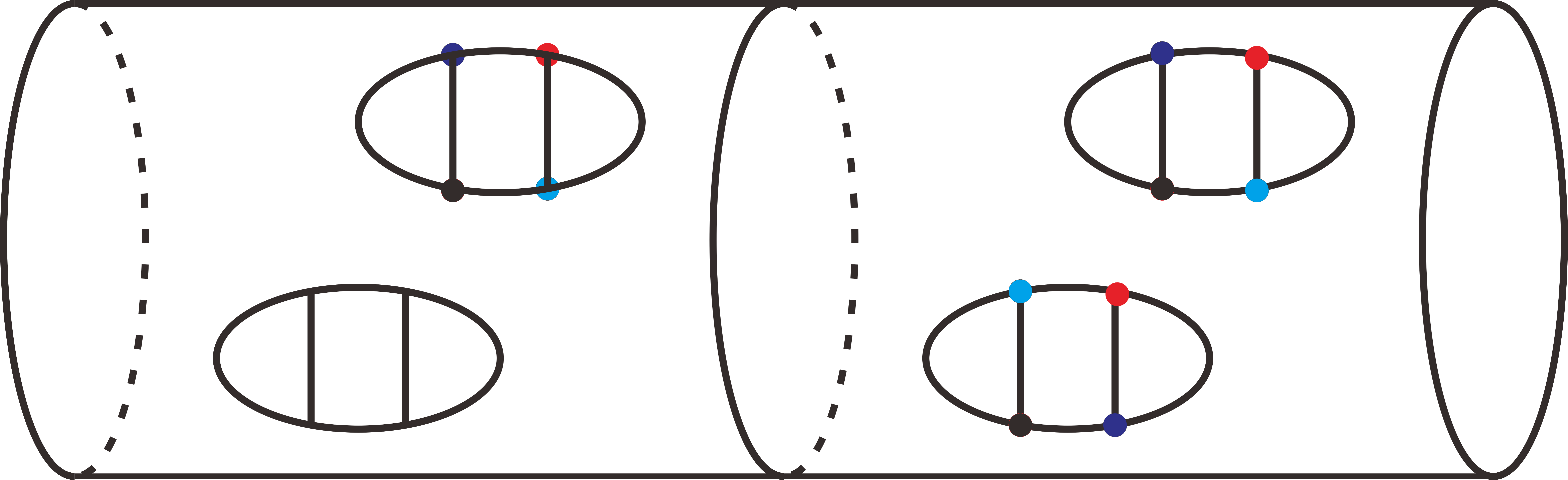}};
		
		\coordinate [label=below:$u_{1}$] (S) at (-1.35,0.5);
		\coordinate [label=below:$u_{2}$] (S) at (-3.05,2.7);
		\coordinate [label=below:$u_{3}$] (S) at (-3.05,0.5);
		\coordinate [label=below:$u_{4}$] (S) at (-1.35,2.7);
		
		\coordinate [label=below:$A(u_{1})$] (S) at (4.5,0.5);
		\coordinate [label=below:$A(u_{2})$] (S) at (2.3,2.7);
		\coordinate [label=below:$A(u_{3})$] (S) at (2.5,0.5);
		\coordinate [label=below:$A(u_{4})$] (S) at (4.6,2.7);
		
		\coordinate [label=below:$CBC(u_{1})$] (S)at(0.89,-0.2);
		\coordinate [label=below:$CBC(u_{2})$] (S) at (2.9,-2.4);
		\coordinate [label=below:$CBC(u_{3})$] (S) at (0.99,-2.4);
		\coordinate [label=below:$CBC(u_{4})$] (S) at (3.7,-0.2);

		\coordinate [label=below:$C$] (S) at (4.6,-2.2);

		\coordinate [label=left:$ C^{-1}$] (S) at (-0.8,-2.2);
		
		\draw[->] (-2.3,1.55)--(-1.6,3.2);
		\coordinate [label=left:$CBC$] (S) at (-1.0,3.5);
		
		\draw[->] (-3.1,1.55)--(-3.6,3.2);
		\draw[->] (-1.6,1.55)--(-3.6,3.2);
		\coordinate [label=left:$C^{-1}BC$] (S) at (-3.1,3.5);
		
		\draw[->] (3.3,1.55)--(2.3,3.2);
		\coordinate [label=left:$ACBCA^{-1}$] (S) at (3.0,3.5);
		
		\draw[->] (3.99,1.55)--(4.3,3.2);
		\draw[->] (2.7,1.55)--(4.3,3.2);
		\coordinate [label=left:$AC^{-1}BCA^{-1}$] (S) at (6.1,3.5);

		\draw[->] (2.9,-1.55)--(4.6,-3.2);
		\draw[->] (1.5,-1.55)--(4.6,-3.2);
		\coordinate [label=left:$CBC^{-1}$] (S) at (5.7,-3.6);
		
		\draw[->] (2.3,-1.55)--(1.9,-3.2);
		\coordinate [label=left:$ C^{-1}BC^{-1}$] (S) at (2.3,-3.6);
			\end{tikzpicture}
	\end{center}
	\caption{Part of the combinatorial picture of  the ideal boundary of the Ford domain of $\rho_{(\sqrt{2},\operatorname{arccos}(-\frac{7}{8}))}(K)<\mathbf{PU}(2,1)$, which is out side of an infinite annulus in $\mathbb{S}^{3}-q_{\infty}$.  What drawn in the figure is $A^{-1}(X_0) \cup X_0$, that is,  two  copies of a fundamental domain of $\rho_{(\sqrt{2},\operatorname{arccos}(-\frac{7}{8}))}(K)$ on the infinite annulus. }
	\label{figure:abstract}
\end{figure}
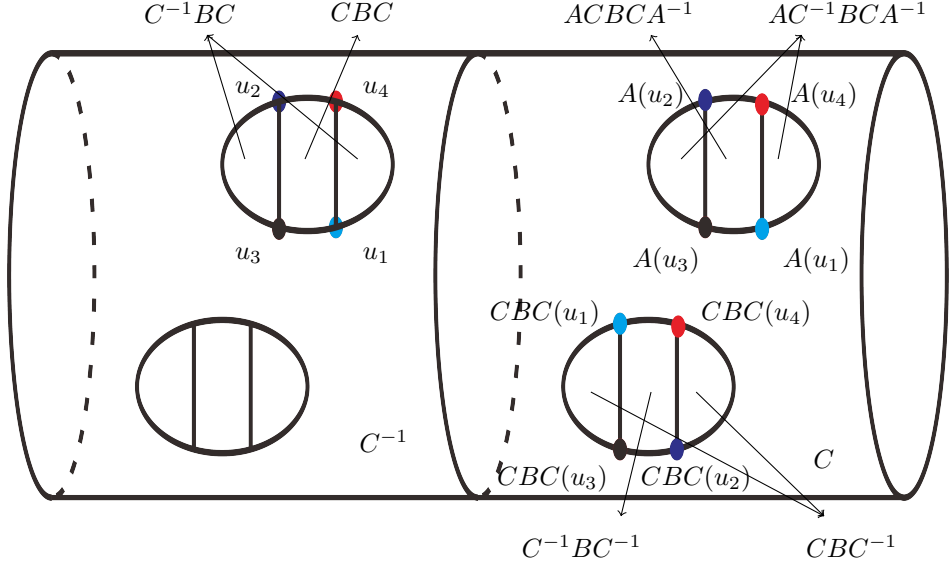

%$\mathcal{A}$
We denote  $\partial D_{R} \cap  \partial  {\bf H}^2_{\mathbb{C}}$ by  $\tilde{X}$, note that $$s(C) \cup s(CBC^{-1}) \cup s(C^{-1}BC^{-1}) \cap s(ACBCA^{-1}) \cup s(AC^{-1}BCA^{-1})) \cap  \partial  {\bf H}^2_{\mathbb{C}}$$ is an annulus by Propositions  \ref{prop:3-sidecBc}, \ref{prop:3-sideACBCa},  \ref{prop:3-sideCBc}, \ref{prop:3-sideAcBCa} and  \ref{prop:3-sideC}, we denote this annulus by $X_0$, then 
$$\tilde{X}= \cup_{k \in \mathbb{Z}} A^{k} (X_0).$$

\begin{prop} \label{prop:unknotted} The line 
	$$\mathcal{L}= \left\{[x+0 \cdot \rm{i}, -\frac{\sqrt{15}}{2}] ~~: ~~ x \in \mathbb{R}\right\}$$ in  $\partial {\bf H}^2_{\mathbb C}-\{q_{\infty}\}=\mathbb{C} \times \mathbb{R}$ is contained in the complement of $D_{R} \cap \partial  {\bf H}^2_{\mathbb C}$. So $\tilde{X}$ is an unknotted  infinite annulus in $\mathbb{S}^{3}-q_{\infty}$.
\end{prop}
\begin{proof} The fundamental domain of $A$  acting on $\mathcal{L}$ is a segment parameterized by
	$$\mathcal{L}_0=\{[x+0 \cdot \rm{i},0] \in \partial {\bf H}^2_{\mathbb C} ~~: ~~\ x\in [-1,1]\}.
	$$
	Note that a spinal sphere is convex. It is easy to check that the interval $\mathcal{L}_0$ is in the interior of the spinal sphere of $I(C)$.
\end{proof}

%\begin{prop}\label{prop:unknotted}
%	$\tilde{X}$ is an unknotted  infinite annulus in $\mathbb{S}^{3}-q_{\infty}$.
%\end{prop}
%\begin{proof}
%\end{proof}

See Figure	\ref{figure:abstract} for part of the combinatorial picture of  the ideal boundary of the Ford domain of $\rho_{(\sqrt{2},\operatorname{arccos}(-\frac{7}{8}))}(K)<\mathbf{PU}(2,1)$, which is out side of the infinite annulus $\tilde{X}$ in $\mathbb{S}^{3}-q_{\infty}$. This figure should be compared with Figures \ref{figure:pointw} and \ref{figure:pointv}.

%In Figure \ref{figure:abstract}, we only draw two  copies of a fundamental domain of $\rho_{(\sqrt{2},\operatorname{arccos}(-\frac{7}{8}))}(K)$ on the infinite annulus. 

 In the following,  for $$a,b \in \left\{u_{1},u_{2},u_{3},u_{4}, CBC(u_1), CBC(u_2), CBC(u_3),CBC(u_4)\right\}$$ and $x,y \in R$, we denote by  $[a,b]_{x,y}$ the edge component of $$I(x) \cap I(y) \cap \partial   {\bf H}^2_{\mathbb{C}}$$ whose end points  are $a$ and $b$, with orientation from $a$ to $b$. For example, the oriented  edge $[u_4,u_1)]_{C^{-1}BC, CBC}$ is the component of $$I(C^{-1}BC) \cap I(CBC) \cap \partial   {\bf H}^2_{\mathbb{C}}$$ with end points $u_4$ and $u_1$.

 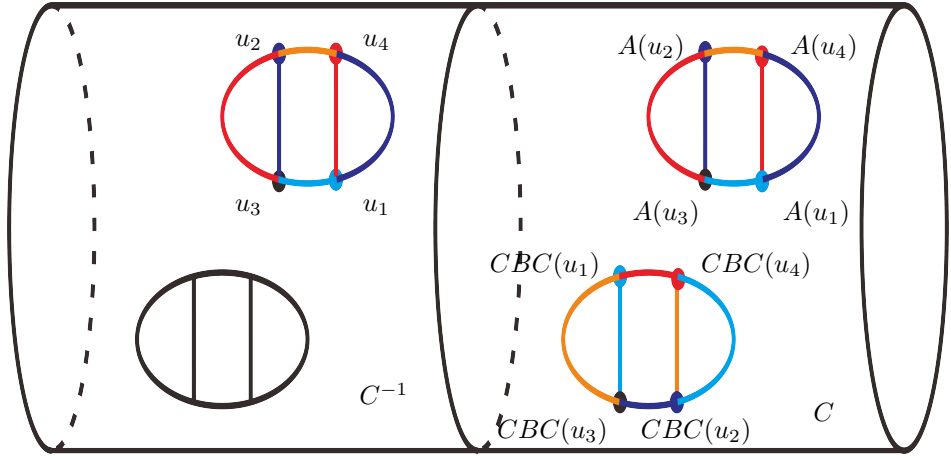
\begin{figure}
 	\begin{center}
 		\begin{tikzpicture}
 		\node at (0,0) {\includegraphics[width=12.5cm,height=6cm]{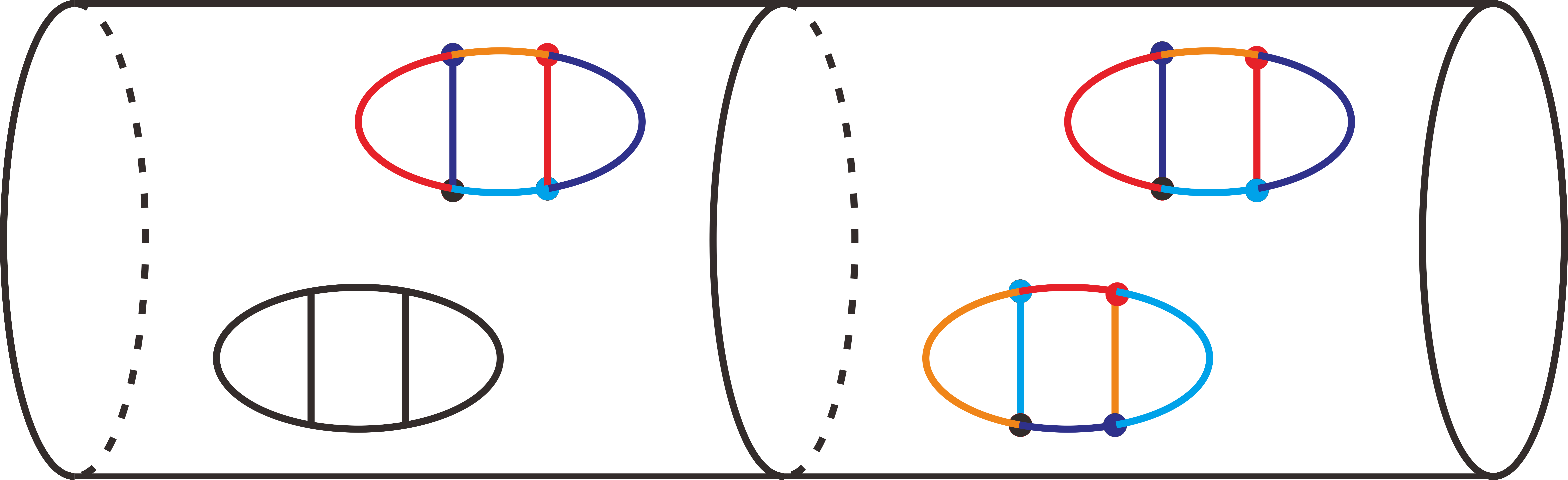}};
 		
 		\coordinate [label=below:$u_{1}$] (S) at (-1.35,0.5);
 		\coordinate [label=below:$u_{2}$] (S) at (-3.05,2.7);
 		\coordinate [label=below:$u_{3}$] (S) at (-3.05,0.5);
 		\coordinate [label=below:$u_{4}$] (S) at (-1.35,2.7);
 		
 		\coordinate [label=below:$A(u_{1})$] (S) at (4.5,0.5);
 		\coordinate [label=below:$A(u_{2})$] (S) at (2.3,2.7);
 		\coordinate [label=below:$A(u_{3})$] (S) at (2.5,0.5);
 		\coordinate [label=below:$A(u_{4})$] (S) at (4.6,2.7);
 		
 		\coordinate [label=below:$CBC(u_{1})$] (S)at(0.89,-0.2);
 		\coordinate [label=below:$CBC(u_{2})$] (S) at (2.9,-2.4);
 		\coordinate [label=below:$CBC(u_{3})$] (S) at (0.99,-2.4);
 		\coordinate [label=below:$CBC(u_{4})$] (S) at (3.7,-0.2);

 		\coordinate [label=below:$C$] (S) at (4.6,-2.2);

 		\coordinate [label=left:$ C^{-1}$] (S) at (-0.8,-2.2);

 		\end{tikzpicture}
 	\end{center}
 	\caption{The edge classes $x_1$, $x_2$, $x_3$ and  $x_4$  are drawn in red, blue, orange and cyan colors  respectively}
 	\label{figure:edgeclass}
 \end{figure}

These are some edges in Figure \ref{figure:abstract}, we consider their equivalent classes under $\rho_{(\sqrt{2},\operatorname{arccos}(-\frac{7}{8}))}(K)$-action:
\begin{itemize}
	\item
for the edge $[u_4,u_1)]_{C^{-1}BC, CBC}$,
from the ridge circle of $I(C^{-1}BC)\cap I(CBC)$ in Subsection \ref{subsec:poincare2dim},  we have the edge circle
\begin{flalign}
\nonumber & [u_4,u_1)]_{C^{-1}BC, CBC} \xrightarrow{CBC}  [CBC(u_4),CBC(u_1)]_{C^{-1}BC^{-1}, C} \xrightarrow{C}
&  \\
& [u_3,u_2]_{C^{-1}, C^{-1}BC}  \xrightarrow{C^{-1}BC}  [u_4,u_1]_{C^{-1}BC, CBC} .&\nonumber
\end{flalign}
 We denote this edge class by $x_1$ with orientation, note that the edge circle preserves the orientations of the edges.
	\item
for the edge $[u_2,u_3)]_{C^{-1}BC, CBC}$, from the ridge circle of $I(C^{-1}BC)\cap I(CBC)$ in Subsection \ref{subsec:poincare2dim}  we have the edge circle
 \begin{flalign}
\nonumber & [u_2,u_3)]_{C^{-1}BC, CBC} \xrightarrow{CBC}  [CBC(u_2),CBC(u_3)]_{C^{-1}, C^{-1}BC}   \xrightarrow{C} 
&  \\
& [u_1,u_4]_{C^{-1}, C^{-1}BC}\xrightarrow{C^{-1}BC} [u_2,u_3]_{C^{-1}BC, CBC} .&\nonumber
\end{flalign}
This class of oriented edges is denoted by $x_2$.
	\item
Similarly, for the edge $[CBC(u_3),CBC(u_1)]_{C^{-1}BC^{-1}, CBC^{-1}}$,
from the ridge circle of $I(C^{-1}BC^{-1})\cap I(CBC^{-1})$,  we have the edge circle
\begin{flalign}
\nonumber & [CBC(u_2),CBC(u_4)]_{C^{-1}BC^{-1}, CBC^{-1}} \xrightarrow{CBC^{-1}}  [CBC(u_1),CBC(u_3)]_{CBC^{-1}, C} \xrightarrow{C}
&  \\
& [u_2,u_4]_{C^{-1}, CBC}  \xrightarrow{CBC}  [CBC(u_2),CBC(u_4)]_{C^{-1}BC^{-1}, CBC^{-1}}.&\nonumber
\end{flalign}
This class of oriented edges is denoted by $x_3$:

%\begin{flalign}
%\nonumber & [CBC(u_3),CBC(u_1)]_{C^{-1}BC^{-1}, CBC^{-1}} %\xrightarrow{CBC^{-1}}  [CBC(u_4),CBC(u_2)]_{CBC^{-1}, C} \xrightarrow{C}
%&  \\
%& [u_3,u_1]_{C^{-1}, CBC}  \xrightarrow{CBC}  %[CBC(u_3),CBC(u_1)]_{C^{-1}BC^{-1}, CBC^{-1}} .&\nonumber
%\end{flalign}

%and

	\item
for  the edge 
 $[CBC(u_4),CBC(u_2)]_{CBC^{-1}, C}$, 
from the ridge circle of $I(CBC^{-1})\cap I(C)$ in subsection \ref{subsec:poincare2dim},  we have the edge circle
 \begin{flalign}
\nonumber &[CBC(u_4),CBC(u_2)]_{CBC^{-1}, C} \xrightarrow{C}  [u_3,u_1]_{C^{-1}, CBC} \xrightarrow{CBC}
&  \\
& [CBC(u_3),CBC(u_1)]_{C^{-1}BC^{-1}, CBC^{-1}}  \xrightarrow{CBC^{-1}}  [CBC(u_4),CBC(u_2)]_{CBC^{-1}, C}.&\nonumber
\end{flalign}
This class of oriented edges is denoted by $x_4$.
%\begin{flalign}
%\nonumber & [CBC(u_1),CBC(u_3)]_{CBC^{-1}, C} \xrightarrow{C}  [u_2,u_4]_{C^{-1}, CBC} \xrightarrow{CBC}
%&  \\
%& [CBC(u_2),CBC(u_4)]_{C^{-1}BC^{-1}, CBC^{-1}}  \xrightarrow{CBC^{-1}}  %[CBC(u_1),CBC(u_3)]_{CBC^{-1}, C} .&\nonumber
%\end{flalign}

%and
\end{itemize}

Since it is a little indistinct, so we re-color the edge classes $x_1$, $x_2$, $x_3$ and  $x_4$  in red, blue, orange and cyan colors  respectively in Figure 	\ref{figure:edgeclass}.

{\bf The proof of the second part of Theorem  \ref{thm:3-mfd}:}
From the above, we have two equivalent classes of vertices in $\tilde{X}$:
\begin{itemize}
	\item the class $U_1$ consists of all  $x(u_1)$ for  $x \in R$. For example, 
	since $C^{-1}(u_1)=u_2$, we have $u_2 \in U_1$. Moreover, $CBC(u_1), CBC(u_2) \in U_1$;
	\item  the class $U_3$ consists of all  $x(u_3)$ for  $x \in R$. For example, 
	since $C^{-1}(u_3)=u_4$, we have $u_4 \in U_3$. Moreover, $CBC(u_3), CBC(u_4) \in U_3$.
	\end{itemize} 

We also take a point $u_5$, which has coordinates $(r,s)=\left(\pi, \pi-\arccos(\frac{1}{4})\right)$ in our parametrization of $I(C) \cap I(C^{-1})$  in Proposition \ref{prop:cCintersection}. Then we have $u_5 \in  \partial {\bf H}^2_{\mathbb{C}}$, we  also have  two more points  $C(u_5)$ and $C^{-1}(u_5)$, see Figure \ref{figure:group}. We denote by $U_5$  the  classes of points $u_5$, $C(u_5)$ and $C^{-1}(u_5)$.   The points $u_5$, $C(u_5)$ and $C^{-1}(u_5)$ separate  the circle $I(C) \cap I(C^{-1}) \cap \partial  {\bf H}^2_{\mathbb{C}}$ into three edges. These three edges are equivalent under $C$-action, we denote this edge class by $x_5$.

Recall  $\tilde{X}=(\partial  D_{\Sigma} ) \cap  \partial   {\bf H}^2_{\mathbb{C}}$ and $\Sigma=\rho_{(\sqrt{2},\operatorname{arccos}(-\frac{7}{8}))}(K)$. 
There is a projection map $\Pi:  D_{\Sigma}  \cap \partial   {\bf H}^2_{\mathbb{C}} \rightarrow \tilde{X}$  which is equivalent with respect of the $\Sigma$-action.  We denote $X$ the quotient space of $ \tilde{X}$ with the equivalent relation by $\Sigma$-action. Let $M$ be the 3-manifold at infinity of $\Sigma$, that is, $M= \Omega/\Sigma$, where  $\Omega$ is  the set of discontinuity of the discrete subgroup  $\Sigma$ acting on $\partial \mathbf{H}^2_{\mathbb C}=\mathbb{S}^3$, $\Omega$ is tilled by $\Sigma$-copies of $D_{\Sigma}  \cap  \partial   {\bf H}^2_{\mathbb{C}}$.
  $M$ is also the quotient space of $D_{\Sigma}  \cap \partial   {\bf H}^2_{\mathbb{C}}$ with the equivalent relation by $\Sigma$-action. From  the projection map $\Pi$, then $X$ is a 2-spine of $M$, so we have $\pi_1(X)=\pi_1(M)$.

\begin{figure}
	\begin{center}
		\begin{tikzpicture}
		\node at (0,0) {\includegraphics[width=10cm,height=10cm]{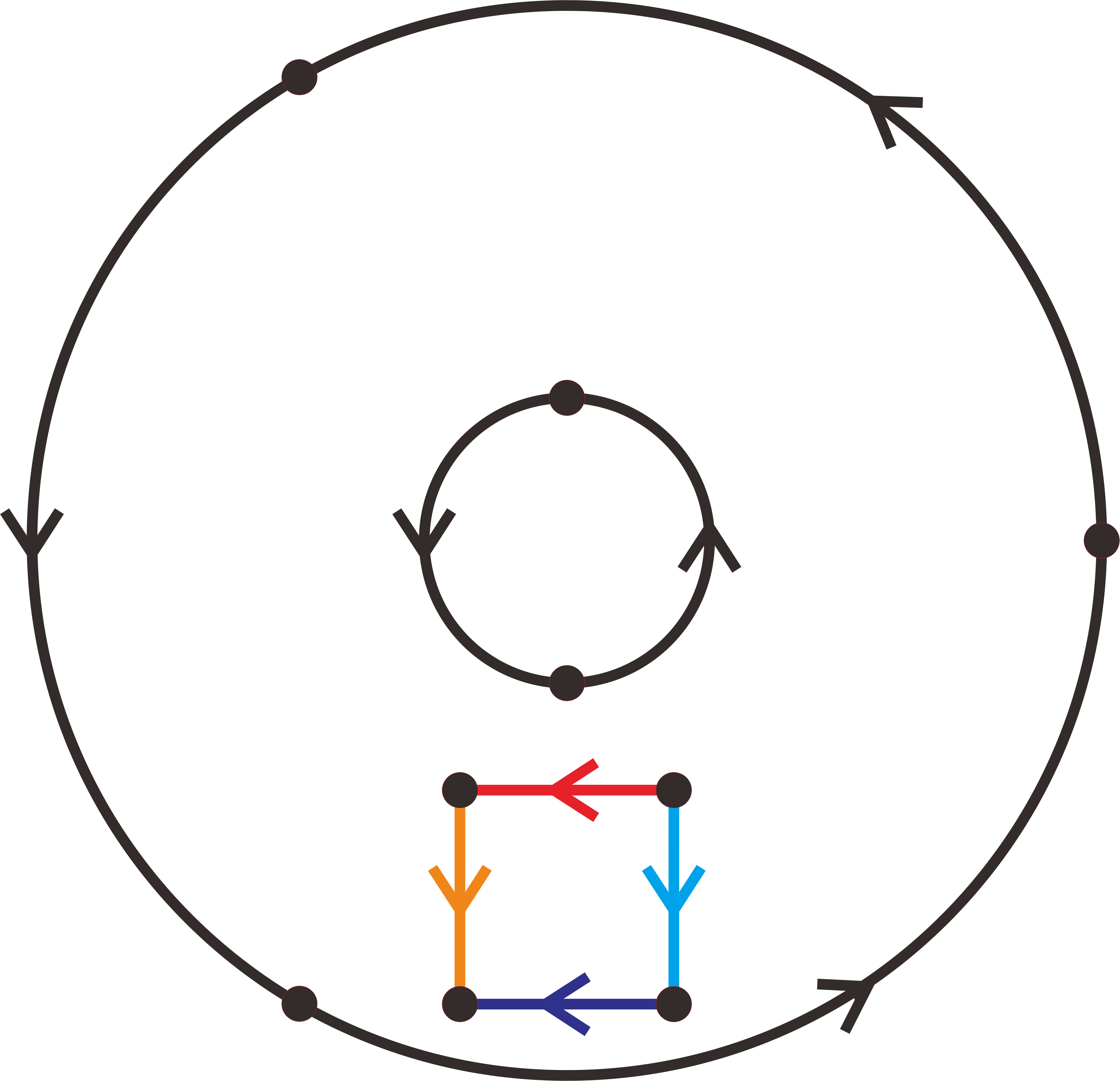}};
		
		\coordinate [label=below:$x_{5}$] (S) at (3.95,3.6);
			\coordinate [label=below:$x_{5}$] (S) at (-4.95,-0.4);
		
			\coordinate [label=below:$x_{5}$] (S) at (3,-4.6);

		\coordinate [label=below:$AC(u_{6})$] (S) at (-0.0,-0.4);
		\coordinate [label=below:$u_{6}$] (S) at (0,1.);
		
		\coordinate [label=below:$u_{5}$] (S) at (5.5,0.1);
		\coordinate [label=below:$C(u_{5})$] (S) at (-3.1,-4.0);
		\coordinate [label=below:$C^{-1}(u_{5})$] (S) at (-4.0,4.39);
		
			\coordinate [label=below:$CBC(u_{1})$] (S) at (-1.9,-1.7);
				\coordinate [label=below:$CBC(u_{2})$] (S) at (1.8,-3.6);
					\coordinate [label=below:$CBC(u_{4})$] (S) at (1.8,-1.7);
					\coordinate [label=below:$CBC(u_{3})$] (S) at (-1.9,-3.6);

		\coordinate [label=below:$x_6$] (S)at(1.79,-0.);
			\coordinate [label=below:$x_6$] (S)at(-1.79,-0.);
			\coordinate [label=below:$x_1$] (S)at(-0,-1.6);
				\coordinate [label=below:$x_2$] (S)at(-0,-4.5);
					\coordinate [label=below:$x_3$] (S)at(-1.3,-3.1);
						\coordinate [label=below:$x_4$] (S)at(1.4,-3.1);

		\end{tikzpicture}
	\end{center}
	\caption{The 3-holed 2-sphere from $s(C) \cap \partial  {\bf H}^2_{\mathbb{C}}$, which is one part of the  space $X$.}
	\label{figure:group}
\end{figure}
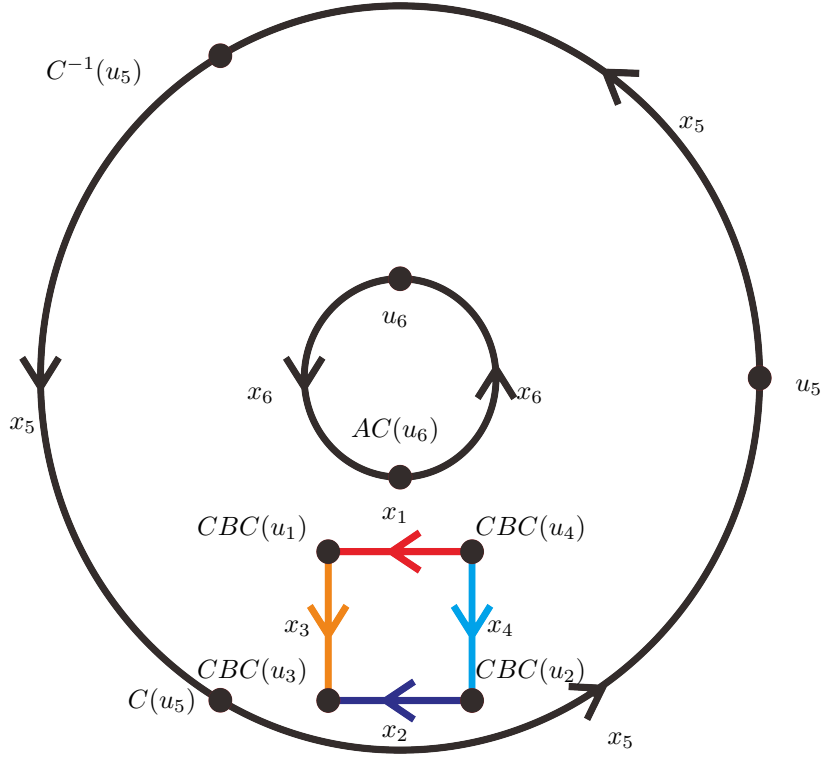

Consider $AC$-action on $s(C) \cap \partial {\bf H}^2_{\mathbb{C}}$. Note that $s(C) \cap \partial {\bf H}^2_{\mathbb{C}}$ is a 4-holed 2-sphere. The boundary of $s(C) \cap \partial {\bf H}^2_{\mathbb{C}}$ consists of four circles:

	\begin{enumerate}
	\item  $s(C) \cap s(ACA^{-1})\cap \partial {\bf H}^2_{\mathbb{C}}$;
	\item  $s(C) \cap s(A^{-1}CA)\cap \partial {\bf H}^2_{\mathbb{C}}$;
	
	\item  $s(C) \cap (s(CBC^{-1}) \cup s(C^{-1}BC^{-1}))\cap \partial {\bf H}^2_{\mathbb{C}}$, which is a union of four arcs in the edges classes $x_1,x_3,-x_2$ and $-x_4$ counterclockwise in Figures 	\ref{figure:abstract} and 	\ref{figure:edgeclass}. Here by $-x_2$ and $-x_4$ we mean the inverse orientations on $x_2$ and $x_4$;
		\item  $s(C) \cap (s(AC^{-1}BCA^{-1}) \cup s(ACBCA^{-1}))\cap \partial {\bf H}^2_{\mathbb{C}}$, which is also a union of four arcs  in the edges classes $-x_1,x_4,x_2$ and $-x_3$ counterclockwise in Figures 	\ref{figure:abstract} and 	\ref{figure:edgeclass}.
	\end{enumerate}	
$AC$ is an order two element in $\mathbf{PU}(2,1)$ which has exactly one fixed point in  ${\bf H}^2_{\mathbb{C}}$. Since $AC$ preserves the 2-sphere  $I(C)\cap \partial {\bf H}^2_{\mathbb{C}}$ invariant as a set, the quotient of  $I(C)\cap \partial {\bf H}^2_{\mathbb{C}}$ by $AC$  is the 2-dimensional real projective space $\mathbb{R}{\bf P}^{2}$.
$AC$ exchanges the first two boundary components of $s(C) \cap \partial {\bf H}^2_{\mathbb{C}}$  above, and $AC$ exchanges the third and the fourth  boundary components of $s(C) \cap \partial {\bf H}^2_{\mathbb{C}}$ above.  Take a simple closed curve $\mathcal{C}$  in $s(C)\cap  \partial  {\bf H}^2_{\mathbb{C}}$ which is $AC$-invariant and separates the first and the third boundary components of 
$s(C)\cap \partial {\bf H}^2_{\mathbb{C}}$ from the second  and the fourth  boundary components. Take a  point $u_6$ in  $\mathcal{C}$, then  $u_6$ and  $AC(u_6)$   separates
$\mathcal{C}$ into two oriented edges. These two edges are equivalent  by $AC$, we denote this edge class by $x_6$. We denote the vertex class of $u_6$ and  $AC(u_6)$ by $U_6$.

%Recall $H$ is the hidden ridge, $H \cap \partial  {\bf H}^2_{\mathbb{C}}$ is a circle. We take a point $u_6$ in $H \cap \partial  {\bf H}^2_{\mathbb{C}}$. For examples we may take $u_6=????$.  Then  $u_6$ and  $CA(u_6)$   separates
%$H \cap \partial  {\bf H}^2_{\mathbb{C}}$ into two edges. These three edges are equivalent, we denote this edge class by $x_6$. We denote the vertex class of $u_6$ and  $CA(u_6)$ by $U_6$.

\begin{figure}
	\begin{center}
		\begin{tikzpicture}
		\node at (0,0) {\includegraphics[width=10cm,height=7cm]{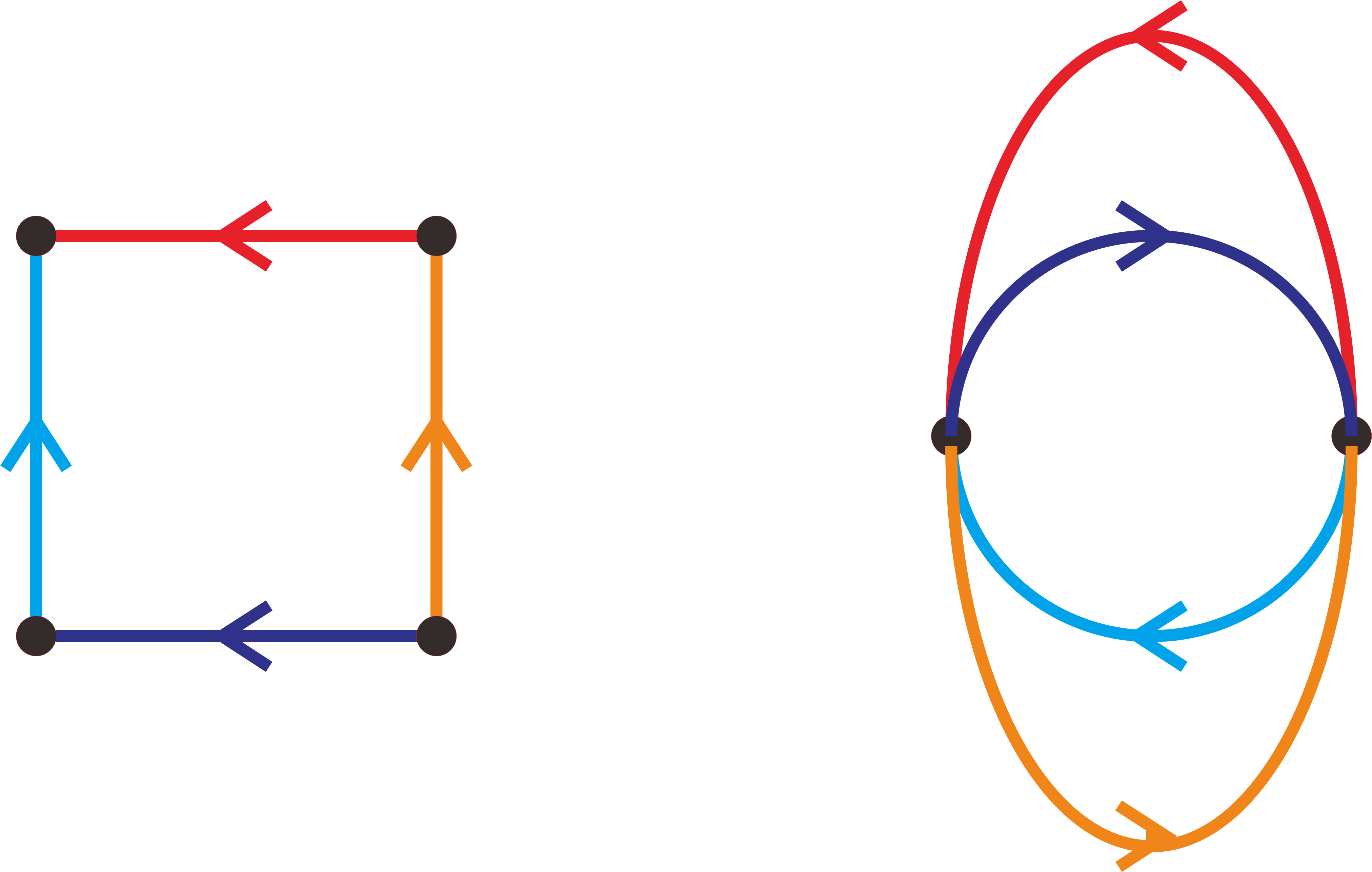}};

		\coordinate [label=below:$x_{4}$] (S) at (-5.37,0);
		\coordinate [label=below:$U_{1}$] (S) at (1.1,0);
		\coordinate [label=below:$x_{3}$] (S) at (-1.2,0);
		\coordinate [label=below:$U_{3}$] (S) at (5.37,0);

		\coordinate [label=below:$x_1$] (S)at(-3,2.2);
	
		\coordinate [label=below:$x_2$] (S)at(-3,-1.0);

			\coordinate [label=below:$x_2$] (S)at(3,1.4);
		\coordinate [label=below:$x_1$] (S)at(3,3.6);
		\coordinate [label=below:$x_4$] (S)at(3,-1.0);
		\coordinate [label=below:$x_3$] (S)at(3,-3.5);

			\coordinate [label=below:$C^{-1}BC^{-1}$] (S)at(-3.2,0.5);

			\coordinate [label=below:$C^{-1}BC$] (S)at(3.4,2.5);
			\coordinate [label=below:$CBC^{-1}$] (S)at(3.4,-2.2);

		\end{tikzpicture}
	\end{center}
	\caption{The left subfigure is the 2-disk $s(C^{-1}BC^{-1}) \cap  {\bf H}^2_{\mathbb{C}}$; The right subfigure is the union of two 2-disks, which are half of $(s(CBC^{-1}) \cup s(C^{-1}BC))\cap  {\bf H}^2_{\mathbb{C}}$ (we glue together the points $u_1$ and $u_2$, and denote it by $U_1$, we also glue together the points $u_3$ and $u_4$, and denote it by $U_3$).   The bigon labeled by $C^{-1}BC$ is  one  of the two disks $ s(C^{-1}BC)\cap  {\bf H}^2_{\mathbb{C}}$. Similarly,  the bigon labeled by $CBC^{-1}$ is  one  of the two disks $ s(CBC^{-1})\cap  {\bf H}^2_{\mathbb{C}}$. Those are  parts of the  space $X$.}
	\label{figure:homotopy}
\end{figure}

The topological space $X$ can be obtained from three parts by gluing:

	\begin{enumerate}
	\item half of  $s(C) \cap \partial {\bf H}^2_{\mathbb{C}}$, that is, a three holed sphere with boundaries $\mathcal{C}$,   $s(C) \cap s(A^{-1}CA)\cap \partial {\bf H}^2_{\mathbb{C}}$
	and   $s(C) \cap (s(CBC^{-1}) \cup s(C^{-1}BC^{-1}))\cap \partial {\bf H}^2_{\mathbb{C}}$. Note that half of  $s(C^{-1}) \cap \partial {\bf H}^2_{\mathbb{C}}$, that is, a three holed sphere with boundaries $C(\mathcal{C})$,   $s(C) \cap s(A^{-1}CA)\cap \partial {\bf H}^2_{\mathbb{C}}$
	and   $s(A^{-1}CA) \cap (s(C^{-1}BC) \cup s(CBC))\cap \partial {\bf H}^2_{\mathbb{C}}$ are identified with the first three holed sphere by $C$;
	
		\item a square $s(C^{-1}BC^{-1}) \cap \partial {\bf H}^2_{\mathbb{C}}$. Note that the disk  $s(C^{-1}BC^{-1}) \cap \partial {\bf H}^2_{\mathbb{C}}$ and  the disk $s(CBC) \cap \partial {\bf H}^2_{\mathbb{C}}$ are identified by $CBC$;	
			\item One of the two bigons  $s(CBC^{-1}) \cap \partial {\bf H}^2_{\mathbb{C}}$ and one of the two bigons   $s(C^{-1}BC) \cap \partial {\bf H}^2_{\mathbb{C}}$.  Note that two components of  $s(CBC^{-1}) \cap \partial {\bf H}^2_{\mathbb{C}}$ are exchanged by $CBC^{-1}$, and two components of  $s(C^{-1}BC) \cap \partial {\bf H}^2_{\mathbb{C}}$ are exchanged by $C^{-1}BC$.
		\end{enumerate}	 
So $X$ has four vertices:  $U_1$,  $U_3$, $U_5$ and  $U_6$. $X$ has six oriented edges:  $x_1$,  $x_2$, $x_3$,  $x_4$, $x_5$ and  $x_6$.  We add more 2-cells to get $X$.  See Figures \ref{figure:group} and \ref{figure:homotopy}.

Then it is easy to see
$$\pi_1(X)=\langle a,b,c |  a^2=b^3,c^2=id \rangle$$ with base point $U_3$.
Where $a$ is the homotopy class of $x_6$,  $b$ is the homotopy class of $x_5$, and  $c$ is the homotopy class of $x_1 \cdot x_3$.  The disk  $s(C^{-1}BC^{-1}) \cap \partial {\bf H}^2_{\mathbb{C}}$ implies $x_1 \cdot x^{-1}_4 \cdot x^{-1}_2 \cdot x_3$ is homotopical trivial. The bigons  $s(CBC^{-1}) \cap \partial {\bf H}^2_{\mathbb{C}}$ and    $s(C^{-1}BC) \cap \partial {\bf H}^2_{\mathbb{C}}$ imply $x_1 \cdot x_2$ and  $x_4 \cdot x_3$ are homotopical trivial.

It is also easy to see $M$ has exactly one torus cusp given by $q_{\infty}$.
Note that  the group $H$ with a presentation  $\langle a,b |  a^2=b^3 \rangle$ is the fundamental group of the trefoil knot complement 	 \cite{Rolfsen}, which is compact with exactly one torus boundary. Moreover, this manifold is the only manifold without 2-sphere as boundary up to homeomorphism with this fundamental group. $\mathbb{Z}/2\mathbb{Z}$ is the  fundamental group of the 3-dimensional  real projective space ${\mathbb R}{\mathbf P}^3$, and  ${\mathbb R}{\mathbf P}^3$ is also the only closed 3-manifold with fundamental group $\mathbb{Z}/2\mathbb{Z}$.
Since $\pi_1(M)$ is the free product of $H$  and $\mathbb{Z}/2\mathbb{Z}$.  Then it is standard that $M$ is the connected sum of two manifolds with fundamental groups $H$ and  $\mathbb{Z}/2\mathbb{Z}$ respectively, and  possibly  a homotopic 3-sphere	\cite{Hempel}.
By the solution of the  Poincar\'e conjecture, this homotopic 3-sphere is exactly the 3-sphere $\mathbb{S}^3$.
So $M$ is the connected sum of the trefoil knot complement 	and  the 3-dimensional  real projective space ${\mathbb R}{\mathbf P}^3$.
This ends the proof of  Theorem  \ref{thm:3-mfd}.

\section{The proof of Theorem \ref{thm:complex3dim}}\label{sec:complex3dim}
%In this section, we prove Theorem \ref{thm:complex3dim}

%\subsection{The combinatorics of ridges of  $D_{R}$ in  ${\bf H}^3_{\mathbb{C}}$}\label{subsec:3dimridge}

%We study carefully the  combinatorics of ridges of  $D_{R}$ in this subsection.
%Which are crucial for the application of the  Poincar\'e polyhedron theorem.
%We first take four points $u_i$ for $i=1,2,3,4$, which lie on the isometric spheres $I(C^{-1})$,  $I(CBC)$ and $I(CBC^{-1})$. See Figures \ref{figure:pointw}, \ref{figure:pointv} and  \ref{figure:abstract}.

First we have the  following lucky Lemma, which is a little surprising to the author.
  \begin{lemma} \label{lemma:coplane} For any $(h,t) $ in the moduli space $\mathcal{M}$, $$q_{\infty}-C^{-1}(q_{\infty})+CBC^{-1}(q_{\infty})-CBC(q_{\infty})$$
  	is the zero vector  in $\mathbb{C}^{3,1}$.
\end{lemma}
\begin{proof}This is proved by direct calculation, where 
	\begin{equation}\label{eq:2p2q34}
	C^{-1}(q_{\infty})=\left[
	\begin{array}{c}
	\frac{1}{4h^2}+1+\rm{e}^{t \rm{i}}\\[ 6 pt]
	\frac{1}{2}\\[ 6 pt]
\frac{\rm{e}^{t \rm{i}}\sqrt{4h^2 \cos(t)+3h^2+1}}{2h}\\[ 6 pt]
	-\frac{1}{2}
	\\
	\end{array}
	\right],
	\end{equation}

	\begin{equation}\label{eq:2p2q34}
	CBC^{-1}(q_{\infty})=\left[
	\begin{array}{c}
	\frac{1}{4h^2}+1+\rm{e}^{-t \rm{i}}+4h^2(1+\rm{e}^{t \rm{i}})\\[ 6 pt]
	\frac{1}{2}-2h^2(1+\rm{e}^{-t \rm{i}})\\[ 6 pt]
	\frac{(4h^2(1+\rm{e}^{t \rm{i}})+\rm{e}^{t \rm{i}})\sqrt{4h^2 \cos(t)+3h^2+1}}{2h}\\[ 6 pt]
	-\frac{1}{2}-4h^2-4h^2 \cos(t)
	\\
	\end{array}
	\right]
	\end{equation}
	
	and 
	
	\begin{equation}\label{eq:2p2q34}
	CBC(q_{\infty})=\left[
	\begin{array}{c}
	4h^2(1+\rm{e}^{t \rm{i}})+1-2 \sin(t) \cdot \rm{i}\\[ 6 pt]
-2h^2(1+\rm{e}^{-t \rm{i}})\\[ 6 pt]
2h(	\rm{e}^{t \rm{i}}+1)\sqrt{4h^2 \cos(t)+3h^2+1}\\[ 6 pt]
	-4h^2(1+ \cos(t))
	\\
	\end{array}
	\right].
	\end{equation}
	Then we have $q_{\infty}-C^{-1}(q_{\infty})+CBC^{-1}(q_{\infty})-CBC(q_{\infty})=0$.
\end{proof}

So for any $(h,t) \in \mathcal{M}$,  the span of  $q_{\infty}$, $C^{-1}(q_{\infty})$, $CBC^{-1}(q_{\infty})$ and  $CBC(q_{\infty})$ is a 3-dimensional subspace  of  $\mathbb{C}^{3,1}$, the intersection of ${\bf H}^{3}_{\mathbb{C}} \subset {\bf P}^{3}_{\mathbb{C}}$ with the projection of this 3-dimensional subspace  into ${\bf P}^{3}_{\mathbb{C}}$ is denoted by $L$ , then  $L$ is a totally geodesic ${\bf H}^{2}_{\mathbb{C}} \hookrightarrow {\bf H}^{3}_{\mathbb{C}}$.

  \begin{lemma} \label{lemma:intersection} For any $(h,t) \in \mathcal{M}$ with $h >1$, the intersection $I(C)\cap I(CBC^{-1})\cap I(C^{-1}BC^{-1}) \cap L$ are two crossed straight segments  in the Giraud disk  $I(C)\cap I(CBC^{-1}) \cap L$. So the 
  	intersection of the ridge
  	$s(C) \cap s(CBC^{-1})$ with $L$ are two  sectors with common vertices.
  	
\end{lemma}

\begin{proof}We  re-denote $q_{\infty}$, $C^{-1}(q_{\infty})$, $CBC^{-1}(q_{\infty})$ and  $CBC(q_{\infty})$ by $e_1$, $e_2$, $e_3$ and $e_4$.
	
	We denote $H_{L}$ by the matrix $(e^*_{i}He_{j})_{1 \leq i,j \leq 3}$, then 
	$$H_{L}=\left(\begin{array}{ccc}
	0&-\frac{1}{2}&-\frac{1}{2}-4h^2(\cos(t)+1)\\[ 6 pt]
	-\frac{1}{2}  & 0&-4h^2(\cos(t)+1)\\[ 6 pt]
	-\frac{1}{2}-4h^2(\cos(t)+1)& -4h^2(\cos(t)+1) & 0\\
	\end{array}\right). $$
So $H_{L}$ is the Hermitian form on the space with the basis $e_1$, $e_2$ and $e_3$.

 The vector $x_1 e_1+x_2 e_2+x_3e_3$ is denoted by the vector ${\bf x}$, and the vector $y_1 e_1+y_2 e_2+y_3e_3$  is denoted by the vector ${\bf y}$, here 
\begin{equation}\label{eq:2p2q12}
{\bf x}=\left[
\begin{array}{c}
x_1 \\
x_2 \\
x_3 \\
\end{array}
\right],\quad
{\bf y}=\left[
\begin{array}{c}
y_1\\
y_2\\
y_3 \\
\end{array}
\right]
\end{equation}
with $x_{i},y_{i} \in \mathbb{C}$.
So \begin{equation}\label{eq:2p2q12}
{ \bf E_1}=\left[
\begin{array}{c}
1 \\
0 \\
0 \\
\end{array}
\right],\quad
{ \bf E_2}=\left[
\begin{array}{c}
0\\
1\\
0 \\
\end{array}
\right],\quad
{ \bf E_3}=\left[
\begin{array}{c}
0\\
0\\
1 \\
\end{array}
\right]
\end{equation}
represent the vectors $e_1$, $e_2$ and $e_3$.

Comparing to Subsection \ref{subsection:Spinalcoordinates}, we define  the   Hermitian cross-product $\boxtimes_{L}$ with respect to $H_{L}$. 
$$\mathbf{x} \boxtimes_{L}
\mathbf{y}=
\left[\begin{matrix}
\mathbf{x}^*H_{L}(1,2) \cdot \mathbf{y}^*H_{L}(1,3)-\mathbf{y}^*H_{L}(1,2) \cdot \mathbf{x}^*H_{L}(1,3)\\ \mathbf{x}^*H_{L}(1,3) \cdot \mathbf{y}^*H_{L}(1,1)-\mathbf{y}^*H_{L}(1,3) \cdot \mathbf{x}^*H_{L}(1,1)\\ \mathbf{x}^*H_{L}(1,1) \cdot \mathbf{y}^*H_{L}(1,2)-\mathbf{y}^*H_{L}(1,1) \cdot \mathbf{x}^*H_{L}(1,2)
\end{matrix}
\right].$$
Here $\mathbf{x}^*H_{L}$ is a one-by-three matrix, so $\mathbf{x}^*H_{L}(1,2)$ is the second entry of   $\mathbf{x}^*H_{L}$.

 Then the intersection  $I(C)\cap I(CBC^{-1}) \cap L$ are parametrizaried  by $V=V(z_1,z_2)$ with $\langle V,V \rangle <0$.
Where
$$V=E_2\boxtimes_{L}E_3+z_1 \cdot E_1\boxtimes_{L}E_3+z_2 \cdot E_1\boxtimes_{L}E_2$$  and  $(z_1,z_2) =(\rm{e}^{r \rm{i}},\rm{e}^{s \rm{i}})\in S^{1}\times S^1$.

We note that 
\begin{equation}\label{eq:2p2q12}
E_2\boxtimes_{L}E_3=\left[
\begin{array}{c}
-16h^4(\cos(t)+1)^2 \\[ 6 pt]
2h^2(8h^2\cos(t)+8h^2+1)(\cos(t)+1) \\[ 6 pt]
2h^2(\cos(t)+1) \\
\end{array}
\right],
\end{equation}
\begin{equation}\label{eq:2p2q12}
E_1\boxtimes_{L}E_3=\left[
\begin{array}{c}
-2h^2(8h^2\cos(t)+8h^2+1)(\cos(t)+1) \\[ 6 pt]
 \frac{(8h^2\cos(t)+8h^2+1)^2}{4}\\[ 6 pt]
-\frac{1}{4}-2h^2-2h^2\cos(t) \\
\end{array}
\right],
\end{equation}
and 
\begin{equation}\label{eq:2p2q12}
E_1\boxtimes_{L}E_2=\left[
\begin{array}{c}
2h^2(\cos(t)+1)\\[ 6 pt]
\frac{1}{4}+2h^2+2h^2\cos(t) \\[ 6 pt]
-\frac{1}{4} \\
\end{array}
\right].
\end{equation}

Then $V(1,1) \cdot e_1+V(2,1) \cdot e_2+V(3,1) \cdot e_3$  is a vector in in $\mathbb{C}^{3,1}$, we also denote it  by $V$.
Now  $|\langle V, q_{\infty} \rangle|^2$ is 
$$ 4h^4(\cos(t)+1)^2(8\cos(t)h^2+8h^2+1)^2,$$ 
 $|\langle V, e_4 \rangle|^2$ is the product of 
$$ (4h^4(\cos(t)+1)^2(8\cos(t)h^2+8h^2+1)^2$$ and $$ 2\sin(r)\sin(s)+2\cos(s)\cos(r)+2\cos(s)+2\cos(r)+3.$$
(this is nontrivial even with Maple).
So the solutions of $|\langle V, q_{\infty} \rangle|^2=|\langle V, e_4 \rangle|^2$
are  $r=\pi$, $s=\pi$, or $r-s=\pi$.
 When $r=\pi$ and $s=0$, then for the corresponding $V$ we have  $\langle V, V \rangle$ is
 $$-128h^4(\cos(t)+1)^2(h^2 \cos(t)+h^2+\frac{1}{8}), $$ which is negative, so $I(C)\cap I(CBC^{-1}) \cap L$ is non-empty, then it is a Giraud disk. 
We note that when $s=\pi$, then for the corresponding $V$ we have $\langle V, V \rangle$ is the product of $$512h^2(1+ \cos(t))(h^2 \cos(t)+h^2+\frac{1}{8})^2$$ and   $$\frac{1-\cos(r)}{8}+h^2(1+ \cos(r))(1+\cos(t)), $$ so $\langle V, V \rangle$ is positive when $t \in [0, \pi)$. Then the solution $s=\pi$ is out side $L$.
We may assume $r \in [0,2 \pi]$ and $s \in [-\pi,\pi]$. So the solutions  $r=\pi$ and  $r=s$ are two crossed straight segments  in the Giraud disk.

\end{proof}

{\bf Proof of Theorem  \ref{thm:complex3dim}.}
With Lemma  \ref{lemma:intersection}, the proof of Theorem  \ref{thm:complex3dim} runs the same line as the proofs in Section \ref{sec:Ford2dim}.  We assume  $(h,t)\in \mathcal{M}$ is  close to $(\sqrt{2},\arccos(-\frac{7}{8}))$. By the continuity of the Cygan distance, results similar to
Propositions  \ref{prop:C}, \ref{prop:inverseCBinverseC}, \ref{prop:CBC},  \ref{prop:CBinverseC} and \ref{prop:inverseCBC}  also hold (view these isometric spheres in  ${\bf H}^{3}_{\mathbb{C}}$).  The intersection of $I(C)$ and $I(A^{-1}CA)$ is a non-empty 4-disk.  We have showed  
the stability of  the intersection of three isometric spheres  $I(C)$,  $I(CBC^{-1})$ and $I(C^{-1}BC^{-1})$ for the group $\rho_{(h, t)}(K)  <\mathbf{PU}(3,1)$ in ${\bf H}^{3}_{\mathbb C}$ in Lemma  \ref{lemma:intersection}, which is a union of two 3-balls which intersect in a 2-ball. So
results similar to
Propositions \ref{prop:cBcandCintersecCBinverseC},  \ref{prop:cBcandCBinverseCintersectionC} and \ref{prop:CBcandCintersectioninverseCBinverseC} hold for $\rho_{h,t}(K)$.  Then we have results similar to
Propositions \ref{prop:3-sidecBc}, \ref{prop:3-sideACBCa},  \ref{prop:3-sideCBc}, \ref{prop:3-sideAcBCa} and \ref{prop:3-sideC}.
Then we can use the Poincar\'e polyhedron theorem with the same set of side paring maps in Subsection \ref{subsec:poincare2dim} to prove Theorem  \ref{thm:complex3dim}.

%From Lemma \ref{lemma:coplane}, we take a matrix $M$ in $\mathbf{PU}(3,1)$, such that $M(q_{\infty})=q_{\infty}$, and the third entries of $MC^{-1}(q_{\infty})$,  $MCBC^{-1}(q_{\infty})$ and $MCBC(q_{\infty})$ are zeros.  Then we can  parametrize the intersections of three isometric spheres of $I(C)$, $I(CBC^{-1})$ and $I(C^{-1}BC^{-1})$ as in the case of ${\bf H}^{2}_{\mathbb{C}}$-geometry.
%More precisely, we consider the isometric spheres of $MCM^{-1}$, $MCBC^{-1}M^{-1}$ and $MCBC^{-1}M^{-1}$. We note that these elements do not preserve a ${\bf H}^{2}_{\mathbb{C}} \hookrightarrow {\bf H}^{3}_{\mathbb{C}}$ invariant, but it does not matter.  We also remark we should be careful when consider the parametrization  of the intersections of two isometric spheres  using Hermitian cross product, here we mean we use $q_{\infty}$,  $MC^{-1}(q_{\infty})$, $MCBC^{-1}(q_{\infty})$ and  $MCBC(q_{\infty})$ in $\mathbb{C}^{3,1}$, we can not normalize any of them by product a scalar.

 \section{Discussion on related topics}\label{sec:question}
 
 We propose some questions which are closed related to  results in this paper.

  \subsection{The interpolation  of geometries
 }\label{subsec:interpolation}
 
 We view the moduli space $\mathcal{M}$ as the interpolation between ${\bf H}^3_{\mathbb{R}}$-geometry 
 and  ${\bf H}^2_{\mathbb{C}}$-geometry by  ${\bf H}^3_{\mathbb{C}}$-geometry for the group $G$.
 Then as a  generalization  of  Theorem \ref{thm:3-mfd}, it seems the following is true.
 
 \begin{question}\label{question:2dimcomplexhyp}  Show that for any $h > \sqrt{2}$, when  $t=\operatorname{arccos}(-\frac{3h^2+1}{4h^2})$, $\rho_{(h,t)}$   is a discrete and faithful  representation of $G$ into $\mathbf{PU}(2,1)$. Moreover, the 3-manifold at infinity  of $\rho_{(h,t)}(K)$ is also the connected sum of the trefoil knot  complement in $\mathbb{S}^3$ and a real projective space  ${\mathbb R}{\mathbf P}^3$.
 \end{question}
 
  It seems the combinatorics of Ford domains in  Question \ref{question:2dimcomplexhyp} 
 are the same as in Theorem \ref{thm:3-mfd}, but the calculations are more  difficult.  It seems that when $h$ decreasing to $1.29326$ numerically, the combinatorics of the Ford domain   of $\rho_{(h,\operatorname{arccos}(-\frac{3h^2+1}{4h^2})) }$ change  for the first time. In other words, it seems  the combinatorics of Ford domains do not change  when $h \in [1.29326, \sqrt{2}]$, but the author can not prove this.

 %The author also notes that when $h$  converges to $1.01936$, the combinatorics of the Ford domain seem more complicated,  see Question \ref{question:2dim2parabolic}  later.  

%For example, it seems that when paramatrization intersection of three isometric spheres $I(C)\cap I(CBC^{-1})\cap I(C^{-1}BC^{-1})$, the solutions are $r=0$, $s=\pi$ or $r=s$ as in Propositions \ref{prop:cBcandCintersecCBinverseC}, \ref{prop:cBcandCBinverseCintersectionC} %and \ref{prop:CBcandCintersectioninverseCBinverseC}.

 \begin{question} \label{question:change} 	From  Theorem \ref{thm:complex3dim}, Theorem \ref{thm:3-mfd} and Question \ref{question:2dimcomplexhyp}, it is reasonable to guess that for all $h\in [\sqrt{2}, \infty)$ and $t \in [0, \operatorname{arccos}(-\frac{3h^2+1}{4h^2})]$, $\rho_{(h,t)}$  is a discrete and faithful  representation of $G$.
 	Moreover, 	 the combinatorics of  Ford domains of $\rho_{(h,t)}(K)$ change only finitely many times  in this region.  
 \end{question}

It is reasonable  to study the Ford domain of $\rho_{(h,\frac{\pi}{2})}(K)$ for some $h \in [\sqrt{2}, \infty)$ before study  Question \ref{question:change}.

\begin{question} \label{question:nearreal} When $h \in [\frac{1}{2}, \infty)$ is fixed and $t=0$ (we assume $h= \cos(\frac{\pi}{p})$ for some integer $p$ if $h <1$), $\Gamma$ preserves a totally geodesic ${\bf H}^3_{\mathbb R} \hookrightarrow {\bf H}^3_{\mathbb C}$ invariant,	 we have a discrete subgroup in $\mathbf{PO}(3,1)$.  The Ford domain of $\rho_{(h,0)}(K) < \mathbf{PO}(3,1)$ in ${\bf H}^3_{\mathbb R}$ is easy to study. 	It is interesting to know the relationship between the Ford domains of $\rho_{(h,0)}(K) < \mathbf{PO}(3,1)$ in ${\bf H}^3_{\mathbb R}$ and $\rho_{(h,0)}(K) < \mathbf{PU}(3,1)$ in ${\bf H}^3_{\mathbb C}$.

\end{question}

Recall that  $A=I_1I_2$, $B=I_3I_1$ and $C=I_4I_1$.
Let
$$S=\{A^{k}CA^{-k},A^{k}BA^{-k}\} _{k \in \mathbb{Z}}$$
be a subset of $\rho_{(h,0)}(K) < \mathbf{PO}(3,1)$, it is not difficult to show  the partial Ford domain 
$D_{S}(\rho_{(h,0)}(K))$ is in fact the Ford domain of  $\rho_{(h,0)}(K)$ in  ${\bf H}^3_{\mathbb R}$ if $h >1$.
If we 
can solve Question \ref{question:nearreal}. For this fixed $h$, when we increasing $t$, we can continue to study the Ford domain of $\rho_{(h,t)}(K) < \mathbf{PU}(3,1)$ when $t$ is small.   
 The author remarks by Page 297 of \cite{Go}, the Ford domains of $\rho_{(h,0)}(K) < \mathbf{PO}(3,1)$ in ${\bf H}^3_{\mathbb R}$ and $\rho_{(h,0)}(K) < \mathbf{PU}(3,1)$ in ${\bf H}^3_{\mathbb C}$ may have different combinatorial structures  a priori.  In other words,  Question \ref{question:nearreal} is also non trivial.

 \subsection{Discussion on the topology of  ${\bf H}^2_{\mathbb{C}} / \rho_{(\sqrt{2},\operatorname{arccos}(-\frac{7}{8}))}(K)$}\label{subsec:4mfd}
 
 The reader may wonder about  the topology of  ${\bf H}^2_{\mathbb{C}} / \rho_{(\sqrt{2},\operatorname{arccos}(-\frac{7}{8}))}(K)$ as a 4-orbifold.

 Let $N$ be the  4-orbifold with three singularities obtained as follows:
 \begin{itemize}
 	\item
 	take  a cone  on lens space $L(3,-1)$, which is in fact a 4-orbifold;
 	\item  take  two  copies of cone on real projective space ${\mathbb R}{\mathbf P}^3$, we get two 4-orbifolds;
 	\item  then  attaching two  1-handles to the above three 4-orbifolds to get a connected 4-orbifold.

 \end{itemize}
 
 %${\bf P}^3_{\mathbb{R}}$

 %The complex surface   ${\bf H}^2_{\mathbb{C}} /\Gamma(\sqrt{2},\operatorname{arccos}(-\frac{7}{8})$  can be obtained from  the neighborhoods of the three isolated singularities by attaching two  1-handles to connecting them.

 	Comparing to  \cite{Ma:2020a}, the author believes the following is true.
 \begin{question} \label{question:4mfd} 	
 	Show the complex  hyperbolic orbifold  ${\bf H}^2_{\mathbb{C}} / \rho_{(\sqrt{2},\operatorname{arccos}(-\frac{7}{8}))}(K)$ is homeomorphic to $N$.

 \end{question}

 The three guessed singularities in Question \ref{question:4mfd} are the quotients of the fixed points of $C$, $CBC^{-1}$ and $AC$   actions on ${\bf H}^2_{\mathbb{C}}$.   In \cite{Ma:2020a}, the author considered the topology of a more complicated ${\bf H}^2_{\mathbb{C}}$-orbifold. We believe the method in \cite{Ma:2020a} can solve Question \ref{question:4mfd}, but it is not trivial.

 % $$G=\langle \iota_1, \iota_2,  \iota_3,\iota_4 | (\iota_{2} \iota_{4})^2=(\iota_{4} \iota_{1})^3=(\iota_{1} \iota_{3})^2=id \rangle,$$
 %which 
 
 From another point of view,   the group $G$
 is the abstract group of an infinite area acute  quadrilaterial $Q$ in ${\bf H}^2_{\mathbb{R}}$. Where  $Q$  has four edges $l_1$, $l_2$, $l_3$ and $l_4$:
 \begin{itemize}
 	\item the pair of  edges $l_2$ and $l_3$  are hyper-parallel;
 	\item the pair of   edges $l_2$ and $l_4$ has  interior angle  $\frac{\pi}{2}$;
 	\item the pair of   edges $l_4$ and $l_1$ has  interior angle $\frac{\pi}{3}$;
 	\item the pair of   edges $l_1$ and $l_2$ has  interior angle $\frac{\pi}{2}$.
 \end{itemize}
 Then $K$ is  the abstract group of the index two subgroup of $G$ generated by $\iota_1\iota_2$, $\iota_3\iota_1$ and $\iota_4\iota_1$.
 
 From the canonical Teichm\"uller theory, there are two dimensional discrete and faithful representations of $G$ into $\mathbf{PO}(2,1) <\mathbf{PU}(2,1)$.
 But there are four dimensional representations of $G$ into $\mathbf{PU}(2,1)$.
 Take any  discrete and faithful representation $\rho_0:G \rightarrow \mathbf{PO}(2,1) <\mathbf{PU}(2,1)$,  ${\bf H}^2_{\mathbb{C}}/\rho_0(K)$ is  homeomorphic to the tangent space of an infinite area  open disk  with three cone angles $\pi$, $\frac{2\pi}{3}$ and  $\pi$.  So ${\bf H}^2_{\mathbb{C}}/\rho_0(K)$ is homeomorphic to $N$ above.

 %the quadrilaterial $Q$ above. 
 
 The group $\rho_{(\sqrt{2},\operatorname{arccos}(-\frac{7}{8}))(K)}$  can be view as a discrete and faith representation $\rho_1:K \rightarrow \mathbf{PU}(2,1)$ with one accidental parabolic, by this we mean $\rho_1(\iota_1\iota_2) =I_1I_2$ is parabolic. If we can show there is a path $\rho_{t}:K \rightarrow \mathbf{PU}(2,1)$ connecting $\rho_0$ and $\rho_1$, where each $\rho_{t}$ for $t \in [0,1)$ is discrete, faithful and type-preserving, then ${\bf H}^2_{\mathbb{C}}/\rho_1(K)$ should be homeomorphic to ${\bf H}^2_{\mathbb{C}}/\rho_0(K)$.  Even through it is trivial to get a path of representations  $\rho_{t}:K \rightarrow \mathbf{PU}(2,1)$ connecting $\rho_0$ and $\rho_1$, but  showing $\rho_{t}$ is discrete and  faithful for each $t \in (0,1)$ is highly non-trivial.

 Note in $\pi_1(M)$ of the 3-manifold $M$ in Theorem  \ref{thm:3-mfd}, if we add the relation $a^2=b^3=id$, then we get the fundamental group of $L(3,-1)\#  \mathbb{R}{\bf P}^{3}\#  \mathbb{R}{\bf P}^{3}$, where $\#$ is the connect sum.  Which is exactly the fundamental group of ${\bf H}^2_{\mathbb{C}}/\Gamma(\sqrt{2},0)$, and $L(3,-1)\#  \mathbb{R}{\bf P}^{3} \#  \mathbb{R}{\bf P}^{3}$  is exactly the manifold at infinity of $L$.  So $M$ can also be viewed as  drilling out a simple closed curve in $L(3,-1)\# \mathbb{R}{\bf P}^{3}\#  \mathbb{R}{\bf P}^{3}$. For related topics, see \cite{Ma:2020d}.
 
% {\bf P}^3_{\mathbb{R}}
 	\subsection{Several remarks on the moduli space  $\mathcal{M}$}\label{remark:denegarating}

\begin{question} \label{question:h1over2}The case   $h = \cos(\frac{\pi}{m})$    for $m \geq 3$ is even more difficult and  extremely interesting to study. In this case, the planes $\mathcal{P}_{2}$ and $\mathcal{P}_{3}$ has angle $\frac{\pi}{m}$, and the planes $\mathcal{P}_{3}$ and $\mathcal{P}_{4}$ also has angle $\frac{\pi}{m}$.
	The  matrices presentations in Subsection  \ref{subsec:matricesin3dim} also hold   when $h = \cos(\frac{\pi}{m})$  for $m \geq 3$. For example, if $h=\frac{1}{2}$, then $(I_2I_3)^3=(I_3I_4)^3=id$. Moreover, if $h=\frac{1}{2}$ and $t=\frac{2\pi}{3}$, then $I_1I_4I_1I_2I_1I_4I_3$
 	has order 6. It is the pink diamond marked point in Figure \ref{figure:moduli}. It is interesting to know whether this group is discrete.
\end{question}

\begin{question} \label{question:frontier}
Moreover, we draw several curves in Figure \ref{figure:moduli}.  The blue  curve is the locus where $\mathcal{H}(I_1I_2I_3I_4)=0$;  The cyan  curve is the locus where $\mathcal{H}(I_1I_4I_1I_2I_1I_4I_3)=0$; The black  curve is the locus where $\mathcal{H}(I_1I_3I_4I_1I_2)=0$; The green  curve is the locus where $\mathcal{H}(I_1I_3I_2I_4I_1I_3I_4)=0$. It is interesting to know whether these curves are the frontiers of the discreteness and faithfulness of the representations of  the group $G$ 
(when $h=\cos(\frac{\pi}{p})$, there are obvious relations $(\iota_2 \iota_3)^{p}=(\iota_3 \iota_4)^{p}=id$ we have to add in the presentation of the abstract group $G$). 
 
 \end{question}

\bibliographystyle{amsplain}

\end{document}